\documentclass[12pt,reqno]{amsart}
\usepackage[margin=1in]{geometry}
\usepackage{amsmath,amssymb,amsthm,graphicx,amsxtra, setspace}
\usepackage[utf8]{inputenc}
\usepackage{mathrsfs}
\usepackage{hyperref}
\usepackage{upgreek}
\usepackage{mathtools}
\usepackage[mathcal]{euscript}
\usepackage{xcolor}
\allowdisplaybreaks

\usepackage[pagewise]{lineno}

\usepackage{graphicx,eurosym}
\usepackage{hyperref}
\usepackage{mathtools}

\usepackage[cyr]{aeguill}

\colorlet{darkblue}{blue!50!black}

\hypersetup{
	colorlinks,%
	citecolor=blue,%
	filecolor=red,%
	linkcolor=darkblue,%
	urlcolor=blue,%
	pdfnewwindow=true,%
	pdfstartview={FitH}
}

\usepackage{graphicx,amscd,mathrsfs,wrapfig,mathrsfs,lipsum}
\usepackage{tikz}
\usepackage{multicol}
\usepackage{caption}
\usetikzlibrary{arrows}

\colorlet{darkblue}{blue!50!black}

\newtheorem{theorem}{Theorem}[section]
\newtheorem{lemma}[theorem]{Lemma}
\newtheorem{proposition}[theorem]{Proposition}

\newtheorem{corollary}[theorem]{Corollary}
\newtheorem{definition}[theorem]{Definition}

\newtheorem{remark}[theorem]{Remark}
\newtheorem{hypothesis}[theorem]{Hypothesis}

\newtheorem{condition}[theorem]{Condition}

\let\emptyset\varnothing

\let\originalleft\left
\let\originalright\right
\renewcommand{\left}{\mathopen{}\mathclose\bgroup\originalleft}
\renewcommand{\right}{\aftergroup\egroup\originalright}

\theoremstyle{definition}

\def\wi{\widehat}
\def\vi{\widetilde}

\def\d{\mathrm{d}}

\def\I{\mathrm{I}}
\def\B{\mathrm{B}}

\def\A{\mathrm{A}}
\def\m{\mathrm{m}}
\def\W{\mathrm{W}}
\def\R{\mathbb{R}}
\def\E{\mathbb{E}}
\def\Q{\mathrm{Q}}
\def\H{\mathrm{H}}

\def\e{\epsilon}
\def\2{\mathcal{E}}
\def\L{\mathrm{L}}

\def\C{\mathrm{C}}

\def\F{\mathrm{F}}
\def\P{\mathbb{P}}
\def\N{\mathbb{N}}

\def\PP{\mathscr{P}}

\def\vp{\varpi}

\def\M{\mathrm{M}}

\def\UU{\mathfrak{U}}
\def\SS{\mathbb{S}}

\def\p{\mathfrak{p}}
\def\c{\mathfrak{c}}
\def\EE{\mathcal{E}}

\newcommand{\Addresses}{{
		\footnote{
			\noindent \textsuperscript{1,3}Department of Mathematics, Indian Institute of Technology Roorkee-IIT Roorkee,
			Haridwar Highway, Roorkee, Uttarakhand 247667, INDIA.
			\\
			\textsuperscript{2}Theoretical Statistics and Mathematics Unit, Indian Statistical Institute Banglore Centre-ISI Banglore, 8th Mile Mysore Road, Bangalore, Karnataka 560059, India. \par\nopagebreak
			\noindent  \textit{e-mail:} \texttt{Manil T. Mohan: maniltmohan@ma.iitr.ac.in, maniltmohan@gmail.com.}
			
			\textit{e-mail:} \texttt{Ankit Kumar: akumar14@mt.iitr.ac.in.}
			
			\textit{e-mail:} \texttt{Vivek Kumar: vivekkumar\_ra@isibang.ac.in, vivekmsc118@gmail.com.}
			
			\noindent \textsuperscript{*}Corresponding author.
			
			\textit{Keywords:} Stochastic Burgers-Huxley equations, mild solution, large deviations.
			
			Mathematics Subject Classification (2020): Primary 60H15, 60F10; Secondary 37L55, 60H20.

}}}
\begin{document}	
	
	\title[Well-posedness and ULDP for SGBH equations]{Well-posedness and uniform large deviation principle for stochastic Burgers-Huxley equation perturbed by a multiplicative  noise
		\Addresses}

	\author[A. Kumar, V. Kumar and M. T. Mohan]
	{Ankit Kumar\textsuperscript{1}, Vivek Kumar\textsuperscript{2} and Manil T. Mohan\textsuperscript{3*}}

	\maketitle

	\begin{abstract}
		In this work, we focus on the global solvability and  uniform large deviations for the solutions of stochastic generalized Burgers-Huxley (SGBH) equation perturbed by a small multiplicative white in time and colored in space noise. The SGBH equation has the nonlinearity of polynomial order and noise considered in this work  is infinite dimensional with a coefficient having linear growth.  First, we prove the existence of a \textsl{unique local mild solution}  in the sense of Walsh  to SGBH equation with the help of a truncation argument and contraction mapping principle. Then the global solvability results are established  by using uniform bounds of the local mild solution, stopping time arguments, tightness properties and Skorokhod's representation theorem.  
		By using the uniform Laplace principle, we obtain the   \textsl{large deviation principle} (LDP) for the law of solutions to SGBH equation by using  variational representation methods. Further, we derive the \textsl{uniform large deviation principle} (ULDP) for the law of solutions in two different topologies by using a weak convergence method.  First, in the $\mathrm{C}([0, T ];\mathrm{L}^p([0,1])) $ topology where the uniformity is over $\mathrm{L}^p([0,1])$-bounded sets of initial conditions, and secondly in the $\mathrm{C} ([0, T ] \times[0,1])$ topology with uniformity being over bounded subsets in the $\mathrm{C}([0,1])$-norm.
		Finally, we consider SGBH equation perturbed by a space-time white noise with bounded noise coefficient   and  establish the ULDP for the laws of  solutions. The results obtained in this work hold true for stochastic Burgers' as well as Burgers-Huxley equations. 
	\end{abstract}
	
	\section{Introduction}\label{Intro}
\textsl{Burgers' equation} (\cite{BJM}) is one of the celebrated fundamental partial differential equations and is   the simplest model for analyzing combined effect of nonlinear convection and diffusion. Burgers' equation appears   in various areas of applied mathematics  such as fluid mechanics, traffic flow, nonlinear acoustics, gas dynamics,   etc.  	The \textsl{Burgers-Huxley equation} is a mathematical model that describes the dynamics of a non-linear, one-dimensional wave system that is subject to both deterministic and stochastic influences. It is a combination of the Burgers equation \cite{BJM}, which describes the nonlinear behavior of a fluid flow, and the Huxley equation \cite{XW}, which describes the dynamics of an action potential in a nerve or muscle cell. The Burgers-Huxley equation is a good example of convection-diffusion-reaction equations.   In the stochastic version of the equation, a random noise is added to the system to account for the influence of external factors such as thermal fluctuations or other sources of randomness.

	\subsection{The model}
	In this work, we consider the following generalized version of Burgers-Huxley equation driven by a random forcing indexed by $\e\in(0,1]$:
	\begin{align}\label{1.1}
		\frac{\partial u_\e(t,x)}{\partial t}&=\nu\frac{\partial^2 u_\e(t,x)}{\partial x^2}- \alpha u_\e^\delta(t,x) \frac{\partial u_\e(t,x)}{\partial x}
		+\beta u_\e(t,x)(1-u_\e^\delta(t,x))(u_\e^\delta(t,x)-\gamma)\nonumber\\
		&\quad+ \sqrt{\epsilon}g(t, x, u_\e(t,x))\mathcal{W}(t,x),~~~~~t>0,~ x\in (0,1),
	\end{align}where $\nu>0$ and $\alpha>0$ are the viscosity and convection coefficients, respectively, $\beta>0,\ \delta \geq1$, $\gamma\geq 1$ are parameters and $\e\in(0,1]$ is the intensity of the noise, with the Dirichlet boundary condition and initial data 
	\begin{align}\label{1.2}
		u_\e(t,0)=u_\e(t,1)=0,\ \ t\geq0, \ \ \text{ and } \ \ u_\e(0,x)=u^0(x),
	\end{align}respectively. The noise coefficient $g(\cdot,\cdot,\cdot)$ has linear growth and it satisfies a Lipschitz condition  in the third variable (see Hypothesis \ref{hyp1} below).

In the deterministic case, for $\beta=0$ and $\delta=1$, the equation \eqref{1.1} represents the \textsl{classical viscous Burgers equation} and for $\delta=1$, it is called the \textsl{Burgers-Huxley equation}. The equation is commonly used in the study of a wide range of physical and biological systems, including fluid dynamics, neural networks, and pattern formation. For $\beta=0$ and $\alpha=1$, the equation gives the following generalized Burgers equations
\begin{align}\label{Burgers}
	\frac{\partial u_\e(t,x)}{\partial t}=\nu\frac{\partial^2 u_\e(t,x)}{\partial x^2}-\alpha u_\e^\delta(t,x)\frac{\partial u_\e(t,x)}{\partial x}+\sqrt{\e}g(t,x,u_\e(t,x))\mathcal{W}(t,x).
\end{align}
which is very famous as a toy  model to study the turbulence in fluids and the model has been explored a lot in the literature, for instance see \cite{DDR, GDDG, IG, IGDN, VMG, LZN}, etc. and references therein.  For $\alpha=0$, the equation \eqref{1.1} gives rise to the following equation:
\begin{align}\label{Huxley}\nonumber
	\frac{\partial u_\e(t,x)}{\partial t}&=\nu\frac{\partial^2 u_\e(t,x)}{\partial x^2}+\beta u_\e(t,x)+\beta u_\e(t,x)(t-u_\e^\delta(t,x))(u_\e^\delta(t,x)-\gamma)\\&\quad +\sqrt{\e}g(t,x,u_\e(t,x))\mathcal{W}(t,x), \ \ \text{ for } \ \ (t,x)\in(0,T)\times(0,1).
\end{align}
The equation \eqref{Huxley} is known as the \textsl{stochastic Huxley equation} and it models the propagation of nerve pulse in nerve fibers and wall motion in liquid crystals (see \cite{XW}).

\subsection{A short note on the noise $\mathcal{W}(\cdot,\cdot)$}\label{SN}
Let us now characterize the noise $\mathcal{W}(\cdot,\cdot)$ 
 appearing in  \eqref{1.1} on the probability space $(\Omega,\mathscr{F},\{\mathscr{F}_t\}_{t\geq0},\P)$ (see \cite{SVLBLR}). Consider a sequence  $\{\xi_j\}_{j\geq 1 }$ of independent standard Gaussian random variables on this basis (see  Lemma II.2.3, \cite{NVK}). If $\{\m_j(t)\}_{j\in\N}$ is an  orthonormal basis in $\mathrm{L}^2([0,T]), $ then
	
	\[\beta(t):=\sum_{j\in\N} \left(\int_{0}^{t}\m_j(s)ds\right) \xi_j\]
	is a standard Brownian motion with covariance $\E[\beta(t)\beta(s)]=(t\wedge s)$. 
	Consider a sequence $\{\beta_j(t)\}_{j\in\N}$ of independent standard Brownian motions and  an orthonormal basis $\{\varphi_j(x)\}_{j\in \N}$ in the space $\L^2([0,1])$. Let us define
	\[h_j(x):=\int_{0}^{x} \varphi_j(y)dy.\] 
	Then the process
	$$\W(t,x):=\sum_{j\in\N}h_j(x)\beta_j(t)$$
	is Gaussian  with the mean zero and covariance $\E[\W(t,x)\W(s,y)]=(t\wedge s)(x\wedge y).$ The process $\W$ is known as the \textsl{Brownian sheet.} The formal term-by-term differentiation of the series in $\beta(t)$ suggests a representation
	\begin{align*}
		\dot{\beta}(t)=\sum_{k\in\N}\m_k(t)\xi_k,
	\end{align*}
	so that 
	\begin{align}
		\frac{\partial^2\W}{\partial t\partial x}(t,x)=\dot{\W}(t,x)=\sum_{j\in\N}\varphi_j(x)\dot{\beta}_j(t), 
	\end{align}
	and we call the process $\dot{\W},$ the \textsl{(Gaussian) space-time white noise.} Even though the series  diverges, it defines a random generalized function on $\L^2([0,T]\times[0,1])$:
	\begin{align*}
		\dot{\W}(f)=\sum_{j\in\N}\int_0^T\left(\int_0^1f(t,x)\varphi_j(x)\d x\right)\d \beta_j(t)=\int_0^T\int_0^1f(t,x)\W(\d t,\d x). 
	\end{align*}
	
	To construct a noise, that is, white in time and colored in space, take a sequence of non-negative numbers $\{q_j\}_{j\in\N}$ and define 
	\begin{align}\label{15}
			\frac{\partial^2\W^{\Q}}{\partial t\partial x}(t,x)=\dot{\W}^{\Q}(t,x)=\mathcal{W}(t,x)=\sum_{j\in\N}q_j\varphi_j(x)\dot{\beta}_j(t),
	\end{align}	
and we call the process $\dot{\W}^\Q,$ the \textsl{(Gaussian) white in time and colored in space noise.} Whenever we consider $\dot{\W}^\Q$ in the work, we assume that \begin{align}\label{181}\sum\limits_{j\in\N}q_j^2<\infty.\end{align} Let  the family  $\{\varphi_j(x)\}_{j\in\N}$ be the orthonormal basis of $\L^2([0,1])$ consisting of the eigenfunctions of the operator $-\frac{\partial^2}{\partial x^2}$ corresponding to the eigenvalues $\{\lambda_j\}_{j\in\N}$. In fact, for $j=1,2,\ldots$
\begin{equation*}
\lambda_j=j^2 \pi^2, \ 	\varphi_j(x)=\sqrt{2}\sin(j\pi x).\
\end{equation*} Then an example of a family $\{q_j\}_{j\in\N}$ satisfying \eqref{181}  is $\{\lambda_j^{-\eta}\}_{j\in\N}$ for $\eta>\frac{1}{4}$, since $\sum\limits_{j\in\N}\frac{1}{j^{4\eta}}<\infty$ provided $\eta>\frac{1}{4}$.

	\subsection{Assumptions and solution concept} We first discuss the function spaces used in this work and then provide the assumptions satisfied by $g(\cdot,\cdot,\cdot)$.

	\subsubsection{Function spaces} Let us denote the space of all infinite times differentiable functions having compact support in $[0,1]$ by $\C_0^\infty([0,1])$. The Lebesgue spaces are denoted by $\L^p([0,1])$, for $p\geq1$, and the norm in the space $\L^p([0,1])$ is denoted by $\|\cdot\|_{\L^p}$. For $p=2$, the inner product in the space $\L^2([0,1])$ is represented by $(\cdot,\cdot)$. Let us denote the Hilbertian Sobolev spaces by $\H^k([0,1])$, for $k\in\N$ and $\H_0^1([0,1])$ represent the closure of $\C_0^\infty([0,1])$ in $\H^1$-norm.  Since we are working in a bounded domain, by the Poincar\'e inequality ($\sqrt{\lambda_1}\|u\|_{\L^2}^2\leq \|\partial_xu\|_{\L^2}^2$), we have the norm $(\|\cdot\|_{\L^2}^2+\|\partial_x\cdot\|_{\L^2}^2)^\frac{1}{2}$ is equivalent to the seminorm $\|\partial_x\cdot\|_{\L^2}$ and hence $\|\partial_x\cdot\|_{\L^2}$ defines a norm in $\H_0^1([0,1])$. Moreover, we obtain the Gelfand triplet $\H_0^1([0,1])\hookrightarrow \L^2([0,1])\hookrightarrow\H^{-1}([0,1])$, where $\H^{-1}([0,1])$ stands for the dual of  $\H_0^1([0,1])$. For the bounded domain $[0,1]$, one has the compact embedding $\H_0^1([0,1])\hookrightarrow\L^2([0,1])$. The duality paring between $\H_0^1([0,1])$ and its dual $\H^{-1}([0,1])$ as well as between $\L^p([0,1])$ and its dual $\L^{\frac{p}{p-1}}([0,1])$ is denoted by $\langle\cdot,\cdot\rangle$. In one dimension, the embedding $\H^s([0,1])\hookrightarrow\L^p([0,1])$ is compact for any $s>\frac{1}{2}-\frac{1}{p}$, for $p>2$. 
	\subsubsection{Hypotheses  and mild solution} Let us introduce the linear growth and Lipschitz conditions on the noise coefficient.
	\begin{hypothesis}\label{hyp1}
		The function $g:[0,T]\times[0,1]\times\R\to\R$ is a measurable function, satisfying the following conditions:
		\begin{align}\label{2.1}
			|g(t,x,r)|\leq K(1+|r|),\ \ \text{ and } \ \ |g(t,x,r)-g(t,x,s)|\leq L|r-s|,			
		\end{align}for all $t\in[0,T],\ x\in[0,1]$ and $r,s\in\mathbb{R}$, where $K,L$ are some positive constants.
	\end{hypothesis}

	The \textsl{fundamental solution} $G(t,x,y)$ of the heat equation in the interval $[0,1]$ with the Dirichlet boundary conditions is given by 
	\begin{align}\label{19}
		G(t,x,y)=\frac{1}{\sqrt{4\pi t}}\sum_{m=-\infty}^{\infty}\bigg[e^{-\frac{(y-x-2m)^2}{4t}}-e^{-\frac{(y+x-2m)^2}{4t}}\bigg],
	\end{align}
for all $t\in(0,T]$ and $x,y\in[0,1]$. 
	
	Next, we provide the definition of \textsl{mild solution} in the sense of Walsh (\cite{JBW}) to the problem \eqref{1.1}-\eqref{1.2}. 
	\begin{definition}[Mild solution]\label{def1}
		An $\L^p([0,1])$-valued and $\mathscr{F}_t$-adapted stochastic process $u_\e:[0,\infty)\times[0,1]\times\Omega\to\R$ with $\P$-a.s. continuous trajectories on $t\in[0,T]$, is called a \textsf{mild solution} to the system \eqref{1.1}-\eqref{1.2}, if for any $T>0$, $u_\e(t):=u_\e(t,\cdot,\cdot)$ satisfies the following integral equation:
		\begin{align}\label{2.2}\nonumber
			u_\e(t)&=\int_0^1G(t,x,y)u^0(y)\d y+\frac{\alpha}{\delta+1}\int_0^t\int_0^1\frac{\partial G}{\partial y}(t-s,x,y)\p(u_\e(s,y))\d y\d s\\&\nonumber\quad +\beta\int_0^t\int_0^1G(t-s,x,y)\c(u_\e(s,y))\d y\d s\\&\quad+ \sqrt{\e}
			\int_0^t\int_0^1G(t-s,x,y)g(s,y,u_\e(s,y))\W^\Q(\d s,\d y), \ \ \P\text{-a.s.},
		\end{align}for all $t\in[0,T]$, where $\p(u)=u^{\delta+1}, \ \c(u)=u(1-u^\delta)(u^\delta-\gamma)$ and $G(\cdot,\cdot,\cdot)$ is the fundamental solution of the heat equation in $[0,1]$ with the Dirichlet boundary conditions.
	\end{definition}
	If $u_{\e}(\cdot)$ is a mild solution  to the system \eqref{1.1}-\eqref{1.2} in the sense of Definition \ref{def1}, then by using Proposition 3.7,  \cite{IGCR}, one can show that $u_{\e}(\cdot)$ has a continuous modification:
	\begin{align}\label{WF}\nonumber
		\int_0^1u_\e(t,y)\phi(y)\d y& =\int_0^1u^0(y)\phi(y)\d y +\nu\int_0^t\int_0^1u_\e(s,y)\phi''(y)\d y\d s\\&\nonumber\quad+\frac{\alpha}{\delta+1}\int_0^t\int_0^1 \p(u_\e(s,y))\phi'(y)\d y\d s+\beta \int_0^t\int_0^1 \c(u_\e(s,y))\phi(y)\d y\d s\\&\quad +\sqrt{\e}\int_0^t\int_0^1 g(s,y,u_\e(s,y))\phi(y)\W^\Q(\d s,\d y), 
	\end{align}
	for every test function $\phi\in\C^2([0,1]), \ \phi(x)=0, \ x\in \{0,1\}$ and all $t\in[0,T]$. The existence of a unique mild solution to the system \eqref{1.1}-\eqref{1.2} will be established in Section \ref{Sec3}.

	\subsection{Literature Survey}
	The SGBH equation \eqref{1.1} covers both the models given by equations \eqref{Burgers} and \eqref{Huxley}. The global solvability results of deterministic generalized Burgers-Huxley (GBH) equation and numerical studies have been carried out in \cite{VJE, AKMT,MTMAK}, etc. and references therein.  For the global solvability results of the stochastic counterpart, one can see  \cite{AKMTM3,AKMTM2,MTMSBH,MTMSGBH}, etc. In the works \cite{AKMTM3,AKMTM2,MTMSBH}, the noise  coefficient $g(\cdot,\cdot,\cdot)$ is taken as a bounded function, while in \cite{MTMSGBH} Burgers-Huxley equation is considered with the $\mathrm{Q}$-Wiener process
	having noise coefficient as a Hilbert-Schmidt operator which satisfy a linear growth and Lipschitz  conditions.

	The \textsl{large deviation theory}  is one of the most popular research topics in probability theory and it explores the limiting behavior of the probabilities of rare events in the terms of a rate function. Though the topic of large deviations  is much older,  the modern theories  were mainly developed by  Donsker and Vardhan \cite{DV85, V84}, and Freidlin and Wentzell \cite{MIFADW}.  In this work, we are following the \textsl{uniform large deviation principle} (ULDP) introduced by Freidlin and Wentzell which is known as the \textsl{Freidlin-Wentzell uniform large deviation principle} (FWULDP) in the honor of their names. It is one of the most popular definitions of ULDP. Specifically, the work \cite{MIFADW} describes the asymptotic behavior of probability of large deviations of the law of the solution to a family of small noise finite-dimensional SDEs, away from its law of large number dynamics. In order to prove ULDP, one can use different types of \textsl{uniform Lapalce principles} (ULP)  with additional assumptions available in the literature. 
		In the work \cite{ADPDVM}, the authors proved FWULDP for diverse families of infinite-dimensional stochastic dynamical systems by using certain variational representations.  They developed the weak convergence method for proving the ULDP by using the ULP. In order to obtain the ULDP for the law of the solution to certain SPDEs, they have provided a  sufficient condition, and to verify the condition, one needs some basic qualitative properties of the solution like existence, uniqueness, and tightness of the controlled analogues of the original SPDEs. Adopting the approach established in \cite{ADPDVM},  the authors in \cite{ABPFD, QCHG, FCAM,MS,MSABPD19, MSLS, LS}, etc. obtained several results in this direction for different  types of SPDEs. The author in the work \cite{MS} established  the equivalences between four different kinds of definitions  of  ULDP and ULP.
	
	One of the main concerns in the method provided by \cite{ADPDVM} is that it requires the convergence of initial conditions and therefore one cannot use this method directly to prove ULDP  for SPDEs having non-compact sets of the initial conditions. For several applications, for example, in the case of the study of exit times of a stochastic process $``u"$ from its domain, the large deviations of $``u"$ must be uniform over the initial conditions lie in the bounded subsets of the space that are not necessarily compact \cite{MSABPD19}. Recently, in the work \cite{MS}, the  author  introduced a new definition called the \textsl{equicontinuous uniform Laplace principle} (EULP) of ULP and showed that the definition of EULP is equivalent to the definition of  FWULDP without any compactness assumptions. In order to prove  EULP, he came with  a sufficient condition under which the law of the solution follows the  ULDP with uniformity being over initial conditions that belong to bounded but not necessarily compact sets. In the work \cite{MSLS}, the authors  adopted the methodology from \cite{MS} and established the ULDP for the law of the solution of the stochastic Burgers type equations. Here, they  considered the nonlinearity of polynomial growth of any order and the driving noise as a finite-dimensional Wiener process. Later, the author in \cite{LS23} proved the ULDP for the law of the solutions to a class of  semilinear SPDEs driven by  Brownian sheet with bounded noise coefficient. In both works \cite{MSLS} and  \cite{LS23}, the uniformity is considered with respect to the initial conditions that are  bounded and do not necessarily belongs to a compact set.

	\subsection{Objectives and novelties}
	In this work, we are mainly considering the SGBH system \eqref{1.1}-\eqref{1.2} consisting of a multiplicative noise which is white in time and colored in space with the coefficient $g(\cdot,\cdot,\cdot)$ having a linear growth with respect to the unknown $u_{\epsilon}$ (Hypothesis \ref{hyp1}). To the best of our knowledge, this type of noise for this model and the following results on this model have not been explored yet.
	\begin{itemize}
		\item First we  prove the existence of  a unique local mild solution to the  system \eqref{1.1}-\eqref{1.2} in the sense of Walsh (Definition \ref{def1})  by using a truncated system and fixed point arguments. Then by using the uniform energy estimates of the local solution, stopping time arguments, tightness  properties and Skorokhod's representation theorem, we prove the global existence of the solution.
		\item Next, we move to the LDP part, where we discuss two different types of LDP  depending on the availability of the initial conditions.
		\begin{enumerate}	 
			\item  The first LDP result is based on a variational representation method developed  in \cite{ADPDVM} and to use it, we need the initial data over a compact subset.
			\item If we are given the non-compact sets of initial data, then it is not enough to establish  ULP as mentioned in \cite{ADPDVM}, and hence for this situation, we borrow  ideas from the works \cite{MS, MSLS}. We verify the sufficient condition for EULP and hence ULDP in two different topologies: the $\mathrm{C}([0, T ];\mathrm{L}^p([0,1])) $ topology where the uniformity is over $\mathrm{L}^p([0,1])$-bounded sets of initial conditions (for $p\geq \max\{6,2\delta+1\}$), and secondly in the $\mathrm{C} ([0, T ] \times[0,1])$ topology with uniformity being over bounded subsets in the $\mathrm{C}([0,1])$-norm.  To prove the ULDP in $ \C([0,T]\times [0,1]) $, we use the embedding  $ \C([0,T]\times [0,1])\hookrightarrow\C([0,T]; \L^p([0,1]))$ and Vardhan's contraction principle \cite{V66}.
		\end{enumerate}
		\item Finally, we consider SGBH equation perturbed by a space-time white noise with bounded noise coefficient   and  establish the ULDP for the laws of solutions in the $\mathrm{C}([0, T ];\mathrm{L}^p([0,1])), $  for $p\geq 2\delta+1$ and $\mathrm{C} ([0, T ] \times[0,1])$ topologies. 
	\end{itemize}
The well-posedess and LDP results obtained in this work hold true for stochastic Burgers as well as Burgers-Huxley equations. The main difference between the works \cite{MSLS,LS23,LS} is that none of them considers the infinite dimensional multiplicative time white and space colored noise with coefficient having linear growth.

	The main difficulties in proving our results are the presence of a polynomial type of nonlinearity and the infinite-dimensional of noise with unbounded noise coefficient. The polynomial nonlinearites are handled by using Green's function estimates  (see Appendix \ref{SUR}).  
	We overcame the difficulty of  a proper  Burkholder-Davis-Gundy (BDG) inequality in $\L^p$-spaces by using the results from \cite{MCV,IY}. Recently, the authors in \cite{MCV,IY} obtained BDG inequality in UMD Banach spaces, which helped us to handle several maximal inequalities. Precisely, the authors in \cite{MCV} showed that if $\mathrm{M}$ is an $\L^p$-bounded martingale, $1<p<\infty$, with $\mathrm{M}_0=0$, that takes values in a UMD Banach space $(\mathrm{X},\|\cdot\|_{\mathrm{X}})$ over a measure space $(S,\Sigma,\mu)$, then for all $t\geq 0$, 
	\begin{align}\label{18}
		\E\left[\sup_{0\leq s\leq t}\|\mathrm{M}_s(\sigma)\|_{\mathrm{X}}^p\right]\leq C_{p,\mathrm{X}}\mathbb{E}\left[\|[\mathrm{M}(\sigma)]_t^{\frac{1}{2}}\|_{\mathrm{X}}^p\right],
	\end{align}
	where the quadratic variation $[\mathrm{M}(\sigma)]_t$ is considered pointwise in $\sigma\in S$. As $\L^p([0,1]),$ $1<p<\infty$ is a UMD Banach space, we can use the above BDG inequality in the sequel. We also point out that in the case of space-time white noise with bounded noise coefficient, we are able to prove the ULDP in the  $\mathrm{C}([0, T ];\mathrm{L}^p([0,1]))$ topology  for $p\geq 2\delta+1$, as we are not using Lemma \ref{lemE2} for the tightness of stochastic convolution.

	\subsection{Organization of the paper}
	The existence and uniqueness of local as well as global mild solution in the sense of Walsh of the system \eqref{1.1}-\eqref{1.2} is discussed in  Section \ref{Sec3} (Theorem \ref{thrmE1}, Lemmas \ref{lemUE} and \ref{lemE2} and Theorem \ref{thrmE2}). In Section \ref{Sec4}, we start with the basics of LDP and state a sufficient Condition \ref{Cond2} and recall a general result for LDP (Theorem \ref{thrmLDP}). Then, we move to the proof of the existence and uniqueness of solutions to the  stochastic controlled integral equation \eqref{CE1} and skeleton integral equation \eqref{CE2} (Theorems \ref{thrmLDP1} and \ref{thrmLDP3}, respectively). Later, we prove  our main theorem (Theorem \ref{thrmLDP4}) of this section, that is, LDP for the laws of solutions to the system \eqref{1.1}-\eqref{1.2}. In Section \ref{SecULDP}, we  first  recall some useful definitions and results from \cite{MS} of ULDP and the sufficient Condition \ref{cond} which ensures EULP and hence ULDP by Theorem \ref{thrm3.3}. Then, we state and prove several intermediate results (Proposition \ref{PrioriEstimate}, Corollary \ref{UCPC}) which help us to obtain our main results (Theorems \ref{LDPinLp} and \ref{LDPinC}) of the section. We obtain the proof of Theorem \ref{LDPinLp} by using uniform convergence in probability (Theorem \ref{UCP}) and Theorem \ref{ExistenceofControled}. Similarly, we prove Theorem \ref{UCP} by combining convergence in probability in supremum norm (Theorem \ref{ContinuousCase}) along with Theorem \ref{ExistenceofControled}.	In Section  \ref{ULDPBDD}, we establish the ULDP for the law of the solution of the  system \eqref{1.1}-\eqref{1.2} with space-time white noise with a bounded noise coefficient by adopting an approach given in \cite{MSLS}. We wind up the article by recalling some useful results from \cite{IG,IGCR}  in Appendix \ref{SUR}.


	\section{ Solvability Results}\label{Sec3}\setcounter{equation}{0}
	In this section, we discuss the solvability of the system \eqref{1.1}-\eqref{1.2}. In order achieve it,  we first establish a local solvability result (Theorem \ref{thrmE1}) using the truncation defined in \eqref{E1}  and a fixed point argument. Then, we establish the uniform energy estimate (Lemma \ref{lemUE}) and tightness of stocahstic convolution term (Lemma \ref{lemE2}) followed by global solvability results (Theorem \ref{thrmE2}).

	\subsection{Local solution}
	In order to obtain a mild solution of the system \eqref{1.1}-\eqref{1.2}, we use a truncation technique. 
	Let $\Pi_R(\cdot)$ be  a $\C^1(\R)$ function   such that 
	\begin{equation}\label{E1}
		\Pi_R (r)=\left\{
		\begin{aligned}
			1, \ \ &\text{ if } \ \ |r|\leq R,\\
			0, \ \ &\text{ if } \ \ |r|\geq R+1,
		\end{aligned}
		\right.
	\end{equation}
	and $\big|\frac{\d\Pi_R}{\d r }(r)\big|\leq 2$, for all $r\in\R$. Note that by the mean value theorem, one can show that $|\Pi_R(r_1)-\Pi_R(r_2)|\leq 2|r_1-r_2|$, for all $r_1,r_2\in\R$.  Let us  now introduce the following truncated system:
	\begin{align}\label{E2}\nonumber
		&\frac{\partial u_{\e,R}(t,x)}{\partial t}\nonumber\\&=\nu\frac{\partial^2 u_{\e,R}(t,x)}{\partial x^2}+\frac{\alpha}{\delta+1}\frac{\partial }{\partial x}\Pi_R(\|u_{\e,R}(t)\|_{\L^p})\p(u_{\e,R}(t,x))+\beta\Pi_R(\|u_{\e,R}(t)\|_{\L^p}) \c(u_{\e,R}(t,x)) \nonumber\\&\quad+\sqrt{\e}\Pi_R(\|u_{\e,R}(t)\|_{\L^p})g(t,x,u_{\e,R}(t,x))\frac{\partial^2\W^\Q(t,x) }{\partial x\partial t},
	\end{align}
	for $(t,x)\in(0,T)\times(0,1)$, with the same boundary and initial conditions given in \eqref{1.2}. Let us write down  the mild form of the above truncated equation
	\begin{align}\label{E3}\nonumber
	&	u_{\e,R}(t,x)\nonumber\\&=\int_0^1G(t,x,y)u^0(y)\d y+\frac{\alpha}{\delta+1}\int_0^t\int_0^1\frac{\partial G}{\partial y}(t-s,x,y)\Pi_R(\|u_{\e,R}(s)\|_{\L^p})\p(u_{\e,R}(s,y))\d y\d s\nonumber\\&\nonumber\quad +\beta\int_0^t\int_0^1G(t-s,x,y)\Pi_R(\|u_{\e,R}(s)\|_{\L^p})\c(u_{\e,R}(s,y))\d y\d s\\&\quad+ \sqrt{\e}
		\int_0^t\int_0^1G(t-s,x,y)\Pi_R(\|u_{\e,R}(s)\|_{\L^p})g(s,y,u_{\e,R}(s,y))\W^\Q(\d s,\d y), \ \ \P\text{-a.s.}
	\end{align}Set 
	\begin{align}\nonumber\label{E4}
		\mathscr{A}_1u_{\e,R}(t,x)&:=\int_0^1G(t,x,y)u^0(y)\d y,\\ \nonumber
		\mathscr{A}_2u_{\e,R}(t,x)&:=\int_0^t\int_0^1\frac{\partial G}{\partial y}(t-s,x,y)\Pi_R(\|u_{\e,R}(s)\|_{\L^p})\p(u_{\e,R}(s,y))\d y\d s,\\\nonumber
		\mathscr{A}_3u_{\e,R}(t,x)&:=\int_0^t\int_0^1G(t-s,x,y)\Pi_R(\|u_{\e,R}(s)\|_{\L^p})\c(u_{\e,R}(s,y))\d y\d s,\\ \nonumber
		\mathscr{A}_4u_{\e,R}(t,x)&:=\sqrt{\e}
		\int_0^t\int_0^1G(t-s,x,y)\Pi_R(\|u_{\e,R}(s)\|_{\L^p})g(s,y,u_{\e,R}(s,y))\W^\Q(\d s,\d y),\\
		\mathscr{A}u_{\e,R}(t,x)&:= \mathscr{A}_1u_{\e,R}(t,x)+\frac{\alpha}{\delta+1}\mathscr{A}_2u_{\e,R}(t,x)+\beta\mathscr{A}_3u_{\e,R}(t,x)+\mathscr{A}_4u_{\e,R}(t,x).
	\end{align}
	Let us define a Banach space $\mathcal{H}$ consisting of the adapted processes $u:\Omega\times[0,T]\to \L^p([0,1])$ such that 
	\begin{align*}
		\|u\|_\mathcal{H}^p:=\sup_{t\in[0,T]}e^{-\lambda p t}\E\big[\|u(t)\|_{\L^p}^p\big],
	\end{align*} where $\lambda$ will be fixed later.
	
	\begin{theorem}\label{thrmE1}
		Let us assume that $u^0\in\L^p([0,1])$, for $p\geq2\delta+1$. Then, there exists a unique $\L^p([0,1])$-valued $\mathscr{F}_t$-adapted continuous process $u_{\e,R}(\cdot,\cdot)$ satisfying the truncated integral equation \eqref{E3} such that 
		\begin{align}\label{E5}
			\E\bigg[\sup_{t\in[0,T]}\|u_{\e,R}(t)\|_{\L^p}^p\bigg]\leq C(R,T).
		\end{align}
	\end{theorem}
	\begin{proof}
		The proof of this theorem is based on fixed point arguments.	We divide the proof into two steps. In the first step, we show that the operator $\mathscr{A}$ is well defined in the space $\mathcal{H}$. In the second step, we prove that $\mathscr{A}$ is a contraction map and then the existence and uniqueness of the truncated integral equation \eqref{E3} follows from the  contraction mapping principle.

		\vspace{2mm}
		\noindent
		\textbf{Step 1.} \textsl{For $p\geq 2\delta+1$, $\E\bigg[\sup\limits_{t\in[0,T]}\|\mathscr{A}u_{\e,R}(t)\|_{\L^p}^p\bigg]\leq C(R,T)$.}
		
		\noindent Using \eqref{A1} and Young's inequality, we find 
		\begin{align}\label{E6}
			\E\big[\|\mathscr{A}_1u_{\e,R}(t)\|_{\L^p}^p\big] =\bigg\|\int_0^1G(t,\cdot,y)u^0(y)\d y\bigg\|_{\L^p}^p \leq C\|u^0\|_{\L^p}^p.
		\end{align}
		Using similar arguments as in Proposition 3.3 \cite{AKMTM3}, we obtain the following estimates:
		\begin{align}\label{E7}
			\E\bigg[	\|\mathscr{A}_2u_{\e,R}(t)\|_{\L^p}^p\bigg] &\leq C(R+1)^{p(\delta+1)}\int_0^t(t-s)^{-1+\frac{p-\delta}{2p}}\d s\leq C(R+1)^{p(\delta+1)}t^{\frac{p-\delta}{2p}} ,\\ \nonumber\label{E8}
			\E\bigg[	\|\mathscr{A}_3u_{\e,R}(t)\|_{\L^p}^p\bigg]&\leq C\int_0^t\bigg\{(R+1)^{p(\delta+1)}(1+\gamma)(t-s)^{-\frac{1}{2}+\frac{p-\delta}{2p}}+(R+1)^p\gamma \\&\nonumber\qquad+(R+1)^{p(2\delta+1)}(t-s)^{-\frac{1}{2}+\frac{p-2\delta}{2p}}\bigg\}\d s\\ &\leq C \left((R+1)^{p(\delta+1)}(1+\gamma)t^{\frac{1}{2}+\frac{p-\delta}{2p}}+(R+1)^p\gamma t +(R+1)^{p(2\delta+1)}t^{\frac{1}{2}+\frac{p-2\delta}{2p}}  \right),
		\end{align}for $p\geq2\delta+1$. Similarly, we consider the term $\E\big[	\|\mathscr{A}_4u_{\e,R}(t)\|_{\L^p}^p\big]$, and estimate it using Hypothesis \ref{hyp1}, Fubini's theorem, BDG (cf. \eqref{18}), H\"older's,   Minkowski's inequalities, \eqref{A1}, \eqref{A7} and Young's   inequality for convolution as 
		\begin{align}\label{E9}\nonumber
			&\E\bigg[	\|\mathscr{A}_4u_{\e,R}(t)\|_{\L^p}^p\bigg] \\&\nonumber= \e^{\frac{p}{2}} \E\bigg[\bigg\|	\int_0^t\int_0^1G(t-s,\cdot,y)\Pi_R(\|u_{\e,R}(s)\|_{\L^p})g(s,y,u_{\e,R}(s,y))\W^\Q(\d s,\d y) \bigg\|_{\L^p}^p\bigg]\\& \nonumber= \e^{\frac{p}{2}} \E\bigg[\bigg\|\sum_{j\in\N}	\int_0^t\big(G(t-s,\cdot,y)\Pi_R(\|u_{\e,R}(s)\|_{\L^p})g(s,\cdot,u_{\e,R}(s,y)),q_j\varphi_j(y)\big)\d \beta_j(s) \bigg\|_{\L^p}^p\bigg]	\\&\nonumber\leq C  \e^{\frac{p}{2}} \E\bigg[\bigg\|\bigg(\sum_{j\in\N}	\int_0^t\big|\big(G(t-s,\cdot,y)\Pi_R(\|u_{\e,R}(s)\|_{\L^p})g(s,y,u_{\e,R}(s,y)),q_j\varphi_j(y)\big)\big|^2\d s\bigg)^{\frac{1}{2}}\bigg\|_{\L^p}^p\bigg]	\nonumber\\&\leq C  \e^{\frac{p}{2}} 
			\E\bigg[\bigg\|\bigg(\sum_{j\in\N}\int_0^1q_j^2\varphi_j^2(y)\d y\nonumber\\&\qquad\qquad\nonumber\times\bigg(\int_0^t\int_0^1G^2(t-s,\cdot,y)\Pi_R^2(\|u_{\e,R}(s)\|_{\L^p})g^2(s,y,u_{\e,R}(s,y))\d y\d s\bigg)\bigg)^{\frac{1}{2}}\bigg\|_{\L^p}^p\bigg]\nonumber\\&\nonumber \leq C \e^{\frac{p}{2}} \E\bigg[\bigg\|\int_0^t\int_0^1G^2(t-s,\cdot,y)\Pi_R^2(\|u_{\e,R}(s)\|_{\L^p})g^2(s,y,u_{\e,R}(s,y))\d y\d s\bigg\|_{\L^{\frac{p}{2}}}^{\frac{p}{2}}\bigg]
			\\&\nonumber \leq C\e^{\frac{p}{2}}\E\bigg[\bigg(\int_0^t(t-s)^{-1}\big\|e^{-\frac{|\cdot|^2}{a_1(t-s)}}*\big|\Pi_R^2(\|u_{\e,R}(s)\|_{\L^p})g^2(s,\cdot,u_{\e,R}(s,\cdot))\big|\big\|_{\L^{\frac{p}{2}}}\d s\bigg)^{\frac{p}{2}}\bigg] \nonumber\\&\nonumber\leq  CK \e^{\frac{p}{2}} t^{\frac{p-2}{4}}\E\bigg[ \int_0^t(t-s)^{-\frac{1}{2}}\Pi_R^p(\|u_{\e,R}(s)\|_{\L^p})\big(1+\|u_{\e,R}(s)\|_{\L^p}^p\big)\d s\bigg]
			\\&\leq CK \e^{\frac{p}{2}}\big(1+R^p\big) t^\frac{p}{4},
		\end{align}where we have used the fact that $\{\varphi_j(x)\}_{j\in\N}$ is an orthonormal family of $\L^2([0,1])$ and $\sum\limits_{j\in\N}q_j^2<\infty$. Combining \eqref{E6}-\eqref{E9}, the claim follows. 
		
		\vspace{2mm}
		\noindent
		\textbf{Step 2.} \textsl{$\mathscr{A}$ is a contraction.} Let $u_{\e,R},v_{\e,R}\in\mathcal{H}$.   With no loss of generality, we can assume that $\|u_{\e,R}\|_{\L^p}\leq \|v_{\e,R}\|_{\L^p}$. Using Taylor's formula, H\"older's inequality, the embedding $\L^q([0,1])\hookrightarrow\L^p([0,1])$, for $q\geq p$,  mean value theorem, and the definition of mapping $\Pi_R(\cdot)$ given in \eqref{E1}, we have 
		\begin{align}\label{E10}\nonumber
			&	\|\Pi_R(\|u_{\e,R}\|_{\L^p})\c(u_{\e,R})-\Pi_R(\|v_{\e,R}\|_{\L^p})\c(v_{\e,R})\|_{\L^1} \\&\nonumber\leq \|(\Pi_R((\|u_{\e,R}\|_{\L^p}))-\Pi_R(\|v_{\e,R}\|_{\L^p}))\c(u_{\e,R})\|_{\L^1} + \|\Pi_R(\|v_{\e,R}\|_{\L^p})(\c(u_{\e,R})-\c(v_{\e,R}))\|_{\L^1} \\&\nonumber \leq \|(\Pi_R((\|u_{\e,R}\|_{\L^p}))-\Pi_R(\|v_{\e,R}\|_{\L^p}))((1+\gamma)u_{\e,R}^{\delta+1}-\gamma u_{\e,R}-u_{\e,R}^{2\delta+1})\|_{\L^1}\\&\nonumber \quad + \|\Pi_R(\|v_{\e,R}\|_{\L^p}((1+\gamma)(u_{\e,R}^{\delta+1}-v_{\e,R}^{\delta+1})-\gamma(u_{\e,R}-v_{\e,R})-(u_{\e,R}^{2\delta+1}-v_{\e,R}^{2\delta+1}))\|_{\L^1}\\&\nonumber\leq \big|\Pi_R((\|u_{\e,R}\|_{\L^p}))-\Pi_R(\|v_{\e,R}\|_{\L^p})\big|\left\{(1+\gamma)\|u_{\e,R}\|_{\L^{\delta+1}}^{\delta+1}+\gamma\|u_{\e,R}\|_{\L^1}+\|u_{\e,R}\|_{\L^{2\delta+1}}^{2\delta+1}\right\}\\&\nonumber\quad +\Pi_R(\|v_{\e,R}\|_{\L^p})\big\{(1+\gamma)(1+\delta)\|(u_{\e,R}-v_{\e,R})(\theta u_{\e,R}+(1-\theta) v_{\e,R}))^\delta\|_{\L^1}\\&\nonumber\qquad+\gamma \|u_{\e,R}-v_{\e,R}\|_{\L^1}+(1+2\delta)\|(u_{\e,R}-v_{\e,R})(\theta u_{\e,R}+(1-\theta) v_{\e,R}))^{2\delta}\|_{\L^1}\big\}\\&\nonumber \leq \big|\Pi_R((\|u_{\e,R}\|_{\L^p}))-\Pi_R(\|v_{\e,R}\|_{\L^p})\big|\left\{(1+\gamma)\|u_{\e,R}\|_{\L^p}^{\delta+1}+\gamma\|u_{\e,R}\|_{\L^p}+\|u_{\e,R}\|_{\L^p}^{2\delta+1}\right\}\\&\nonumber\quad +
			\Pi_R(\|v_{\e,R}\|_{\L^p})\big\{(1+\gamma)(1+\delta)\|\theta u_{\e,R}+(1-\theta) v_{\e,R})\|_{\L^\frac{p\delta}{p-1}}^\delta\\&\nonumber\qquad+\gamma \|u_{\e,R}-v_{\e,R}\|_{\L^p}+(1+2\delta)\|\theta u_{\e,R}+(1-\theta) v_{\e,R})\|_{\L^\frac{2p\delta}{p-1}}^{2\delta}\big\}\|u_{\e,R}-v_{\e,R}\|_{\L^p}
			\\&\nonumber \leq \big|\Pi_R((\|u_{\e,R}\|_{\L^p}))-\Pi_R(\|v_{\e,R}\|_{\L^p})\big|\left\{(1+\gamma)\|u_{\e,R}\|_{\L^p}^{\delta+1}+\gamma\|u_{\e,R}\|_{\L^p}+\|u_{\e,R}\|_{\L^p}^{2\delta+1}\right\}\\&\nonumber\quad +
			\Pi_R(\|v_{\e,R}\|_{\L^p})\big\{(1+\gamma)(1+\delta)(\| u_{\e,R}\|_{\L^p}+\|v_{\e,R}\|_{\L^p})^\delta\nonumber\\&\quad\nonumber+\gamma+(1+2\delta)(\| u_{\e,R}\|_{\L^p}+\|v_{\e,R}\|_{\L^p})^{2\delta}\big\}\|u_{\e,R}-v_{\e,R}\|_{\L^p}
			\\&\nonumber\leq 2\|u_{\e,R}-v_{\e,R}\|_{\L^p}\left\{(1+\gamma)(R+1)^{\delta+1}+\gamma R+1+(R+1)^{2\delta+1}\right\}\\&\nonumber\quad +
			\big\{2^\delta(1+\gamma)(1+\delta)(R+1)^\delta+\gamma +2^{2\delta}(1+2\delta)(R+1)^{2\delta}\big\}\|u_{\e,R}-v_{\e,R}\|_{\L^p} \\&\leq C(\delta,\gamma,R)\|u_{\e,R}-v_{\e,R}\|_{\L^p},
		\end{align}
		for $p\geq 2\delta+1$.	Using similar calculations as  in the above inequality yields
		\begin{align}\label{E11}
			\|\Pi_R(\|u_{\e,R}\|_{\L^p})\p(u_{\e,R})-\Pi_R(\|v_{\e,R}\|_{\L^p})\p(v_{\e,R})\|_{\L^1} 
			\leq C(\delta,R)\|u_{\e,R}-v_{\e,R}\|_{\L^p},
		\end{align} for $p\geq \delta+1$. Using \eqref{E11}, \eqref{A2}, \eqref{A7}, Minkowski's, H\"older's and Young's inequalities, we find 
	\small{	\begin{align}\label{E12}\nonumber
			&\|\mathscr{A}_2(u_{\e,R}(t))-\mathscr{A}_2(v_{\e,R}(t))\|_{\L^p}\\&\nonumber\leq \bigg\| \int_0^t\int_0^1\frac{\partial G}{\partial y}(t-s,\cdot,y)\big(\Pi_R(\|u_{\e,R}(s)\|_{\L^p})\p(u_{\e,R}(s,y))-\Pi_R(\|v_{\e,R}(s)\|_{\L^p})\p(v_{\e,R}(s,y))\big)\d y\d s\bigg\|_{\L^p}	\\&\nonumber\leq \int_0^t(t-s)^{-1}\big\| e^{\frac{-|\cdot|^2}{a_2(t-s)}}*\big|\Pi_R(\|u_{\e,R}(s)\|_{\L^p})\p(u_{\e,R}(s,\cdot))-\Pi_R(\|v_{\e,R}(s)\|_{\L^p})\p(v_{\e,R}(s,\cdot))\big|\big\|_{\L^p}\d s\\&\nonumber\leq \int_0^t(t-s)^{-1} \|e^{\frac{-|\cdot|^2}{a_2(t-s)}}\|_{\L^p}\big\|\Pi_R(\|u_{\e,R}(s)\|_{\L^p})\p(u_{\e,R}(s,\cdot))-\Pi_R(\|v_{\e,R}(s)\|_{\L^p})\p(v_{\e,R}(s,\cdot))\big\|_{\L^1}\d s
			\\&\nonumber\leq \int_0^t(t-s)^{-1+\frac{1}{2p}} \big\|\Pi_R(\|u_{\e,R}(s)\|_{\L^p})\p(u_{\e,R}(s,\cdot))-\Pi_R(\|v_{\e,R}(s)\|_{\L^p})\p(v_{\e,R}(s,\cdot))\big\|_{\L^1}\d s\\&\leq 
			C(\delta,R)	\int_0^t(t-s)^{-1+\frac{1}{2p}} \|u_{\e,R}(s)-v_{\e,R}(s)\|_{\L^p}\d s.
		\end{align}}
	From above inequality, we deduce for all $t\in[0,T]$
	
		\begin{align}\label{E13}\nonumber
			&	\|\mathscr{A}_2(u_{\e,R})-\mathscr{A}_2(v_{\e,R})\|_{\mathcal{H}}^p \nonumber\\&\leq 	C(\delta,R)\sup_{t\in[0,T]}e^{-\lambda p t}\E\bigg[	\bigg(\int_0^t(t-s)^{-1+\frac{1}{2p}} \|u_{\e,R}(s)-v_{\e,R}(s)\|_{\L^p}\d s\bigg)^p\bigg]\nonumber\\&\leq  C(\delta,R,p)T^{\frac{p-1}{2p}}\sup_{t\in[0,T]}\bigg(	\int_0^t(t-s)^{-1+\frac{1}{2p}}e^{-\lambda(t-s)}e^{-\lambda s}\E\big[\|u_{\e,R}(s)-v_{\e,R}(s)\|_{\L^p}^p\big]\d s\bigg)\nonumber\\&\leq C(\delta,R,p,T) \left(\frac{\Gamma(\frac{1}{2p})}{\lambda^{\frac{1}{2p}}}\right)\|u_{\e,R}-v_{\e,R}\|_\mathcal{H}^p,
		\end{align}
		where we have used H\"older's inequality and $\Gamma(\cdot)$ denotes the gamma function.
		
		Using \eqref{E11}, \eqref{A2}, \eqref{A7}, Minkowski's, H\"older's and Young's inequalities, we find 
		\begin{align}\label{E14}\nonumber
			&\|\mathscr{A}_3(u_{\e,R}(t))-\mathscr{A}_3(v_{\e,R}(t))\|_{\L^p}\\&\nonumber \leq \bigg\| \int_0^t\int_0^1G(t-s,\cdot,y)\big(\Pi_R(\|u_{\e,R}(s)\|_{\L^p})\c(u_{\e,R}(s,y))-\Pi_R(\|v_{\e,R}(s)\|_{\L^p})\c(v_{\e,R}(s,y))\big)\d y\d s\bigg\|_{\L^p}\\&\nonumber\leq 
			\int_0^t(t-s)^{-\frac{1}{2}}\big\| e^{\frac{-|\cdot|^2}{a_1(t-s)}}*\big|\Pi_R(\|u_{\e,R}(s)\|_{\L^p})\c(u_{\e,R}(s,\cdot))-\Pi_R(\|v_{\e,R}(s)\|_{\L^p})\c(v_{\e,R}(s,\cdot))\big|\big\|_{\L^p}\d s\\&\nonumber\leq \int_0^t(t-s)^{-1} \|e^{\frac{-|\cdot|^2}{a_2(t-s)}}\|_{\L^p}\big\|\Pi_R(\|u_{\e,R}(s)\|_{\L^p})\c(u_{\e,R}(s,\cdot))-\Pi_R(\|v_{\e,R}(s)\|_{\L^p})\c(v_{\e,R}(s,\cdot))\big\|_{\L^1}\d s\\& \leq 
			C(\delta,\gamma,R)\int_0^t(t-s)^{-\frac{1}{2}+\frac{1}{2p}}\|u_{\e,R}(s)-v_{\e,R}(s)\|_{\L^p}\d s,
		\end{align} for all $t\in[0,T]$. From the above inequality, we find
		\begin{align}\label{E15}\nonumber
			&	\|\mathscr{A}_3(u_{\e,R})-\mathscr{A}_3(v_{\e,R})\|_{\mathcal{H}}^p \nonumber\\&\leq C(\delta,R,p)T^{\frac{p^2-1}{2p}}\sup_{t\in[0,T]}\bigg(	\int_0^t(t-s)^{-\frac{1}{2}+\frac{1}{2p}}e^{-\lambda(t-s)}e^{-\lambda s}\E\big[\|u_{\e,R}(s)-v_{\e,R}(s)\|_{\L^p}^p\big]\d s\bigg)\nonumber\\&\leq C(\delta,\gamma,R,p,T) \left(\frac{\Gamma(\frac{1}{2}+
				\frac{1}{2p})}{\lambda^{\frac{1}{2}+\frac{1}{2p}}}\right)\|u_{\e,R}-v_{\e,R}\|_\mathcal{H}^p.
		\end{align}Using similar arguments as in \eqref{E9}, Hypothesis \ref{hyp1} (Lipschitz continuity of $g$) and Minkowski's inequality, we obtain 
		\begin{align}\label{E16}\nonumber
			&\E\big[\|\mathscr{A}_4(u_{\e,R}(t))-\mathscr{A}_4(v_{\e,R}(t))\|_{\L^p}^p\big]\\&\nonumber  \leq C\e^p\E\bigg[\bigg(\bigg\| \int_0^t\int_0^1G^2(t-s,\cdot,y)\big(\Pi_R(\|u_{\e,R}(s)\|_{\L^p})g(s,y,u_{\e,R}(s,y))\\&\nonumber\qquad-\Pi_R(\|v_{\e,R}(s)\|_{\L^p})g(s,y,v_{\e,R}(s,y))\big)^2\d y\d s\bigg)^\frac{1}{2}\bigg\|_{\L^p}^p\bigg]	\\&\nonumber\leq
			C(\e,p)\E\bigg[\bigg(\int_0^t(t-s)^{-1}\big\|e^{-\frac{|\cdot|^2}{a_1(t-s)}}*\big|\Pi_R(\|u_{\e,R}(s)\|_{\L^p})g(s,\cdot,u_{\e,R}(s,\cdot))\nonumber\\&\qquad\nonumber-\Pi_R(\|v_{\e,R}(s)\|_{\L^p})g(s,\cdot,v_{\e,R}(s,\cdot))\big|^2\big\|_{\L^{\frac{p}{2}}}\d s\bigg)^{\frac{p}{2}}\bigg]
			\\&\nonumber\leq
		C(\e,L,R,p)\E\bigg[\bigg(\int_0^t(t-s)^{-1+\frac{1}{p}}\|u_{\e,R}(s)-v_{\e,R}(s)\|_{\L^2}^2\d s\bigg)^{\frac{p}{2}}\bigg]
			\\&\nonumber\leq
			C(\e,L,R,p)\left(\bigg\{\E\bigg[\bigg(\int_0^t(t-s)^{-1+\frac{1}{p}}\|u_{\e,R}(s)-v_{\e,R}(s)\|_{\L^p}^2\d s\bigg)^{\frac{p}{2}}\bigg]\bigg\}^{\frac{2}{p}}\right)^\frac{p}{2}
			\\&\nonumber\leq
			C(\e,L,R,p)\left(\int_0^t(t-s)^{-1+\frac{1}{p}}\bigg\{\E\big[\|u_{\e,R}(s)-v_{\e,R}(s)\|_{\L^p}^p\big]\bigg\}^{\frac{2}{p}}\d s\right)^\frac{p}{2}
			\\&\nonumber\leq C(\e,L,R,p)\bigg\{\int_0^t(t-s)^{-1+\frac{1}{p}}e^{2\lambda s}\d s\bigg\}^{\frac{p}{2}}\sup_{s\in[0,t]}\bigg\{e^{-\lambda ps}\E\big[\|u_{\e,R}(s)-v_{\e,R}(s)\|_{\L^p}^p\big]\bigg\}
			\\&\nonumber\leq C(\e,L,R,p)\bigg\{e^{2\lambda t }\int_0^t(t-s)^{-1+\frac{1}{p}}e^{-2\lambda (t-s)}\d s\bigg\}^{\frac{p}{2}}\sup_{s\in[0,t]}\bigg\{e^{-\lambda ps}\E\big[\|u_{\e,R}(s)-v_{\e,R}(s)\|_{\L^p}^p\big]\bigg\}
			\\&\leq C(\e,L,R,p)\left(\frac{e^{2\lambda t} \Gamma\big(\frac{1}{p}\big)}{(2\lambda)^{\frac{1}{p}}}\right)^{\frac{p}{2}}\sup_{s\in[0,t]}\bigg\{e^{-\lambda ps}\E\big[\|u_{\e,R}(s)-v_{\e,R}(s)\|_{\L^p}^p\big]\bigg\},
		\end{align}for $p\geq2$. Thus, from \eqref{E16}, one can conclude that
		\begin{align}\label{E17}
			\|\mathscr{A}_4(u_{\e,R})-\mathscr{A}_4(v_{\e,R})\|_{\mathcal{H}}^p\leq C(\e,L,R,p)\left(\frac{ \Gamma\big(\frac{1}{p}\big)}{\lambda^{\frac{1}{p}}}\right)^{\frac{p}{2}}\|u_{\e,R}-v_{\e,R}\|_{\mathcal{H}}^p. 
		\end{align}
		Combining \eqref{E12}-\eqref{E17}, and fixing $\lambda>0$ sufficiently large such that
		\begin{align}\label{E18}
			C(\alpha,\beta,\gamma,\delta,\e, L,R,p,T)\left[ \left(\frac{\Gamma(\frac{1}{2p})}{\lambda^{\frac{1}{2p}}}\right)+ \left(\frac{\Gamma(\frac{1}{2}+
				\frac{1}{2p})}{\lambda^{\frac{1}{2}+\frac{1}{2p}}}\right)+\left(\frac{ \Gamma\big(\frac{1}{p}\big)}{\lambda^{\frac{1}{p}}}\right)^{\frac{p}{2}}\right]<1, 
		\end{align}for $p\geq2\delta+1$. For the above value of $\lambda$, the map $\mathscr{A}$ is a contraction on the Banach space $\mathcal{H}$. Therefore  there exists a unique fixed point for the map $\mathscr{A}$ by the contraction mapping principle and this gives the existence of a unique solution of the truncated integral equation \eqref{E3} and the estimate \eqref{E5} holds from Step 1.
	\end{proof}

	By the previous result we obtain the existence of a unique local mild solution of the system \eqref{1.1}-\eqref{1.2} up to a stopping time $$	\tau_{R}:=\inf\{t\geq0: \|u_\e(t)\|_{\L^p}\geq R\}\wedge T.$$  
	\subsection{Uniform estimates of the solution} 
	We now move to the proof of the existence of global  solution of the integral equation \eqref{2.2}. We need uniform energy estimates for our solution. Consider the following system:
	\begin{equation}\label{E19}
		\left\{
		\begin{aligned}
			\frac{\partial u_{\e,m}}{\partial t}(t,x)&=\nu\frac{\partial^2 u_{\e,m}}{\partial x^2}(t,x)+\frac{\alpha}{\delta+1}\frac{\partial}{\partial x}\p_m(u_{\e,m}(t,x))+\beta\c_m(u_{\e,m}(t,x))\\&\quad+\sqrt{\e}g_m(t,x,u_{\e,m}(t,x))\frac{\partial^2\W^\Q}{\partial t\partial x }, \ \ (t,x)\in(0,T)\times(0,1), \\ u_{\e,m}(t,0)&=u_{\e,m}(t,1)=0,\\
			u_{\e,m}(0,x)&=u^0_m,
		\end{aligned}
		\right.
	\end{equation}where we take the sequence of bounded Borel measurable functions $\p_m(t,x,r)=\p(t,x,r)$, $\c_m(t,x,r)=\c(t,x,r)$ and $g_m(t,x,r)=g(t,x,r)$, for $|r|\leq m$ and  $\p_m(t,x,r)=\c_m(t,x,r)=g_m(t,x,r)=0$, for $|r|\geq m+1$ such that they are globally Lipschitz in $r\in\mathbb{R}$. The functions $\p_m(t,x,r), \ \c_m(t,x,r)$ and $g_m(t,x,r)$ inherits the properties of $\p(t,x,r),\ \c(t,x,r)$, and $g(t,x,r)$, respectively whenever $|r|\leq m$. Let us consider a bounded and smooth sequence $\{u^0_m\}_{m\in\N}$ converging to $u^0$ in the space $\L^p([0,1])$. We have the following mild formulation to the system \eqref{E19}:
	\begin{align}\label{E20}\nonumber
		u_{\e,m}(t,x)&=\int_0^1G(t,x,y)u^0_m(y)\d y+\frac{\alpha}{\delta+1}\int_0^t\int_0^1\frac{\partial G}{\partial y}(t-s,x,y)\p_m(u_{\e,m}(s,y))\d y\d s\\&\nonumber\quad +\beta\int_0^t\int_0^1G(t-s,x,y)\c_m(u_{\e,m}(s,y))\d y\d s\\&\nonumber\quad+ \sqrt{\e}
		\int_0^t\int_0^1G(t-s,x,y)g_m(s,y,u_{\e,m}(s,y))\W^\Q(\d s,\d y)\\&=: 
		I_{m}^1+\frac{\alpha}{\delta+1}I_{\e,m}^2+\beta I_{\e,m}^3+\sqrt{\e}I_{\e,m}^4.
	\end{align}

	Let us now prove a uniform energy estimate which plays a crucial role to obtain the global solution to the integral equation \eqref{2.2}.
	\begin{lemma}\label{lemUE}
		For each $m\in\N, \ p\geq2\delta+1$, there exists a constant $C>0$ such that
		\begin{align}\label{UE1}\nonumber
			\mathbb{E}\bigg[\sup_{t\in[0,T]}\|u_{\e,m}(t)\|^{p}_{\L^p}+&\nu p(p-1)\int_0^T\||u_{\e,m}(s)|^{\frac{p-2}{2}}\partial_x u_{\e,m}(s)\|_{\L^2}^2\d s+p\beta\int_0^T\|u_{\e,m}(s)\|_{\L^{p+2\delta}}^{p+2\delta}\d s\bigg]\\&\leq C\big(1+\|u^{0}\|_{\L^p}^{p}\big).
		\end{align}
	\end{lemma}
	\begin{proof}Applying It\^{o} formula to the function  $|\cdot|^p$ for the process $u_{\e,m}(\cdot,x)$ for $x\in(0,1)$ and then taking integration over the spatial domain $[0,1]$ (see Lemma 4.2, \cite{IG}), we deduce for all $t\in[0,T]$
		\begin{align}
			\|u_{\e,m}(t)\|_{\L^p}^{p}&\leq \|u^0_m\|_{\L^p}^{p}-\nu p(p-1)\int_{0}^{t}\int_{0}^{1}|u_{\e,m}(s,x)|^{p-2}\left|\frac{\partial}{\partial x}u_{\e,m}(s,x)\right|^2 \d x \d s\nonumber\\
			&\quad+\frac{p(p-1)\alpha}{\delta+1}\int_{0}^{t}\int_{0}^{1}|u_{\e,m}(s,x)|^{p-2}\frac{\partial}{\partial x}u_{\e,m}(s,x) \p_m(u_{\e,m}(s,x))\d x \d s\nonumber\\
			&\quad+p\beta\int_{0}^{t}\int_{0}^{1}|u_{\e,m}(s,x)|^{p-2}u_{\e,m}(s,x)\c_m(u_{\e,m}(s,x))\d x\d s\nonumber\\
			&\quad+p\sqrt{\e}\sum_{j\in\N}\int_{0}^{t}\int_{0}^{1}|u_{\e,m}(s,x)|^{p-2}u_{\e,m}(s,x)g_m(s,x,u_{\e,m}(s,x))q_j\varphi_j(x)\d x \d\beta_j(s)\nonumber\\
			&\quad+\frac{\e}{2}p(p-1)\sum_{j\in\N}\int_{0}^{t}\int_{0}^{1}|u_{\e,m}(s,x)|^{p-2}g_m^2(s,x,u_{\e,m}(s,x))q_j^2\varphi_j^2(x)\d x\d s\nonumber\\
			&\leq\|u^0_m\|_{\L^p}^{p}+I_1(t)+I_2(t)+I_3(t)+I_4(t)+I_5(t),\ \mathbb{P}\text{-a.s.}
		\end{align}
		For $I_1(t)$, we have		  
		\begin{align}
			I_1(t)=-\nu p(p-1)\int_0^t\||u_{\e,m}(s)|^{\frac{p-2}{2}}\partial_x u_{\e,m}(s)\|_{\L^2}^2\d s.
		\end{align}Using the fact that $\big(\p(u),|u|^{p-2}u\big)=0$ (\cite{MTMSGBH}), we obtain $I_2=0$.
		
		Next, we consider the term $I_3(\cdot)$ and estimate it using H\"older's and Young's inequalities as
		\begin{align}\label{UE34}\nonumber
			&I_3(t)\\&\nonumber:=p\beta\int_{0}^{t}\int_{0}^{1}|u_{\e,m}(s,x)|^{p-2}u_{\e,m}(s,x)\c_m(u_{\e,m}(s,x)\d x\d s\\
			&\nonumber=p\beta\int_{0}^{t}\int_{0}^{1} \big((1+\gamma)|u_{\e,m}(s,x)|^{p-2}(u_{\e,m}(s,x))^{\delta+2}-\gamma|u_{\e,m}(s,x)|^p
			\nonumber\\&\qquad-|u_{\e,m}(s,x)|^{p-2}(u_{\e,m}(t,x))^{\delta+2}\big)\d x \d s\nonumber\\
			&\nonumber\leq p\beta \int_{0}^{t}\big((1+\gamma)\|u_{\e,m}(s)\|^{p+\delta}_{\L^{p+\delta}}-\gamma\|u_{\e,m}(s)\|^{p}_{\L^p}-\|u_{\e,m}(s)\|^{p+2\delta}_{\L^{p+2\delta}}\big) \d s\\
			&\nonumber\leq p\beta \int_{0}^{t}\big((1+\gamma)\|u_{\e,m}(s)\|^{\frac{p+2\delta}{2}}_{\L^{p+2\delta}}\|u_{\e,m}(s)\|^{\frac{p}{2}}_{\L^p}	-\gamma\|u_{\e,m}(s)\|^{p}_{\L^p}-\|u_{\e,m}(s)\|^{p+2\delta}_{\L^{p+2\delta}}\big) \d s\\
			&\nonumber\leq p\beta \int_{0}^{t}\left(\frac12\|u_{\e,m}(s)\|^{{p+2\delta}}_{\L^{p+2\delta}}+\frac{(1+\gamma)^2}{2}\|u_{\e,m}(s)\|^{p}_{\L^p}	-\gamma\|u_{\e,m}(s)\|^{p}_{\L^p}-\|u_{\e,m}(s)\|^{p+2\delta}_{\L^{p+2\delta}}\right) \d s\\
			&\leq p\beta \int_{0}^{t}\left(-\frac12\|u_{\e,m}(s)\|^{{p+2\delta}}_{\L^{p+2\delta}}+\frac{(1+\gamma^2)}{2}\|u_{\e,m}(s)\|^{p}_{\L^p}\right) \d s.
		\end{align}
		Using Hypothesis \ref{hyp1}, we obtain
		\begin{align}
			|I_5(t)|&\leq C(\e,p)\int_{0}^{t}\int_{0}^{1}|u_{\e,m}(s,x)|^{p-2}g_m^2(s,x,u_{\e,m}(s,x))x\d x\d s\nonumber\\&\leq C(\e,p)\int_0^t\int_{0}^{1}|u_{\e,m}(s,x)|^{p-2}\big(1+|u_{\e,m}(s,x))|^2\big)\d x\d s\nonumber\\&\leq C(\e,p,T)\left(1+\int_{0}^{t}\|u_{\e,m}(s)\|^{p}_{\L^p}\d s\right).
		\end{align}
		In order to estimate $I_4(t)$ we  apply BDG (see Theorem 1.1.7, \cite{SVLBLR}),  H\"older's and Young's inequalities to get 
		\begin{align}
			\mathbb{E}&\bigg[\sup_{t\in[0,T]}\bigg|p\sqrt{\e}\sum_{j\in\N}\int_{0}^{t}\int_{0}^{1}|u_{\e,m}(s,x)|^{p-2}u_{\e,m}(s,x)g_m(s,y,u_{\e,m}(s,x))q_j\varphi_j(x)\d x \d\beta_j(s)\bigg|\bigg]\nonumber\\
			&\leq C(p,\e)	\mathbb{E}\left[\sum_{j\in\N}\int_{0}^{T}\left(\int_{0}^{1}g_m(s,x,u_{\e,m}(s,x))|u_{\e,m}(s,x)|^{p-1}q_j\varphi_j(x)\d x\right)^2\d s\right]^{1/2}\nonumber\\
			&\leq C(p,K,\e)	\mathbb{E}\left[\sum_{j\in\N}\int_{0}^{T}\left(\int_{0}^{1}(1+|u_{\e,m}(s,x)|^{p}) q_j\varphi_j(x)\d x\right)^2\d s\right]^{1/2}\nonumber\\
			&\leq C(p,K,\e,,T)\nonumber\\&\qquad+C(p,K,\e)\E\left[\int_{0}^{T}\bigg(\int_{0}^{1}|u_{\e,m}(s,x)|^{p}\d x\bigg)\bigg(\sum_{j\in\N}q_j^2\int_0^1|u_{\e,m}(s,x)|^{p}
			\varphi_j^{2}(x)\d x\bigg)\d s\right]^{1/2}\nonumber\\
			&\leq C(p,K,\e,T)+C(p,K,\e)\E\left[\int_{0}^{T}\|u_{\e,m}(s)\|_{\L^p}^{2p}\d s\right]^{1/2}\nonumber\\&\leq  C(p,K,\e,T)+C(p,K,\e)\E\left[\sup_{s\in[0,T]}\|u_{\e,m}(s)\|^{\frac{p}{2}}_{\L^p}\left(\int_{0}^{T}\|u_{\e,m}(s)\|_{\L^p}^p
		 \d s\right)^{1/2}\right] \nonumber\\
			&\leq C(p,K,\e,T)+\frac12\mathbb{E}\bigg[\sup_{s\in[0,T]}\|u_{\e,m}(s)\|^{p}_{\L^p}\bigg]
			+C(p,K,\e)\int_{0}^{T}\mathbb{E} \big[\|u_{\e,m}(s)\|^{p}_{\L^p}\big]\d s.
		\end{align}
		Combining  the estimates for $I_1(t)$-$I_5(t)$, we deduce 
		\begin{align*}
			\mathbb{E}\bigg[\sup_{ t\in[0,T]}\|u_{\e,m}(t)\|^{p}_{\L^p}&+2\nu p(p-1)\int_0^T\||u_{\e,m}(s)|^{\frac{p-2}{2}}\partial_x u_{\e,m}(s)\|_{\L^2}^2\d s+p\beta\int_0^T\|u_{\e,m}(s)\|_{\L^{p+2\delta}}^{p+2\delta}\d s\bigg]\\& \leq 2\|u^0_m\|_{\L^p}^p+ C(\e,\alpha,\beta,\gamma,\delta,K,T,p) \int_0^T \E\big[\|u_{\e,m}(s)\|^{p}_{\L^p}\big]\d s.
		\end{align*}Applying Gronwall's inequality in the above inequality, we conclude
		\begin{align*}
			\mathbb{E}\bigg[\sup_{ t\in[0,T]}\|u_{\e,m}(t)\|^{p}_{\L^p}+2\nu p(p-1)&\int_0^T\||u_{\e,m}(s)|^{\frac{p-2}{2}}\partial_x u_{\e,m}(s)\|_{\L^2}^2\d s+p\beta\int_0^T\|u_{\e,m}(s)\|_{\L^{p+2\delta}}^{p+2\delta}\d s\bigg]
			\\&\leq C\big(1+\|u^0_m\|_{\L^p}^{p}\big),
		\end{align*}which completes the proof. 
	\end{proof}

The next task is to prove the tightness of $ u_{\e,m} $. The tightness of the terms $I_{\e,m}^2$ and $I_{\e,m}^3$, of \eqref{E20} hold from Lemmas \ref{lemma3.1}, \ref{lemma3.2}  and \ref{lemma3.6} in the space $\C([0,T];\L^p([0,1]))$, for $p\geq 2\delta+1$. Let us establish the tightness of the final term $I_{\e,m}^4$ in the right hand side of \eqref{E20}. 
	\begin{lemma}\label{lemE2}
		Assume that the initial data $u^0\in\L^p([0,1])$, for $p>6$, then the stochastic convolution term $I_{\e,m}^4$ is uniformly tight in $\C([0,T]\times[0,1])$.
	\end{lemma}
\begin{proof}The proof of this lemma is based the Arzela-Ascoli theorem and Aldous' tightness criterion. First, we prove that the stochastic term  $I_{\e,m}^4$ is uniformly bounded. We need the following fact about  Green's function $G$ given in \eqref{19} (see \cite{JBW})
	\[G(t,x,y)=\vi{p}(t,x,y)+H(t,x,y),\]
	where $H(t,x,y)$ is a smooth function in $(t,x,y)$ and $\vi{p}$ is the heat kernel defined as
	\[\vi{p}(t,x,y)=\frac{1}{\sqrt{4\pi t}}\exp\left(\frac{-(x-y)^2}{4t}\right).\]
	Using BDG inequality (Theorem 1.1.7, \cite{SVLBLR}), H\"older's inequality (with conjugate exponents $1\leq p_1< \frac{3}{2}$ and $1\leq q_1\leq \frac{p}{2}$), Lemma \ref{lemUE} and Equation (3.13) in  \cite{JBW}, we find for any $ t\in[0,T]$
	\begin{align}\label{T1}\nonumber
	&	\E\bigg[\bigg|\int_0^t\int_0^1 G(t-s,x,y)g(s,y,u_{\e,m}(s,y))\W^\Q(\d s,\d y)\bigg|^p\bigg] \\&\nonumber\leq C\E\bigg[\bigg(\int_0^t\int_0^1G^2(t-s,x,y)g^2(s,y,u_{\e,m}(s,y))\d y\d s\bigg)^{\frac{p}{2}}\bigg]\\&\nonumber\leq C\bigg(\int_0^t\int_0^1 G^{2p_1}(t-s,x,y)\d y\d s\bigg)^{\frac{p}{2p_1}}\E\bigg[\bigg(\int_0^t\int_0^1 g^{2q_1}(s,y,u_{\e,m}(s,y))\d y \d s\bigg)^\frac{p}{2q_1}\bigg]\\
	&\nonumber\leq 
	C(K)t^{\frac{(3-2q_1)p}{4q_1}+\frac{p}{2q_1}-1}\E\bigg[\int_0^t\int_0^1 g^{p}(s,y,u_{\e,m}(s,y))\d y \d s\bigg]\\&\nonumber\leq 
	C(p,K,T)\bigg(1+\E\bigg[\sup_{t\in[0,T]}\|u_{\e,m}(s)\|_{\L^p}^p\bigg]\bigg)
	\\&\leq C(p,K,T,u^0),
	\end{align}
for all $m\in\N$. Note that $I_{\e,m}^4(x,t)$ can be written as 
\begin{align}\label{228}\nonumber
	I_{\e,m}^4(x,t)&=	\int_0^t\int_0^1\tilde{p}(t-s,x,y)g_m(s,y,u_{\e,m}(s,y))\W^\Q(\d s,\d y)\nonumber\\&\nonumber\quad+\int_0^t\int_0^1H(t-s,x,y)g_m(s,y,u_{\e,m}(s,y))\W^\Q(\d s,\d y)\\&=:I_{\e,m}^{4,1}(x,t)+I_{\e,m}^{4,2}(x,t),
\end{align}
for $(x,t)\in[0,1]\times[0,T]$. The second term on the right hand side of \eqref{228}  is a convolution of $\W^{\Q}$ with a smooth function $H$. It can also shown to be smooth and hence we consider the first term only. 
 It can be easily seen that 
\begin{align}\label{TE1}\nonumber
	&\big\{\E\big[\big|I_{\e,m}^{4,1}(x+h,t+k)-I_{\e,m}^{4,1}(x,t)\big|^p\big]\big\}^{\frac{1}{p}}\\&\nonumber\leq C(p)\big\{\E\big[\big|I_{\e,m}^{4,1}(x+h,t+k)-I_{\e,m}^{4,1}(x,t+k)\big|^p\big]\big\}^{\frac{1}{p}} \\&\quad +C(p)\big\{\E\big[\big|I_{\e,m}^{4,1}(x+h,t+k)-I_{\e,m}^{4,1}(x+h,t)\big|^p\big]\big\}^{\frac{1}{p}} .
\end{align} 
 Let us  consider the term $\E\big[\big|I_{\e,m}^{4,1}(x+h,t)-I_{\e,m}^{4,1}(x,t)\big|^p\big]$, and  estimate it using BDG inequality (Theorem 1.1.7, \cite{SVLBLR}), H\"older's inequality (with conjugate  exponents $\frac{p}{2}$ and $q:=\frac{p}{p-2}$), \eqref{E9} and Lemma \ref{lemUE}, as
\begin{align}\label{TE2}\nonumber
	&\E\big[\big|I_{\e,m}^{4,1}(x+h,t)-I_{\e,m}^{4,1}(x,t)\big|^p\big] \\&\nonumber\leq C(p) \E\bigg[\bigg|\int_0^t\int_0^1\big(\vi{p}(t-s,x+h,y)-\vi{p}(t-s,x,y)\big)^2g^2(s,y,u_{\e,m}(s,y))\d y\d s\bigg|^\frac{p}{2}\bigg]\\&\nonumber\leq 
	C(p)\E\bigg[\int_0^t\int_0^1g^p(s,y,u_{\e,m}(s,y))\d y\d s\bigg]\\&\qquad\times\nonumber\bigg[\int_0^t\int_0^1\big|\vi{p}(t-s,x+h,y)-\vi{p}(t-s,x,y)\big|^{2q}\d y\d s\bigg]^{\frac{p}{2q}}\\&\leq C(p,K,T,u^0)\bigg[\int_0^t\int_{0}^{1}\frac{1}{(4\pi (t-s))^{q}}\big|e^{\frac{-(y-x-h)^2}{4(t-s)}}-e^{-\frac{(y-x)^2}{4(t-s)}}\big|^{2q}\d y\d s\bigg]^\frac{p}{2q}.
\end{align}Fix $y-x=hz$ and $t-s=h^2 \ell$, in the above inequality, to find 
\begin{align}\label{TE3}\nonumber
	&\E\big[\big|I_{\e,m}^{4,1}(x+h,t)-I_{\e,m}^{4,1}(x,t)\big|^p\big] \\&\nonumber\leq C(p,K,T,u^0) \bigg[h^{3-2q}\int_0^{t}\int_0^\infty \ell^{-q}\big|e^{-\frac{(z-1)^2}{4\ell}}-e^{-\frac{z^2}{4\ell}}\big|^{2q}\d z \d \ell\bigg]^{\frac{p}{2q}} \\&\leq C(p,K,T,u^0)h^{\frac{p}{2}-3},
	\end{align}provided $q< \frac{3}{2}$, which implies $p>6$. Thus, we conclude that 
\begin{align}\label{TE03}
	\big\{\E\big[\big|I_{\e,m}^{4,1}(x+h,t+k)-I_{\e,m}^{4,1}(x,t+k)\big|^p\big]\big\}^{\frac{1}{p}}\leq C(p,K,T,u^0)h^{\frac{1}{2}-\frac{3}{p}}, \ \ \text{ for } \ \ p>6.
\end{align}

Let us now consider the second term of the inequality \eqref{TE1}. Applying BDG inequality (Theorem 1.1.7, \cite{SVLBLR}), we find
\begin{align}\label{TE4}\nonumber
	&\big\{\E\big[\big|I_{\e,m}^{4,1}(x,t+k)-I_{\e,m}^{4,1}(x,t)\big|^p\big]\big\}^{\frac{1}{p}} \\&\nonumber\leq C(p)\bigg\{\E\bigg[\bigg|\int_0^t\int_0^1\big(\vi{p}(t+k-s,x,y)-\vi{p}(t-s,x,y)\big)^2g^2(s,y,u_{\e,m}(s,y))\d y\d s\bigg|^{\frac{p}{2}}\bigg]\bigg\}^{\frac{1}{p}} \\&\nonumber\quad+ C(p)\bigg\{\E\bigg[\bigg|\int_t^{t+k}\int_0^1\vi{p}^2(t+k-s,x,y)g^2(s,y,u_{\e,m}(s,y))\d y\d s\bigg|^{\frac{p}{2}}\bigg]\bigg\}^{\frac{1}{p}} \\
	&=:J_1+J_2.
\end{align}We consider  the term $J_1$ of above inequality, and estimate it  using H\"older's inequality and similar arguments as in \eqref{TE2} by putting $t-s=k \ell$, $y-x=\sqrt{k}z$,  as
\begin{align}\label{TE5}\nonumber
	J_1&\leq C(p)\bigg\{ \E\bigg[\int_0^t\int_0^1 g^p(s,y,u_{\e,m}(s,y))\d y\d s\bigg]\bigg\}^{\frac{1}{p}}\\&\nonumber\quad\times\bigg[\int_0^t\int_{0}^{\infty} \big|
	(k\ell+k)^{-\frac{1}{2}}e^{-\frac{z^2}{4(\ell+1)}}-(k\ell)^{-\frac{1}{2}}e^{-\frac{z^2}{4\ell}}\big|^{2q}\d z\d \ell\bigg]^{\frac{1}{2q}}\\&\nonumber\leq C(p,K,T,u^0)\bigg[k^{\frac{3}{2}-q}\int_0^t\int_{0}^\infty\big|(\ell+1)^{-\frac{1}{2}}e^{-\frac{z^2}{4(\ell+1)}}-\ell^{-\frac{1}{2}}e^{-\frac{z^2}{4\ell}}\big|^{2q}\d z\d \ell\bigg]^\frac{1}{2q} \\&\leq  C(p,K,T,u^0) k^{\frac{1}{4}-\frac{3}{2p}}, 
\end{align}provided $q<\frac{3}{2}$ which implies $p>6$. Finally, we consider the term $J_2$ of \eqref{TE4} and estimate it using H\"older's inequality and Lemma \ref{lemUE} as
\begin{align}\label{TE6}\nonumber
	J_2&\leq C(p)\E\bigg[\int_t^{t+k}g^p(s,y,u_{\e,m}(s,y))\d y\d s\bigg]^{\frac{1}{p}}\bigg[\int_0^k\int_0^1\vi{p}^{2q}(t-s,x,y)\d y\d s\bigg]^{\frac{1}{2q}}\\&\nonumber\leq C(p,K,T,u^0)k^{\frac{1}{p}}\bigg[\int_0^k\int_{0}^{\infty}(t-s)^{-q}e^{-\frac{q(y-x)^2}{t-s}}\d y\d s\bigg]^{\frac{1}{2q}} \\&\leq C(p,K,T,u^0) k^{\frac{1}{4}-\frac{1}{2p}}, \ \ \text{ for } \ \ p>6.
\end{align}Combining \eqref{TE1}-\eqref{TE6}, we obtain for $p>6$
\begin{align}\label{236}
	\left\{\E\big[\big|I_{\e,m}^{4,1}(x+h,t+k)-I_{\e,m}^{4,1}(x,t)\big|^p\big]\right\}^\frac{1}{p} &\leq C\left[h^{\frac{1}{2}-\frac{3}{p}}+k^{\frac{1}{4}-\frac{3}{2p}}\right], 
\end{align}
For  $p>10$, using Kolmogorov's continuity Theorem  (see Corollary 1.2, \cite{JBW}) we obtain a modification  of $I_{\e,m}^{4,1}(\cdot,\cdot)$ with $\mathbb{P}$-almost all trajectories being H\"older continuous.

By an application of Markov's inequality and using uniform bound of the stochastic term (cf. \eqref{T1}), we have
\begin{align*}
	\lim_{M\to\infty} \limsup_{m\in\N}\P\bigg(|I_{\e,m}^4(t,x)|\geq M\bigg)\leq\lim_{M\to\infty}\frac{C(p,K,T,u^0)}{M^p}=0.
\end{align*}
for all $(t,x)\in[0,T]\times[0,1]$. Furthermore, under the parabolic metric $|(t,x)|=|t|^\frac{1}{2}+|x|$, the condition \eqref{236} implies for any $\e_1>0$, there exists an $\eta(\e_1)=\e_1^{\frac{3+\tilde{\delta}}{p}-\frac{1}{2}},$ for some $\tilde{\delta}>\frac{p}{2}-3>0$  such that 
\begin{align*}
\lim_{\eta\to 0}	\limsup_{m\in\N}\P\bigg(\sup_{ |(t-s,x-y)|\leq\eta}|I_{\e,m}^4(t,x)-I_{\e,m}^4(s,y)|\geq \e_1 \bigg)\leq \lim_{\eta\to 0}\frac{C^p\eta^{\frac{p}{2}-3}}{\e_1^p} =C^p\lim_{\eta\to 0}\eta^{\tilde{\delta}}=0,
\end{align*}
for all $(t,x)\in[0,T]\times[0,1]$.  Hence, by an application of Theorem 7.3, \cite{PB} yields that  $I_{\e,m}^4$ is tight in the space $\C([0,T]\times[0,1])$.
\end{proof}

Combining the tightness obtained above, we conclude that the sequence of the processes \begin{align*}u_{\e,m}=	I_{m}^1+\frac{\alpha}{\delta+1}I_{\e,m}^2+\beta I_{\e,m}^3+\sqrt{\e}I_{\e,m}^4
	\end{align*}
	is uniformly tight in $\C([0,T];\L^p([0,1]))$ for all $m$ and $p>\max\{6,2\delta+1\}$. Let us now move to the  global well-posedness of the integral equation \eqref{2.2}. A similar result can be found in \cite{IG} (cf. \cite{IGCR,AKMTM3} also).
	\begin{theorem}\label{thrmE2}
		Assume that Hypothesis \ref{hyp1} holds and the initial data $u^0\in\L^p([0,1])$, for $p>\max\{6,2\delta+1\}$. Then, there exists a unique $\L^p([0,1])$-valued $\mathscr{F}_t$-adapted continuous process $u_\e(\cdot)$    satisfying the integral equation \eqref{2.2} such that 
		\begin{align}\label{E31}
			\E\bigg[\sup_{t\in[0,T]}\|u_\e(t)\|_{\L^p}^p\bigg]\leq C(T).
		\end{align}
	\end{theorem}
	\begin{proof}
		The proof of uniqueness is quite easy. Let us assume $u_\e(\cdot)$ and $v_\e(\cdot)$ be any two solutions of the integral equation \eqref{2.2}.
		Let us define the following stopping times 
		\begin{align*}
			\tau_{1,R}:=\inf\{t\geq0: \|u_\e(t)\|_{\L^p}\geq R\}\wedge T, \ \text{ and } \ 
			\tau_{2,R}:=\inf\{t\geq0: \|v_\e(t)\|_{\L^p}\geq R\}\wedge T.
		\end{align*}By Theorem \ref{thrmE1}, we obtain $u_\e(t)=v_\e(t),\ t\in[0,\tau_{1,R}\wedge \tau_{2,R}]$.  By the uniform energy estimate \eqref{E31} and an application of Markov's inequality yields \begin{align*}
			\P\big(\tau_{1,R}\wedge \tau_{2,R}<T\big)\to0 \ \ \text{ as } \ \ R\to\infty.
		\end{align*}Hence, we deduce the uniqueness of the solution on the whole interval $[0,T]$.
		
		Let us now establish the existence of the solution $u_\e(\cdot)$ in the interval $[0,T]$. The main ingredient of this proof is the Skorokhod representation theorem (see Theorem 3, \cite{RMD}) and Lemma 1.1, \cite{IGNK}. With the help Lemmas \ref{lemma3.6} and \ref{lemE2}, we obtain the tightness of $u_{\e,m}(\cdot)$, for each $m$ and $p>\max\{6,2\delta+1\}$. For a given pair of subsequences $\{u_{\e,l}\}$ and $\{u_{\e,k}\}$, by Prokhorov's theorem and Skorokhod's representation theorem, there exist subsequences $l_j$ and $k_j$ (denoted by $u_{\e,l_j},\ u_{\e,k_j}$) and a subsequence of random variables $z_{\e,j}:=\big(\vi{u}_{\e,l_j},\bar{u}_{\e,k_j},\bar{\W}_j^\Q\big)_{j\in\N}\in \mathcal{I}$, where $\mathcal{I}:=\C([0,T];\L^p([0,1]))\times \C([0,T];\L^p([0,1]))\times  \C([0,T]\times[0,1])$, for $p>\max\{6, 2\delta+1\}$ in some probability space $(\bar{\Omega},\bar{\mathscr{F}},\bar{\P})$ such that the sequence $z_j$ converges $\bar{\P}$-a.s. in $\mathcal{I}$ to a random variable $z_\e:=\big(\vi{u}_\e,\bar{u}_\e,\bar{\W}^\Q\big)$ as $j\to\infty$, with the same distributions of $z_{\e,j}$ and $\big(u_{\e,l_j},u_{\e,k_j}, \W^\Q\big)$. Here, the two random fields $ \bar{\W}^\Q$ and $ \bar{\W}_j^\Q$ are Brownian fields defined on  different stochastic bases $\Xi= (\bar{\Omega},\bar{\mathscr{F}},\{\bar{\mathscr{F}}_t\}_{t\geq 0},\bar{\P})$ and $\bar{\Xi}_j= (\bar{\Omega},\bar{\mathscr{F}},\{\bar{\mathscr{F}}_t^j\}_{t\geq 0},\bar{\P})$, respectively, where $\bar{\mathscr{F}}_t$ and $\bar{\mathscr{F}}_t^j$ are the completion of the $\sigma$-fields generated by $z_\e(t,x)$ and $z_{\e,j}(t,x)$ for all $s\in[0,t]$, $x\in[0,1]$, respectively. By Lemma 1.1, \cite{IGNK}, we only need to verify that $z_{\e,j}$ converges to $z_\e$ weakly. We know that weak solution and mild solution are equivalent, interested readers can find a proof in \cite{IGCR} (see Proposition 3.7). 	Passing the limit $j\to\infty$ in the weak formulation \eqref{WF} with $\vi{u}_{\e,j}(\cdot)$ replaced by $u_\e(\cdot)$, for every smooth function $\phi\in \C^2([0,1])$, with $\phi(0)=\phi(1)=0$, we obtain
		\begin{align}\label{E32}\nonumber
			\int_0^1 \vi{u}_{\e}(t,y)& \phi(y)\d y=\int_0^tu^0(y)\phi(y)\d y+\nu\int_0^t\int_0^1\vi{u}_{\e}(s,y)\phi''(y)\d y\d s\\&\nonumber\quad +\beta \int_0^t\int_0^1\c(\vi{u}_{\e}(t,y))\phi(y)\d y\d s+\frac{\alpha}{\delta+1}\int_0^t\int_0^1\p(\vi{u}_\e(s,y))\phi'(y)\d y\d s\\&\quad +\sqrt{\e}\int_0^t\int_0^1g(s,y,\vi{u}_\e(s,y))\phi(y)\bar{\W}^\Q(\d s,\d y), \ \ \bar{\P}\text{-a.s.},
		\end{align}for all $t\in[0,T]$ on the probability space $\bar{\Xi}$.  Furthermore, \eqref{E32} remains valid for $\bar{u}_\e$, then by the uniqueness of the solution we conclude that $\bar{u}_\e=\vi{u}_\e$. Now, applying Lemma 1.1 \cite{IGNK}, we infer that $u_{\e,m}$ converges in $\C([0,T];\L^p([0,1]))$, for $p>\max\{6, 2\delta+1\}$ in probability to some element $u_\e\in\C([0,T];\L^p([0,1]))$, for $p>\max\{6, 2\delta+1\}$. Hence, the proof is over. 
	\end{proof}
	\begin{remark}\label{rem3.6}
		Let us assume that the initial data $u^0\in\C([0,1])$. Using the above existence and uniqueness result, Lemmas \ref{lemma3.1}, \ref{lemma3.2} and \ref{lemma3.6}, we deduce that $u_\e$ (solution of the integral equation \eqref{2.2}) has a modification with continuous paths on $[0,T]$ with the values in $\C([0,1])$.
	\end{remark}

	\section{Large deviation principle} \label{Sec4}\setcounter{equation}{0}
	In this section, first we recall some basic definitions related to the LDP  and then we state a sufficient condition of the LDP, followed by our main Theorem \ref{thrmLDP4} of this section.
	\begin{definition}[LDP]\label{LDP}
		Let $\I:\EE\to[0,\infty]$ be a rate function on a Polish space $\EE$. The sequence $\{X_\e\}_{\e\in(0,1]}$ satisfies the LDP on the space $\EE$ with the rate function $\I(\cdot)$, if the following hold:
		\begin{enumerate}
			\item For any closed subset $\F\subset\EE$, 
			\begin{align*}
				\limsup_{\e\to0}\e\log \P\big(X_\e\in\F\big) \leq -\inf_{x\in\F}\I(x).
			\end{align*}
			\item For any closed subset $\mathrm{O}\subset\EE$, 
			\begin{align*}
				\liminf_{\e\to0}\e\log \P\big(X_\e\in\F\big) \geq -\inf_{x\in\mathrm{O}}\I(x).
			\end{align*}
		\end{enumerate}
	\end{definition}
	Our proof of LDP for the solution of the system \eqref{1.1}-\eqref{1.2} is based on a weak convergence approach. From the work \cite{MS}, one can obtain the equivalency between the Laplace principle and LDP.
	
	\begin{definition}[Laplace principle]\label{LP}
		The sequence of the random variable $\{X_\e\}$ on the Polish space $\EE$ is said to satisfy the Laplace principle with the rate function $\I(\cdot)$ if for each $f\in \C_b(\EE)$ (space of bounded continuous functions from $\EE$ to $\R$), 
		\begin{align*}
			\lim_{\e\to0}\e\log \E\bigg\{\exp\bigg[-\frac{f(X_\e)}{\e}\bigg]\bigg\}=-\inf_{x\in\EE}\big\{f(x)+\I(x)\big\}.
		\end{align*}
	\end{definition}
On the probability space  $(\Omega,\mathscr{F},\{\mathscr{F}_t\}_{t\geq 0},\P)$, let us define a predictable process $\phi:\Omega\times[0,T]\to \L^2([0,T]\times[0,1])$ and the following sets:
	\begin{align*}
		\UU^M &:=\bigg\{\phi\in\L^2([0,T]\times[0,1]): \int_0^T\int_0^1|\phi(s,y)|^2\d y\d s\leq M \bigg\}, \ M\in\N,\\
		\PP_2&:=\bigg\{\phi:\int_0^T\int_0^1|\phi(s,y)|^2\d y\d s<\infty, \ \ \P\text{-a.s.}\bigg\},\\
		\PP_2^M&:=\bigg\{\phi\in\PP_2: \ \phi(\omega)\in\UU^M, \ \ \P\text{-a.s.}\bigg\}, 
	\end{align*}where $\UU^M$ is a compact metric space equipped with the weak topology of $\L^2([0,T]\times[0,1])$, $\PP_2^M$ is an admissible control set. Let $\EE$ and $\EE_0$ be any two Polish spaces. We  assume that the initial data takes values in a compact subspace of $\EE_0$ and the solution in $\EE$. For every $\e>0$, define $\mathscr{G}_\e:\EE_0\times\C([0,T]\times[0,1])\to\EE$ as a family of measurable maps. Set $\mathrm{Y}_{\e,u^0}=\mathscr{G}_\e\big(u^0,\sqrt{\e}\W^\Q\big)$. Let us provide the sufficient condition for the uniform Laplace principle (ULP) for the family $\mathrm{Y}_{\e,u^0}$.
	\begin{condition}\label{Cond2}
		There exists a measurable map $\mathscr{G}_0:\EE_0\times \C([0,T]\times[0,1])\to\EE$ such that the following hold:
		\begin{enumerate}
			\item For $M\in\N$, let $\phi_m,\phi\in\UU^M$ be such that $\phi_m\to\phi$ and $u^0_m\to u^0$ as $m\to\infty$. Then 
			\begin{align*}
				\mathscr{G}_0\bigg(u^0_m,\int_0^{\cdot}\int_0^{\cdot}\phi_m(s,y)\d y\d s\bigg)\to 	\mathscr{G}_0\bigg(u^0,\int_0^{\cdot}\int_0^{\cdot}\phi(s,y)\d y\d s\bigg), \ \ \text{ as } \ \ m\to\infty.
			\end{align*}
			\item For $\M\in\N$, let $\phi_\e,\phi\in\PP_2^M$  be such that $\phi_\e$ converges in distribution (as $\UU^M$  valued random elements) to $\phi$, and the initial data $u^0_\e\to u^0$ as $\e\to0$ in $\EE_0$. Then 
			\begin{align*}
				\mathscr{G}_\e\bigg(u^0_\e,\sqrt{\e}\W^\Q+\int_0^{\cdot}\int_0^{\cdot}\phi_\e(s,y)\d y\d s\bigg)\Rightarrow	\mathscr{G}_0\bigg(u^0,\int_0^{\cdot}\int_0^{\cdot}\phi(s,y)\d y\d s\bigg),  \text{ as }  \e\to 0,
			\end{align*}where $``\Rightarrow"$ represents the convergence in distribution.
		\end{enumerate}
	\end{condition}
	For $\psi\in\EE$ and $u^0\in\EE_0$, we define 
	\begin{align*}
		\SS_\psi:=\bigg\{\phi\in\L^2([0,T]\times[0,1]): \ \psi=\mathscr{G}_0\bigg(u^0,\int_0^{\cdot}\int_0^{\cdot}\phi(s,y)\d y\d s\bigg)\bigg\}, 
	\end{align*}and the rate function as 
	\begin{equation}\label{RF}
		\I_{u^0}(\psi)=
		\left\{
		\begin{aligned}
			&	\inf_{\psi\in\SS_{\psi}}\bigg\{\frac{1}{2}\int_0^T\int_0^1|\phi(s,y)|^2\d y\d s\bigg\}, \ \ &&\text{ if } \ \ \psi\in\EE,\\
			&+\infty, \ \ &&\text{ if } \ \ \SS_\psi=\emptyset,
		\end{aligned}
		\right.
	\end{equation}where \begin{align*}&\inf_{\psi\in\SS_{\psi}}\bigg\{\frac{1}{2}\int_0^T\int_0^1|\phi(s,y)|^2\d y\d s\bigg\}\\&\hspace{1cm}=\inf_{\left\{\phi\in \L^2([0,T]\times[0,1]): \psi=\mathscr{G}_0(u^0,\int_0^{\cdot}\int_0^{\cdot}\phi(s,y)\d y\d s)\right\}}\bigg\{\frac{1}{2}\int_0^T\int_0^1|\phi(s,y)|^2\d y\d s\bigg\}
	\end{align*}

	Let us recall a general criteria for the LDP form \cite{ADPDVM}.
	
	\begin{theorem}[Theorem 7, \cite{ADPDVM}]\label{thrmLDP}
		Let $\mathscr{G}_0:\EE_0\times \C([0,T]\times[0,1])\to\EE$ be a measurable map, and the Condition \ref{Cond2} holds. Then, the family $\{\mathrm{Y}_{\e,u^0}\}$ satisfies the Laplace principle on $\EE$ with the rate function $\I_{u^0}(\cdot)$ given in \eqref{RF}, uniformly over the initial data $u^0$ on the compact subsets of $\EE_0$.
	\end{theorem}
	If we verify Condition \ref{Cond2}, then by Theorem \ref{thrmLDP}, the family $\{\mathrm{Y}_{\e,u^0}\}$ satisfies the Laplace principle and by equivalency we obtain the family $\{\mathrm{Y}_{\e,u^0}\}$ satisfies LDP. Now, our focus is to verify the Condition \ref{Cond2}. 
	
	\subsection{Controlled and skeleton equations}
	In this subsection, we discuss controlled and skeleton equations. Let us fix $\EE=\C([0,T];\L^p([0,1]))$,  $\EE_0=\L^p([0,1])$  and the solution map of the integral equation \eqref{2.2} represented by $u_\e=\mathscr{G}_\e\big(u^0,\sqrt{\e}\W^\Q\big)$. Let us denote the solution $u_{\e,\phi}(t,x):=\mathscr{G}_\e\big(u^0,\sqrt{\e} \W^\Q+\int_0^{\cdot}\int_0^{\cdot}\phi(s,y)\d y\d s\big)$ of the following stochastic controlled integral (SCI) equation 
	\begin{align}\label{CE1}
		\nonumber
		u_{\e,\phi}(t,x)&=\int_0^1G(t,x,y)u^0(y)\d y+\frac{\alpha}{\delta+1}\int_0^t\int_0^1\frac{\partial G}{\partial y}(t-s,x,y)\p(u_{\e,\phi}(s,y))\d y\d s\\&\nonumber\quad +\beta\int_0^t\int_0^1G(t-s,x,y)\c(u_{\e,\phi}(s,y))\d y\d s\\&\nonumber\quad+\sqrt{\e} 
		\int_0^t\int_0^1G(t-s,x,y)g(s,y,u_{\e,\phi}(s,y))\W^\Q(\d s,\d y)\\&\quad+ 
		\int_0^t\int_0^1G(t-s,x,y)g(s,y,u_{\e,\phi}(s,y))\phi(s,y)\d y\d s.
	\end{align}If we choose $\e=0$, then the above SCI equation reduced to skeleton equation or limiting case. Let us denote the solution  $u_{0,\phi}(t,x):=\mathscr{G}_0\big(u^0,\int_0^{\cdot}\int_0^{\cdot}\phi(s,y)\d y\d s\big)$ of the following skeleton equation 
	\begin{align}\label{CE2}
		\nonumber
		u_{0,\phi}(t,x)&=\int_0^1G(t,x,y)u^0(y)\d y+\frac{\alpha}{\delta+1}\int_0^t\int_0^1\frac{\partial G}{\partial y}(t-s,x,y)\p(u_{0,\phi}(s,y))\d y\d s\\&\nonumber\quad +\beta\int_0^t\int_0^1G(t-s,x,y)\c(u_{0,\phi}(s,y))\d y\d s\\&\quad+ 
		\int_0^t\int_0^1G(t-s,x,y)g(s,y,u_{0,\phi}(s,y))\phi(s,y)\d y\d s.
	\end{align}For $\psi\in \C([0,T];\L^p([0,1]))$, we define the rate function
	\begin{align}\label{CE3}
		\I_{u^0}(\psi ):=\frac{1}{2}\inf_{\left\{\phi\in \L^2([0,T]\times[0,1]): \psi=\mathscr{G}_0(u^0,\int_0^{\cdot}\int_0^{\cdot}\phi(s,y)\d y\d s)\right\}} \int_0^T\int_0^1|\phi(s,y)|^2\d y\d s,
	\end{align}where $\psi$ satisfies the skeleton equation
	\begin{align}\label{CE4}
		\nonumber
		\psi(t,x)&=\int_0^1G(t,x,y)u^0(y)\d y+\frac{\alpha}{\delta+1}\int_0^t\int_0^1\frac{\partial G}{\partial y}(t-s,x,y)\p(\psi(s,y))\d y\d s\\&\nonumber\quad +\beta\int_0^t\int_0^1G(t-s,x,y)\c(\psi(s,y))\d y\d s\\&\quad+ 
		\int_0^t\int_0^1G(t-s,x,y)g(s,y,\psi(s,y))\phi(s,y)\d y\d s.
	\end{align}The following results discusses the existence and uniqueness of solution to the above stochastic controlled and skeleton integral equations.
	\begin{theorem}\label{thrmLDP1}
		Assume that Hypothesis \ref{hyp1} holds and $\phi\in\PP_2^M$. Then, for any initial data $u^0\in\L^p([0,1])$,  for $p>\max\{6,2\delta+1\}$, there exists a unique $\L^p([0,1])$-valued $\mathscr{F}_t$-adapted continuous process $u_{\e,\phi}(\cdot,\cdot)$ satisfying SCI equation \eqref{CE1}.
	\end{theorem}
	\begin{proof}
		The proof of this theorem is based on the Girsanov theorem  (Theorem 10.14, \cite{DaZ}). 
		
		For any fixed $\phi \in\PP_2^M$, we define 
		\begin{align*}
			\frac{\d \vi{\P}}{\d \P}:=\exp\bigg\{-\frac{1}{\sqrt{\e}}\int_0^T\int_0^1\phi(s,y) \W^\Q(\d s,\d y)-\frac{1}{2\e}\int_0^T\int_0^1|\phi(s,y)|^2\d y\d s\bigg\},
		\end{align*}
		where $\W^\Q(\cdot,\cdot)$ is a white in time and colored in space noise  and the stochastic integral $\displaystyle\int_0^t\int_0^1\phi(s,y)\W^\Q(\d s,\d y)$ is considered in the sense of Walsh (see \cite{JBW}). Then the stochastic process 
		\begin{align*}\vi{\W}^\Q(t,x):= \W^\Q(t,x)+\frac{1}{\sqrt{\e}}\int_0^t\int_0^x\phi(s,y)\d y\d x, \ \ (t,x)\in[0,T]\times[0,1], 
		\end{align*}
		is a white in time and colored in space noise  under the probability measure $\vi{\P}$. Also, 
		\begin{align*}
			\exp\bigg\{-\frac{1}{\sqrt{\e}}\int_0^T\int_0^1\phi(s,y)\W^\Q(\d s,\d y)-\frac{1}{2\e}\int_0^T\int_0^1|\phi(s,y)|^2\d y\d s\bigg\},
		\end{align*}is an exponential martingale and $\vi{\P}$ is an another probability measure on the probability space $(\Omega,\mathscr{F},\{\mathscr{F}_t\}_{t\geq 0},\P)$ and $\vi{\P}$ is equivalent to the probability measure $\P$. By Girsanov's theorem, we conclude that $\vi{\W}^\Q(\cdot,\cdot)$ is a real valued Wiener process with respect to the new probability measure $\vi{\P}$, and using the equivalency of the measure, we find that $\vi{\W}^\Q(t,x)$ is a  white in time and colored in space noise. On the probability space $(\Omega,\mathscr{F},\{\mathscr{F}_t\}_{t\geq 0},\vi{\P})$, using Theorem \ref{thrmE2} we ensure the existence and uniqueness of the solution of the SCI equation \eqref{CE1} for the newly constructed probability measure $\vi{\P}$. Hence, we obtain the well-posedness for the  SCI equation \eqref{CE1} with respect to the probability measure $\P$  also.
	\end{proof}
	
	\begin{theorem}\label{thrmLDP3}
		Assume that Hypothesis \ref{hyp1}, and $\phi\in\L^2([0,T]\times[0,1])$. Then, for any initial data $u^0\in\L^p([0,1])$, for $p\geq 2\delta+1$, there exists a unique solution $u_{0,\phi}\in \C([0,T];\L^p([0,1]))$ to the skeleton integral equation \eqref{CE2}.
	\end{theorem}
	\begin{proof}
		Since $\phi\in\L^2([0,T]\times[0,1])$, 	and 
		\begin{align*}
			&\left\|	\int_0^t\int_0^1G(t-s,x,y)g(s,y,u_{0,\phi}(s,y))\phi(s,y)\d y\d s\right\|_{\L^p}\\&\leq C\left(\int_0^t(t-s)^{\frac{1}{p}-1}\|g(s,u_{0,\phi}(s))\|_{\L^2}^2\d s\right)^{1/2}\left(\int_0^t\|\phi(s)\|_{\L^2}^2\d s\right)^{1/2}\nonumber\\&\leq CK\|\phi\|_{\L^2([0,T]\times[0,1])}\left(\int_0^t(t-s)^{\frac{1}{p}-1}\big(1+\|u_{0,\phi}(s)\|_{\L^2}\big)^2\d s\right)^{1/2}\\&\leq CKT^{\frac{1}{2p}}\|\phi\|_{\L^2([0,T]\times[0,1])}\sup_{t\in[0,T]}\big\{1+\|u_{0,\phi}(t)\|_{\L^2}\big\}<\infty,
		\end{align*}
		the well-posedness of the skeleton integral equation \eqref{CE2}  can be obtained  in a similar way as in the proof of Proposition 3.4, \cite{AKMTM3} (see Section \ref{Sec3} above also) with some minor modifications in the proof. 
	\end{proof}
	\subsection{Proof of LDP for the solution of the system \eqref{1.1}-\eqref{1.2}} In this subsection, we state and prove the main result of this section. Let us first  state the result. 
	\begin{theorem}\label{thrmLDP4}
		Assume that Hypothesis \ref{hyp1} holds and the initial data $u^0\in\L^p([0,1])$, for $p>\max\{6,2\delta+1\}$. Then, there exists a stochastic process $\{u_\e(t,x):\ t\in[0,T], \ x\in[0,1]\}$ satisfying the Laplace principle on $\C([0,T];\L^p([0,1]))$ with the rate function $\I_{u^0}(\cdot)$ defined in \eqref{RF}, uniformly over the initial data $u^0$ on the compact subsets of $\EE_0$. 
	\end{theorem}
	In view of Theorem \ref{thrmLDP}, our focus will be on the verification of Condition \ref{Cond2}. Validation of  Condition \ref{Cond2} (2) is easy compared to (1). Therefore, we are providing a proof of verification of Condition \ref{Cond2} (2) only.
	\begin{proof}
		First, we  establish the convergence of the controlled process. Let $M>0$, and  $\phi_\e,\phi\in\PP_2^M$  be such that $\phi_\e\Rightarrow\phi$, and the initial data $u^0_\e\to u^0$ as $\e\to0$. Our aim is to show that  $$	u_{\e,\phi_\e}=\mathscr{G}_\e\bigg(u^0_\e,\sqrt{\e}\W^\Q+\int_0^{\cdot}\int_0^{\cdot}\phi_\e(s,y)\d y\d s\bigg)\Rightarrow u_{0,\phi}=	\mathscr{G}_0\bigg(u^0,\int_0^{\cdot}\int_0^{\cdot}\phi(s,y)\d y\d s\bigg),$$ as $\e\to 0$,  where  $u_{0,\phi}$ is the solution of the integral equation \eqref{CE2}. In order to establish this,  we use  compactness arguments.
		
		Consider the following integral equation for $(t,x)\in[0,T]\times[0,1]$:
		\begin{align}\label{CE5}
			\nonumber
			u_{\e,\phi_{\e}}(t,x)&=\int_0^1G(t,x,y)u^0_\e(y)\d y+\frac{\alpha}{\delta+1}\int_0^t\int_0^1\frac{\partial G}{\partial y}(t-s,x,y)\p(u_{\e,\phi_{\e}}(s,y))\d y\d s\\&\nonumber\quad +\beta\int_0^t\int_0^1G(t-s,x,y)\c(u_{\e,\phi_{\e}}(s,y))\d y\d s\\&\nonumber\quad+\sqrt{\e} 	\int_0^t\int_0^1G(t-s,x,y)g(s,y,u_{\e,\phi_{\e}}(s,y))\W^\Q(\d s,\d y)\\&\nonumber\quad+ 	\int_0^t\int_0^1G(t-s,x,y)g(s,y,u_{\e,\phi_{\e}}(s,y))\phi_{\e}(s,y)\d y\d s\\&=: I_{1,\e}+\frac{\alpha}{\delta+1}I_{2,\e}+\beta I_{3,\e}+\sqrt{\e}I_{4,\e}+I_{5,\e}.
		\end{align}
		For the term $I_{1,\e}$, we have $\|I_{1,\e}\|_{\L^p}\leq C\|u^0_\e\|_{\L^p}$. Then, applying Markov's inequality, we deduce 
		\begin{align*}
			\P\big(\|I_{1,\e}\|_{\L^p}>R\big)\leq R^{-1}\E\big[\|I_{1,\e}\|_{\L^p}\big] \leq C R^{-1}(1+\|u^0\|_{\L^p}),
		\end{align*}which implies $I_{1,\e}$ is tight in the space $\C([0,T];\L^p([0,1]))$, for any $p\geq1$.
		
		Let us consider the term $I_{3,\e}$ and estimate it using \eqref{A1}, \eqref{A7}, Minkowski's and Young's inequalities as 
		\begin{align}\label{CE6}	\nonumber
			&\|I_{3,\e}\|_{\L^p}=	\bigg\|\int_0^t\int_0^1G(t-s,\cdot,y)\c(u_{\e,\phi_{\e}}(s,y))\d y\d s\bigg\|_{\L^p} \\&	\nonumber\leq C\int_0^t\left\{(t-s)^{-\frac{\delta}{2p}}\|u_{\e,\phi_\e}(s)\|_{\L^p}^{\delta+1}+\|u_{\e,\phi_\e}(s)\|_{\L^p}+(t-s)^{-\frac{\delta}{p}}\|u_{\e,\phi_\e}(s)\|_{\L^p}^{2\delta+1}\right\}\d s\\& \leq 
			C\left\{T^{1-\frac{\delta}{2p}}\sup_{s\in[0,T]}\|u_{\e,\phi_\e}(s)\|_{\L^p}^{\delta+1}+T\sup_{s\in[0,T]}\|u_{\e,\phi_\e}(s)\|_{\L^p}+T^{1-\frac{\delta}{p}}\sup_{s\in[0,T]}\|u_{\e,\phi_\e}(s)\|_{\L^p}^{2\delta+1}\right\},
		\end{align}for $p\geq2\delta+1$. In view of Lemma \ref{lemma3.6}, we only need to show that the term $\eta_\e =\sup\limits_{s\in[0,T]}\|u_{\e,\phi_{\e}}(s)\|_{\L^p}^{2\delta+1}$ (considering the highest order non-linearity and other terms can be handled in a same manner) is bounded in probability. Using Markov's inequality, similar arguments as in the proofs of Lemma \ref{lemUE} (cf. \eqref{UE1})  and Theorem \ref{thrmE2}, we obtain 
		\begin{align}\label{CE7}
			\lim_{R\to\infty}\sup_{\e\in(0,1]}\P\big(\eta_\e\geq R\big)= \lim_{R\to\infty}\sup_{\e\in(0,1]}\P\bigg(\sup_{t\in[0,T]}\|u_{\e,\phi_\e}(s)\|_{\L^p}^{2\delta+1}\geq R\bigg) 	=0.
		\end{align}Therefore, the term $I_{3,\e}$ is tight in the space  $\C([0,T];\L^p([0,1]))$, for $p\geq2\delta+1$. Similarly, using \eqref{A1}, \eqref{A7}, Minkowski's and Young's inequalities, we estimate the term $I_{2,\e}$ as 
		\begin{align}\label{CE8}
			\|I_{2,\e}\|_{\L^p}\leq C\int_0^t(t-s)^{-\frac{\delta}{2p}-\frac{1}{2}}\|u_{\e,\phi_\e}(s)\|_{\L^p}^{\delta+1}\d s\leq CT^{\frac{1}{2}-\frac{\delta}{2p}}\sup_{s\in[0,T]}\|u_{\e,\phi_\e}(s)\|_{\L^p}^{\delta+1},
		\end{align}for $p\geq \delta+1$. Again, using  similar arguments as in \eqref{CE7}, we conclude that $I_{2,\e}$ is tight in the space  $\C([0,T];\L^p([0,1]))$, for $p\geq\delta+1$. Let us move to the final term $I_{5,\e}$ of the right hand side of \eqref{CE5}, and estimate it using \eqref{A1}, \eqref{A7}, Hypothesis \ref{hyp1} (linear growth of $g$), H\"older's and Minkowski's inequalities as 
		\begin{align}\label{CE9}\nonumber
			\|I_{5,\e}\|_{\L^p} &\leq \bigg(\int_0^t\int_0^1|\phi_\e(s,y)|^2\d y\d s\bigg)^{\frac{1}{2}}\bigg\|\bigg(\int_0^t\int_0^1G^2(t-s,\cdot,y)g^2(s,y,u_{\e,\phi_\e}(s,y))\d y\d s\bigg)^{\frac{1}{2}} \bigg\|_{\L^p} 
			\\&\nonumber\leq C \bigg(\int_0^t\int_0^1|\phi_\e(s,y)|^2\d y\d s\bigg)^{\frac{1}{2}} \bigg\|\int_0^t\int_0^1G^2(t-s,\cdot,y)g^2(s,y,u_{\e,\phi_\e}(s,y))\d y\d s\bigg\|_{\L^\frac{p}{2}}^{\frac{1}{2}}  \\&\leq CT^{\frac{1}{2p}}M^{\frac{1}{2}}K\bigg(1+\sup_{s\in[0,T]}\|u_{\e,\phi_\e}(s)\|_{\L^p}\bigg).
		\end{align}Thus, the tightness of the term $\ I_{5,\e}$ can be obtained in a similar way as we have done for $I_{3,\e}$. Also, the tightness of the penultimate term $I_{4,\e}$ is already established in Lemma \ref{lemE2}. Combining the tightness of all terms, we obtain the tightness of $\{u_{\e,\phi_\e}\}$ in the space $\C([0,T];\L^p([0,1]))$, for $p>\max\{6,2\delta+1\}$.  By Prokhorov's theorem, we are able to find a subsequence of controlled process $\{u_{\e,\phi_\e}\}$ (still denoted by same index)  converges in distribution to some element say $u_{0,\phi}$ in the space $\C([0,T];\L^p([0,1]))$, for $p>\max\{6,2\delta+1\}$. Now, we only need to verify that $u_{0,\phi}$ satisfies \eqref{CE4}, that is, we need to show that as $\e\to 0$
		\begin{align*}
			I_{1,\e} &\Rightarrow\int_0^1G(t,x,y)u^0(y)\d y,\\
			I_{2,\e} &\Rightarrow\int_0^t\int_0^1\frac{\partial G}{\partial y}(t-s,x,y)\p(u_{0,\phi}(s,y))\d y\d s,\\
			I_{3,\e} &\Rightarrow\int_0^t\int_0^1G(t-s,x,y)\c(u_{0,\phi}(s,y))\d y\d s,\\
			I_{5,\e} &\Rightarrow\int_0^t\int_0^1G(t-s,x,y)g(s,y,u_{0,\phi}(s,y))\phi(s,y)\d y\d s,\\
			\sqrt{\e}I_{4,\e} &\Rightarrow0, 
		\end{align*}in  the space $\C([0,T];\L^p([0,1]))$, for $p>\max\{6,2\delta+1\}$. The convergence of the first term $I_{1,\e}$ is trivial, since $u^0_\e\to u^0$ as $\e\to0$. As  $\{u_{\e,\phi_\e}\}$ is tight,  the family of random elements $\left\{\left( u_{\e,\phi_\e},\W^\Q,\phi_\e\right)\right\}$ is also tight. For $\W_\e^\Q:=\W^\Q$, $\e\in(0,1]$,
		using Skorokhod's representation theorem (see Theorem 3, \cite{RMD}), we can construct a new probability space $(\wi{\Omega},\wi{\mathscr{F}},\{\wi{\mathscr{F}}_t\}_{t\geq 0},\wi{\P})$ and a sequence of random variables $\left\{\left(
		\wi{u}_{\e,\wi{\phi}_\e},\wi{\W}^\Q_\e,\wi{\phi}_\e\right)\right\}$ such that 
		\begin{align*}
			\mathscr{L}\left\{	\left(
			\wi{u}_{\e,\wi{\phi}_\e},\wi{\W}_\e^\Q,\wi{\phi}_\e\right)\right\}=\mathscr{L}\left\{\left( u_{\e,\phi_\e},\W^\Q,\phi_\e\right)\right\}, 
		\end{align*}where $\mathscr{L}$ denotes the law, and the sequence of random variables $\left\{\left(	\wi{u}_{\e,\wi{\phi}_\e},\wi{\W}_\e^\Q,\wi{\phi}_\e\right)\right\}$ converges to a random variable $\left( \wi{u}_{0,\wi{\phi}},\wi{\W}^\Q,\wi{\phi}\right)$, $\wi{\P}$-a.s., in the space 
		\begin{align*}
			\C([0,T];\L^p([0,1]))\times \C([0,T]\times[0,1])\times \UU^M, \ \ \text{ for } \ \ p >\max\{6,2\delta+1\}.
		\end{align*}Moreover, $\wi{u}_{\e,\wi{\phi}_\e}$  satisfies the following integral equation:
		\begin{align}\label{CE10}
			\nonumber
			\wi{u}_{\e,\wi{\phi}_\e}(t,x)&=\int_0^1G(t,x,y)u^0_\e(y)\d y+\frac{\alpha}{\delta+1}\int_0^t\int_0^1\frac{\partial G}{\partial y}(t-s,x,y)\p(\wi{u}_{\e,\wi{\phi}_\e}(s,y))\d y\d s\\&\nonumber\quad +\beta\int_0^t\int_0^1G(t-s,x,y)\c(\wi{u}_{\e,\wi{\phi}_\e}(s,y))\d y\d s\\&\nonumber\quad+\sqrt{\e} 
			\int_0^t\int_0^1G(t-s,x,y)g(s,y,\wi{u}_{\e,\wi{\phi}_\e}(s,y))\wi{\W}_\e^\Q(\d s,\d y)\\&\quad+ 
			\int_0^t\int_0^1G(t-s,x,y)g(s,y,\wi{u}_{\e,\wi{\phi}_\e}(s,y))\wi{\phi}_{\e}(s,y)\d y\d s,\ \wi{\P}\text{-a.s.}, 
		\end{align}and we can use the convergences $\wi{u}_{\e,\wi{\phi}_\e}\to 	\wi{u}_{0,\wi{\phi}}$ and $\wi{\phi}_\e\to\wi{\phi}$, $\wi{\P}$-a.s., to prove that 	\begin{align}\label{CE11}
			\nonumber
			\wi{u}_{0,\wi{\phi}}(t,x)&=\int_0^1G(t,x,y)u^0(y)\d y+\frac{\alpha}{\delta+1}\int_0^t\int_0^1\frac{\partial G}{\partial y}(t-s,x,y)\p(\wi{u}_{0,\wi{\phi}}(s,y))\d y\d s\\&\nonumber\quad +\beta\int_0^t\int_0^1G(t-s,x,y)\c(\wi{u}_{0,\wi{\phi}}(s,y))\d y\d s\\&\quad+
			\int_0^t\int_0^1G(t-s,x,y)g(s,y,\wi{u}_{0,\wi{\phi}}(s,y))\wi{\phi}(s,y)\d y\d s.
		\end{align}
		Note that  $\mathscr{L}\{\wi{\phi}\}$ and $\mathscr{L}\{\phi\}$ are the same, since $\mathscr{L}\{\wi{\phi}_\e\}$ and $\mathscr{L}\{\phi_\e\}$ are the same and $\wi{\phi}_\e\to\wi{\phi}$, $\wi{\P}$-a.s. Therefore, by the uniqueness of the solution of the integral equation \eqref{CE11}, we conclude that  $\mathscr{L}\{\wi{u}_{0,\wi{\phi}}\}$ and $\mathscr{L}\{u_{0,\phi}\}$ are the same, which  implies $u_{\e,\phi_\e}\Rightarrow u_{0,\phi}=\mathscr{G}_0\bigg(u^0,\int_0^{\cdot}\int_0^{\cdot}\phi(s,y)\d y\d s\bigg)$. 
		For simplicity, we assume that $u_{\e,\phi_\e} \to u_{0,\phi}$ and $\phi_\e\to\phi,\ \wi{\P}$-a.s., and we need to verify that $u_{0,\phi}=\mathscr{G}_0\big(	u^0,\int_0^{\cdot}\int_0^{\cdot}{\phi}(s,y)\d y \d s\big)$ as $\wi{u}_{0,\wi{\phi}}=\mathscr{G}_0\big(	u^0,\int_0^{\cdot}\int_0^{\cdot}\wi{\phi}(s,y)\d y \d s\big)$.

		We consider the term $I_{2,\e}$, 
		and estimate it using  similar calculations as in the proof of Proposition 3.3, \cite{AKMTM3} (see Step 2), as
		\begin{align}\nonumber
			&\bigg\|I_{2,\e}- \int_0^t\int_0^1\frac{\partial G}{\partial y}(t-s,\cdot,y)\p(u_{0,\phi}(s,y))\d y\d s\bigg\|_{\L^p}\\&\label{313}\leq C\int_0^t(t-s)^{-\frac{1}{2}-\frac{\delta}{2p}}\big(\|u_{\e,\phi}(s)\|_{\L^p}+\|u_{0,\phi}(s)\|_{\L^p}\big)^\delta\|u_{\e,\phi}(s)-u_{0,\phi}(s)\|_{\L^p}\d s\\&\label{CE12}\leq 
			CT^{\frac{1}{2}-\frac{\delta}{2p}} \bigg(\sup_{s\in[0,T]}\big\{\|u_{\e,\phi_\e}(s)\|_{\L^p}+\|u_{0,\phi}(s)\|_{\L^p}\big\}^\delta \bigg)\sup_{s\in[0,T]}\|u_{\e,\phi_\e}(s)-u_{0,\phi}(s)\|_{\L^p},
		\end{align}
		for $p\geq \delta+1$. Since $u_{\e,\phi_\e} \to u_{0,\phi},\ \wi{\P}$-a.s., therefore passing $\e\to0$ in the above inequality, we obtain 
		\begin{align}\label{CE13}
			I_{2,\e} \to \int_0^t\int_0^1\frac{\partial G}{\partial y}(t-s,x,y)\p(u_{0,\phi}(s,y))\d y\d s. 
		\end{align}Now, we consider the term $I_{3,\e}$, and estimate it  using  similar arguments as in the proof of Proposition 3.3, \cite{AKMTM3}, to find 
		\begin{align}\nonumber
			&\bigg\|I_{3,\e}-\int_0^t\int_0^1G(t-s,\cdot,y)\c(u_{0,\phi}(s,y))\d y\d s\bigg\|_{\L^p} \\&\nonumber\leq C\int_0^t\bigg\{(t-s)^{-\frac{\delta}{2p}}(1+\delta)(1+\gamma)\big(\|u_{\e,\phi_\e}(s)\|_{\L^p}+\|u_{0,\phi}(s)\|_{\L^p}\big)^\delta\\&\label{316}\qquad+\gamma+(t-s)^{-\frac{\delta}{p}}(1+2\delta)\big(\|u_{\e,\phi_\e}(s)\|_{\L^p}+\|u_{0,\phi}(s)\|_{\L^p}\big)^{2\delta}\bigg\}\|u_{\e,\phi_\e}(s)-u_{0,\phi}(s)\|_{\L^p}\d s\\&\nonumber \leq C\bigg(T^{1-\frac{\delta}{2p}}(1+\delta)(1+\gamma)\sup_{s\in[0,T]}\big\{\|u_{\e,\phi_\e}(s)\|_{\L^p}+\|u_{0,\phi}(s)\|_{\L^p} \big\}^\delta+ T\gamma\\&\label{CE14}\qquad+T^{1-\frac{\delta}{p}}(1+2\delta)\sup_{s\in[0,T]}\big\{\|u_{\e,\phi_\e}(s)\|_{\L^p}+\|u_{0,\phi}(s)\|_{\L^p} \big\}^{2\delta}\bigg)\sup_{s\in[0,T]}\|u_{\e,\phi_\e}(s)-u_{0,\phi}(s)\|_{\L^p},
		\end{align} for $p\geq 2\delta+1$. Using the convergence $u_{\e,\phi_\e} \to u_{0,\phi},\ \wi{\P}$-a.s., as $\e\to0$ in the above inequality, we conclude that 
		\begin{align}\label{CE15}
			I_{3,\e} \to \int_0^t\int_0^1G(t-s,x,y)\c(u_{0,\phi}(s,y))\d y\d s. 
		\end{align}Now, we consider the term $I_{5,\e}$, and estimate it using the triangle inequality as 
		\begin{align}\label{CE16}\nonumber
			&\bigg\|I_{5,\e}-\int_0^t\int_0^1G(t-s,\cdot,y)g(s,y,u_{0,\phi}(s,y))\phi(s,y)\d y\d s\bigg\|_{\L^p} \\&\nonumber\leq
			C(p)\bigg\|\int_0^t\int_0^1 \big[G(t-s,\cdot,y)\big(g(s,y,u_{\e,\phi_\e}(s,y))-g(s,y,u_{0,\phi}(s,y))\big)\big] \phi_\e(s,y)\d y\d s\bigg\|_{\L^p}\\&\nonumber\quad +C(p)	\bigg\|\int_0^t\int_0^1G(t-s,\cdot,y)g(s,y,u_{0,\phi}(s,y)) \big(\phi_\e(s,y)-\phi(s,y)\big)\d y\d s\bigg\|_{\L^p}  \\&=:I_{5,\e}^1+I_{5,\e}^2.
		\end{align}Using Hypothesis \ref{hyp1}, \eqref{A1}, \eqref{A7}, Cauchy-Schwarz, Minkowski's and Young's inequalities, we estimate the first term term $I_{5,\e}^1$ of the right hand side of the above inequality as 
		\begin{align}\label{CE17}\nonumber
			&\|I_{5,\e}^1\|_{\L^p}\nonumber\\ &\nonumber\leq C(p)M^{\frac{1}{2}}\bigg\|\bigg( \int_0^t\int_0^1G^2(t-s,\cdot,y)\big(g(s,y,u_{\e,\phi_\e}(s,y))-g(s,y,u_{0,\phi}(s,y))\big)^2\d y\d s\bigg)^{\frac{1}{2}}\bigg\|_{\L^p}\\&\nonumber\leq 
			C(p,L) T^{\frac{1}{2p}}M^{\frac{1}{2}} \sup_{s\in[0,T]}\|u_{\e,\phi_\e}(s)-u_{0,\phi}(s)\|_{\L^p} \\&\to 0 \ \ \wi{\P}\text{-a.s.,} \ \ \text{ as } \ \ \e\to0.
		\end{align} 
		Using the convergence $\phi_\e\to\phi$, $\wi{\P}$-a.s, Hypothesis \ref{hyp1}, H\"older's, Minkowski's and Young's inequalities, we estimate the final term $I_{5,\e}^1$ of the right hand side of \eqref{CE16} as 
		\begin{align}\label{CE18}\nonumber
				\|I_{5,\e}^2\|_{\L^p}&\nonumber\leq C(p) \bigg(\int_0^t\int_0^1\big|\phi_\e(s,y)-\phi(s,y)\big|^2\d y\d s\bigg)^\frac{1}{2}\\&\qquad\times\nonumber\bigg\|\bigg(\int_0^t\int_0^1 G^2(t-s,\cdot,y)g^2(s,y,u_{0,\phi}(s,y))\d y\d s\bigg)^\frac{1}{2}  \bigg\|_{\L^p}\\&\nonumber\leq C(p,K) T^{\frac{1}{2p}}\bigg(\int_0^t\int_0^1\big|\phi_\e(s,y)-\phi(s,y)\big|^2\d y\d s\bigg)^\frac{1}{2} \sup_{s\in[0,T]}\big\{1+\|u_{0,\phi}(s)\|_{\L^p}\big\} \\&\to 0 \ \ \wi{\P}\text{-a.s.,} \ \ \text{ as } \ \ \e\to0.
		\end{align}
	Combining \eqref{CE16}-\eqref{CE18}, we obtain 
		\begin{align}\label{CE19}
			I_{5,\e}\to \int_0^t\int_0^1G(t-s,x,y)g(s,y,u_{0,\phi}(s,y))\phi(s,y)\d y\d s.
		\end{align}Finally, the process $I_{4,\e}$ has a modification in the space $\C([0,T]\times[0,1])$ from Lemma \ref{lemE2}, therefore $\sqrt{\e}I_{4,\e}\to0,\ \wi{\P}$-a.s. Therefore, we have established that the stochastic controlled process $u_{\e,\phi}(\cdot)$ (solution of \eqref{CE1}) converges to the solution $u_{0,\phi}(\cdot)$  to the skeleton equation \eqref{CE2} along a subsequence. By the uniqueness,  the result holds for every subsequence $\e_m$. The arbitrariness of $\e_m$ implies 
		\begin{align*}
			u_{\e,\phi_\e}=	\mathscr{G}_\e\bigg(u^0_\e,\sqrt{\e}\W^\Q+\int_0^{\cdot}\int_0^{\cdot}\phi_\e(s,y)\d y\d s\bigg)\Rightarrow u_{0,\phi}=\mathscr{G}_0\left(u^0,\int_0^{\cdot}\int_0^{\cdot}\phi(s,y)\d y\d s\right),
		\end{align*}and Condition \ref{Cond2} (2) is verified. 
		
		For the verification of Condition \ref{Cond2} (1), it is sufficient to combine the well-posedness of the skeleton equation \eqref{4.3} and the convergence of $u_{0,\phi_m}$. Hence, one can conclude the proof of  the theorem..
	\end{proof}  
	
	\section{Uniform large deviation principle} \label{SecULDP}\setcounter{equation}{0}
	In this section, we recall the basic definitions of ULDP (in particular, Freidlin-Wentzell uniform large deviation principle (FWULDP), see \cite{MIFADW}) and equicontinuous uniform Laplace principle (EULP). Later, we provide a sufficient  condition under which  EULP holds. Let us start with the following additional information.
	
	Let $(\EE,d)$ be a Polish space and $\EE_0$ is a set. For any $u^0\in\EE_0$, define a rate function, $\I_{u^0}:\EE\to[0,\infty]$ (cf. \cite{ABPFD,MS,LS}). Let $\Lambda_{u^0}(s_0)$ denote the level set
	\begin{align*} 
		\Lambda_{u^0}(s_0):=\big\{f\in\EE:\I_{u^0}(\psi)\leq s_0\big\},
	\end{align*} 
	for $s_0\in[0,\infty]$. Furthermore, let $\mathscr{T}$ be the collection of subsets of  $\EE_0$ and $a(\e)$ be a function which converges to 0 as $\e\to0$.  
	
	If $h:=h(t,x)$ is a random field and $\EE$ is a function space (will fix later), then $h$-a.s. in $\EE$ means that $h$ has a stochastic version in $\EE$-a.s. For any metric space $(\EE,d)$, the distance between an element $h\in\EE$ and a set $\B\subset \EE$ is defined by 
	\begin{align*}
		\mathrm{dist}_{\EE}(h,\B):=\inf_{g\in\B}d (h,g).
	\end{align*}
	\begin{definition}[FWULDP]\label{def3.1}
		A family of $\EE$-valued random variables $\{X_{\e,u^0}\}_{u^0\in\EE_0,\e\in(0,1]}$ is said to satisfy a \emph{Freidlin-Wentzell uniform large deviation principle} with the speed $a(\e)$ and the rate function $\I_{u^0}(\cdot)$, uniformly over $\mathscr{T}$ if the following hold:
		\begin{enumerate}
			\item For any $\A\in\mathscr{T},\ t_0\geq 0,$ and $\eta>0$
			\begin{align*}
				\liminf_{\e\to0}\inf_{u^0\in\A}\inf_{\psi\in\Lambda_{u^0}(t_0)}\bigg\{a(\e)\log\P\big(d (X_{\e,u^0},\psi)<\eta\big)+\I_{u^0}(\psi)\bigg\}\geq 0.
			\end{align*}\item For any $\A\in\mathscr{T},\ t_0\geq 0,$ and $\eta>0$
			\begin{align*}
				\limsup_{\e\to0}\sup_{u^0\in\A}\sup_{s\in[0,t_0]}\bigg\{a(\e)\log\P\big(\mathrm{dist}_{\EE} (X_{\e,u^0},\Lambda_{u^0}(s))\geq\eta\big)+s\bigg\}\leq 0.
			\end{align*}
		\end{enumerate}	
	\end{definition}
	One should note that FWULDP depends on the choice of the Polish space $\EE$.	Recall that  a family $\B\subset \C_b(\EE)$ of functions from $\EE\to\R$ is equibounded and equicontinuous if 
	\begin{align*}
		\sup_{h\in\B}\sup_{\psi\in\EE}|h(\psi)| <\infty, \ \ \text{ and }\ \ \lim_{\eta\to0}\sup_{h\in\B}\sup_{\substack{f_1,f_2\in\EE,\\d(\psi_1,\psi_2)<\eta}}|h(\psi_1)-h(\psi_2)|=0.
	\end{align*}
	\begin{definition}[EULP]\label{def3.2}
		A family of $\EE$-valued random variables $\{X_{\e,u^0}\}_{u^0\in\EE_0,\e\in(0,1]}$ is said to satisfy an \emph{equicontinuous uniform Laplace principle} with the speed $a(\e)$ and the rate function $\I_{u^0}(\cdot)$, uniformly over $\mathscr{T}$ if for any $\A\in\mathscr{T}$ and any equibounded and equicontinuous family $\B\subset\C_b(\EE)$, 
		\begin{align}\label{3.3}
			\lim_{\e\to0}\sup_{h\in\B}\sup_{u^0\in\A}\bigg|a(\e)\log\E \bigg[\exp\bigg(-\frac{h(X_{\e,u^0})}{a(\e)}\bigg)\bigg]+\inf_{f\in\EE}\big\{h(f)+\I_{u^0}(\psi)\big\}\bigg|=0.
		\end{align}
	\end{definition}
	An equivalence between   ULDP and  EULP is established in \cite{MS}. 
	\begin{theorem}[Theorem 2.10, \cite{MS}]\label{thrm3.3}
		EULP and FWULDP are equivalent.
	\end{theorem}
	Let us provide a sufficient condition for a family of random variables to satisfy  EULP. Then by Theorem \ref{thrm3.3} we obtain FWULDP. For any $M>0$ and $T>0$, let $\PP_2^M$ be the collection of $\mathscr{F}_t$-adapted processes such that $\P(\|\phi\|_{\L^2([0,T]\times[0,1])}\leq M)=1$ (see Section \ref{Sec4} also). 
	
	\begin{condition}\label{cond}
		Assume that there exists a family of measurable maps $\mathscr{G}_\e:\EE_0\times\C([0,T]\times[0,1])\to\EE$, indexed by $\e\in(0,1]$, $u^0\in\EE_0$. Let $\W^\Q(\cdot,\cdot)$ be the noise which is white in time and colored in space and $X_{\e,\phi}:=\mathscr{G}_\e\big(u^0,\sqrt{\e}\W^\Q+\int_0^{\cdot}\int_0^{\cdot}\phi(s,y)\d y\d s\big)$. Define $\mathscr{G}_{0}:\EE_0\times\C([0,T]\times[0,1])\to\EE$, the limiting case of $\mathscr{G}_\e$ as $\e\to0$. Let $\mathscr{T}$ be collection  of the subset of $\EE_0$ such that for any $\B\in \mathscr{T}, \ M>0$ and $\eta>0$, 
		\small{	\begin{align}\label{3.4}
				&	\lim_{\e\to0}\sup_{u^0\in\B}\sup_{\phi\in\PP_2^M}\P\left\{d\bigg(\mathscr{G}_\e\bigg(u^0,\sqrt{\e}\W^\Q+\int_0^{\cdot}\int_0^{\cdot}\phi(s,y)\d y\d s\bigg),\mathscr{G}_{0}\bigg(u^0,\int_0^{\cdot}\int_0^{\cdot}\phi(s,y)\d y\d s\bigg)\bigg)>\eta\right\}\nonumber\\&\hspace{10cm}=0.
		\end{align}}
	\end{condition}

	\begin{theorem}[Theorem 2.13, \cite{MS}]\label{thrm3.4}
		Assume that Condition \ref{cond} holds. Then, the family $X_{\e,u^0}:=\mathscr{G}_\e(u^0,\sqrt{\e}\W^\Q)$ satisfies the EULP uniformly over $\mathscr{T}$ with the speed $a(\e)=\e$, and the rate function $\I_{u^0}:\EE\to\R$ given in \eqref{CE3}. 
	\end{theorem}

	\subsection{A priori bound on the controlled process}
	In this subsection, we introduce the controlled and skeleton systems. For any $\e\in(0,1]$ and the initial data $u^0\in\L^p([0,1])$, for $p>\max\{6,2\delta+1\}$, there exists a measurable mapping $\mathscr{G}_\e:\L^p([0,1])\times\C([0,T]\times[0,1])\to\C([0,T];\L^p([0,1]))$ such that $u_\e:=\mathscr{G}_\e(u^0,\sqrt{\e}\W^\Q)$.

	Let us consider the following controlled SGBH equation:
	\begin{align}\label{3.7}\nonumber
		\frac{\partial u_{\e,\phi}(t,x)}{\partial t}&=\nu\frac{\partial^2 u_{\e,\phi}(t,x)}{\partial x^2}+\frac{\alpha}{\delta+1} \frac{\partial }{\partial x}\p(u_{\e,\phi}(t,x))+\beta \c(u_{\e,\phi}(t,x))+g(t,x,u_{\e,\phi}(t,x))\phi(t,x)\\&\quad +\sqrt{\e}g(t,x,u_{\e,\phi}(t,x))\frac{\partial^2\W^\Q(t,x)}{\partial x\partial t},
	\end{align} 
with initial and boundary conditions given in \eqref{1.2} and $g$ satisfying Hypothesis \ref{hyp1}. 
	
	The global solvability of above controlled SGBH equation can be obtained from Theorem \ref{thrmE2},  with the help of Girsonov's theorem (see Theorem \ref{thrmLDP1}).  
	The process $u_{\e,\phi}(\cdot)$ satisfies the following SCI equation \eqref{CE1}. 
If the noise is absent in the equation \eqref{CE1}, then the equation reduces to the skeleton equation \eqref{CE2} and the solution is denoted by $u_{0,\phi}(\cdot)$. Finally, for any $\psi\in\C([0,T];\L^p([0,1]))$ given in \eqref{CE4}, the   rate function $\I_{u^0}(\psi)$ is defined by \eqref{CE3}. 
	

	\begin{theorem}[Existence and uniqueness of the controlled process]\label{ExistenceofControled}
		Let $\phi\in \PP_{2}^{M},$ for some $M \in \N$. For $\e\geq 0$ and $u^0\in \L^p([0,1])$, for $p> \max\{6, 2\delta+1\}$ define a solution mapping $\mathscr{G}_\e$ such that
		\begin{align*}
			\displaystyle u_{\e,\phi}:=\mathscr{G}_\e\left(u^0,\sqrt{\e}\W^\Q+\int_0^{\cdot}\int_0^{\cdot}\phi(s,y)\d y\d s\right),
		\end{align*}  
		then $u_{\e,\phi}(\cdot)$ is the unique solution of the SCI equation \eqref{CE1}.
	\end{theorem}
	
	\subsection{Main Theorems}
	In this subsection, we prove that the family of solutions $u_{\e}(\cdot)$ of the integral equation \eqref{2.2} satisfies ULDP in two different topologies, $\C([0,T];\L^p([0,1]))$, for $p>\max\{6,2\delta+1\}$ and $\C([0,T]\times [0,1])$.
	
	\begin{theorem}[ULDP in the {$\C([0,T]; \L^p([0,1]))$} topology] \label{LDPinLp} Let $p> \max\{6, 2\delta+1\}$. The family  $\{ u_{\e,\phi}\}$  satisfies a ULDP
		on $\C([0,T];\L^p([0,1]))$ with rate function $\I_{u^0}(\cdot)$ given by \eqref{CE3}, where  the uniformity is over	$\L^p([0,1])$-bounded sets of initial conditions, if the following hold:
		\begin{enumerate}
			\item For any $N>0,\ t_0\geq 0,$ and $\eta>0$
			\begin{align*}
				\liminf_{\e\to0}\inf_{\|u^0\|_{\L^p}\leq N}\inf_{\psi\in\Lambda_{u^0}(t_0)}\bigg\{\e\log\P\big(\|\displaystyle u_{\e,\phi}-\psi\|_{\C([0,T];\L^p([0,1])}<\eta\big)+\I_{u^0}(\psi)\bigg\}\geq 0.
			\end{align*}\item For any  $N>0,\ t_0\geq 0,$ and $\eta>0$
			\begin{align*}
				\limsup_{\e\to0}\sup_{\|u^0\|_{\L^p}\leq N}\sup_{s\in[0,t_0]}\bigg\{\e\log\P\big(\mathrm{dist}_{\C([0,T];\L^p([0,1]))} (u_{\e,u^0},\Lambda_{u^0}(s))\geq\eta\big)+s\bigg\}\leq 0.
			\end{align*}
		\end{enumerate}
	\end{theorem}
	\begin{theorem}[ULDP in the {$\C([0,T]\times[0,1])$} topology]\label{LDPinC}  The family  $\{\displaystyle u_{\e,\phi}\}$  satisfies a ULDP	on $\C([0,T]\times [0,1])$ with rate function $\I_{u^0}(\cdot)$ given by \eqref{CE3}, where  the uniformity is over	$\C([0,1])$-bounded sets of initial conditions, if the following hold:
		\begin{enumerate}
			\item For any $N>0,\ t_0\geq 0,$ and $\eta>0$
			\begin{align*}
				\liminf_{\e\to0}\inf_{\|u^0\|_{ \C([0,1])}\leq N}\inf_{\psi\in\Lambda_{u^0}(t_0)}\bigg\{\e\log\P\big(\|\displaystyle u_{\e,\phi}-\psi\|_{\C([0,T]\times[0,1])}<\eta\big)+\I_{u^0}(\psi)\bigg\}\geq 0.
			\end{align*}\item For any  $N>0,\ t_0\geq 0,$ and $\eta>0$
			\begin{align*}
				\limsup_{\e\to0}\sup_{\|u^0\|_{ \C([0,1])}\leq N}\sup_{s\in[0,t_0]}\bigg\{\e\log\P\big(\mathrm{dist}_{\C([0,T]\times[0,1])} (u_{\e,u^0},\Lambda_{u^0}(s))\geq\eta\big)+s\bigg\}\leq 0.
			\end{align*}
		\end{enumerate}
	\end{theorem}

	\begin{proposition}[Global estimate]\label{PrioriEstimate}
		For any $p>\max\{6, 2\delta+1\}$ there exists a constant $C>0$  depending on $p, M, $ and $T$, such that whenever $\e\in (0,1], u^0\in \L^p([0,1])$ and $\phi\in \PP_2^{M},$
		\begin{align}\label{48}
			\nonumber
			\mathbb{E}\bigg[\sup_{t\in[0,T]}\|u_{\epsilon, \phi}(t)\|_{\L^p}^{p}+&\int_{0}^{T}\||u_{\epsilon, \phi}(s)|^{\frac{p-2}{p}}\partial_x u_{\epsilon, \phi}(s)\|_{\L^2}^{2}ds+\int_{0}^{T}\|u_{\epsilon, \phi}(s)\|_{\L^{p+2\delta}}^{p+2\delta}ds\bigg]\\& \quad \leq C(1+\|u^0\|_{\L^p}^{p}).	\end{align}
	\end{proposition}
	\begin{proof}
		Proof can be obtained in a  similar way as we did for \eqref{UE1} in Lemma \ref{lemUE}.
	\end{proof}
	\begin{corollary}\label{UCPC}
		The random field $\sup\limits_{t\in[0,T]}\|u_{\epsilon, \phi_{\epsilon}}(t)\|_{\L^p}^{p}$, for $p>\max\{6,2\delta+1\}$ is bounded in probability, that is, for any given $T>0$, $M>0$, and $N>0$, 
		\begin{align}\label{3.12}
			\lim_{R\to\infty}\sup_{\|u^0\|_{\L^p}\leq N}\sup_{\phi\in\PP_2^M}\sup_{\e\in(0,1]}\P\bigg[\sup_{t\in[0,T]}\|u_{\epsilon, \phi_{\epsilon}}(t)\|_{\L^p}^{p}>R\bigg]=0.
		\end{align}
	\end{corollary}
	\begin{proof}
		An application of Markov's inequality and  the use of global estimate \eqref{48} leads to the proof of \eqref{3.12}.
	\end{proof}

	\begin{theorem}[Uniform convergence in probability]	\label{UCP}For any $T>0,\ \eta>0,\ N>0$,  $M>0$ and $p>\max\{6,2\delta+1\}$ 
		\begin{align}\label{4.1}
			\lim_{\e\to0}\sup_{\|u^0\|_{\L^p}\leq N}\sup_{\phi\in\PP_2^M}\P\bigg(\|u_{\e,\phi}-u_{0,\phi}\|_{\C([0,T];\L^p([0,1]))}>\eta\bigg)=0.
		\end{align}
	\end{theorem}
	\begin{proof}
		For any $\phi\in\PP_2^M, \ \|u^0\|_{\L^p}\leq N$, for $p>\max \{6,2\delta+1\}$ and $\e\in(0,1]$, we have 
		\begin{align}\label{4.2}\nonumber
			u_{\e,\phi}(t,x)-&u_{0,\phi}(t,x)\\&\nonumber=\frac{\alpha}{\delta+1}\int_0^t\int_0^1\frac{\partial G}{\partial y}(t-s,x,y)\big(\p(u_{\e,\phi}(s,y))-\p(u_{0,\phi}(s,y))\big)\d y\d s\\&\nonumber\quad +\beta\int_0^t\int_0^1G(t-s,x,y)\big(\c(u_{\e,\phi}(s,y))-\c(u_{0,\phi}(s,y))\big)\d y\d s\\&\nonumber\nonumber\quad+
			\int_0^t\int_0^1G(t-s,x,y)\big(g(s,y,u_{\e,\phi}(s,y))-g(s,y,u_{0,\phi}(s,y))\big)\phi(s,y)\d y\d s\\&\nonumber\quad+ 
			\sqrt{\e} 
			\int_0^t\int_0^1G(t-s,x,y)g(s,y,u_{\e,\phi}(s,y))\W^\Q(\d s,\d y)\\&=:\frac{\alpha}{\delta+1}J_\e^1(t,x)+\beta J_\e^2(t,x)+J_\e^3(t,x)+\sqrt{\e}J_\e^4(t,x).
		\end{align}Now, we consider $J_\e^1(\cdot,\cdot)$ and estimate it using similar arguments as in \eqref{313} to find
		\begin{align}\label{4.3}
			\|J_\e^1(t)\|_{\L^p} &\leq C\int_0^t(t-s)^{-\frac{1}{2}-\frac{\delta}{2p}}\big(\|u_{\e,\phi}(s)\|_{\L^p}+\|u_{0,\phi}(s)\|_{\L^p}\big)^\delta\|u_{\e,\phi}(s)-u_{\e,\phi}(s)\|_{\L^p}\d s,
		\end{align}
	for $p\geq \delta+1$.  For any $P>0, \ \e\in(0,1],\ \phi\in\PP_2^M,\ \|u^0\|_{\L^p}\leq N$ and $T>0$, we define 
		\begin{align}\label{4.4}
			\Upsilon_{P}^\e:=\bigg\{\omega\in\Omega:\ \sup_{t\in[0,T]}\|u_{\e,\phi}(t)\|_{\L^p}\leq P, \ \sup_{t\in[0,T]}\|u_{0,\phi}(t)\|_{\L^p}\leq P\bigg\}.
		\end{align}Then, applying H\"older's inequality in \eqref{4.3} and using \eqref{4.4}, we arrive at
		\begin{align}\label{4.5}
			\|J_\e^1(t)\|_{\L^p}^p \chi_{\Upsilon_{P}^\e}\leq C(p,\delta,T,P) \chi_{\Upsilon_{P}^\e}\int_0^t(t-s)^{-\frac{1}{2}-\frac{\delta}{2p}}\|u_{\e,\phi}(s)-u_{\e,\phi}(s)\|_{\L^p}^p\d s,
		\end{align}where $\chi_A$ is the characteristic function of the set $A$.
		Let us consider the term $J_\e^2(\cdot,\cdot)$, and estimate it using similar arguments as in \eqref{316} to obtain
		\begin{align}\label{4.6}\nonumber
			&\|J_\e^2(t)\|_{\L^p} \\&\nonumber\leq C\int_0^t\bigg\{(t-s)^{-\frac{\delta}{2p}}(1+\delta)(1+\gamma)\big(\|u_{\e,\phi}(s)\|_{\L^p}+\|u_{0,\phi}(s)\|_{\L^p}\big)^\delta+\gamma\\&\qquad+(t-s)^{-\frac{\delta}{p}}(1+2\delta)\big(\|u_{\e,\phi}(s)\|_{\L^p}+\|u_{0,\phi}(s)\|_{\L^p}\big)^{2\delta}\bigg\}\|u_{\e,\phi}(s)-u_{0,\phi}(s)\|_{\L^p}\d s,
		\end{align} for $p\geq 2\delta+1$. By \eqref{4.4} and using H\"older's inequality, we have 
		\begin{align}\label{4.7}\nonumber
			\|J_\e^2(t)\|_{\L^p}^p \chi_{\Upsilon_{P}^\e}&\leq C(p,\delta,\gamma,T,P) \chi_{\Upsilon_{P}^\e}\int_0^t\bigg\{(t-s)^{-\frac{\delta}{2p}}(1+\delta)(1+\gamma)+\gamma+(t-s)^{-\frac{\delta}{p}}(1+2\delta)\bigg\}\\&\qquad\times\|u_{\e,\phi}(s)-u_{0,\phi}(s)\|_{\L^p}^p\d s.
		\end{align}
	Using Lemma \ref{lemma3.1}, the fact that $\phi\in\PP_2^M$ and H\"older's inequality, we find 
		\begin{align}\label{4.8}\nonumber
			\|J_\e^3(t)\|_{{\L^p}}^p\chi_{\Upsilon_{P}^\e}&\leq C\chi_{\Upsilon_{P}^\e}\bigg\{\int_0^t(t-s)^{-\frac{1}{2}+\frac{1}{2p}}\|u_{\e,\phi}(s)-u_{0,\phi}(s)\|_{\L^p}\|\phi(s)\|_{\L^2}\d s\bigg\}^p \\&\nonumber\leq C\chi_{\Upsilon_{P}^\e}\bigg(\int_0^t\|\phi(s)\|_{\L^2}^2\d s\bigg)^{\frac{p}{2}}\bigg(\int_0^t(t-s)^{-1+\frac{1}{p}}\|u_{\e,\phi}(s)-u_{0,\phi}(s)\|_{\L^p}^2\d s\bigg)^{\frac{p}{2}}\\&\nonumber\leq CM^{\frac{p}{2}}\chi_{\Upsilon_{P}^\e}\bigg(\int_0^t(t-s)^{\big(-1+\frac{1}{p}\big)\big(\frac{2}{p}+\frac{p-2}{p}\big)}\|u_{\e,\phi}(s)-u_{0,\phi}(s)\|_{\L^p}^2\d s\bigg)^{\frac{p}{2}}\\&\leq C(p)M^{\frac{p}{2}}T^{\frac{p-2}{2p}}\chi_{\Upsilon_{P}^\e}\bigg(\int_0^t(t-s)^{-1+\frac{1}{p}}\|u_{\e,\phi}(s)-u_{0,\phi}(s)\|_{\L^p}^p\d s\bigg).
		\end{align}Combining \eqref{4.2}-\eqref{4.8} and applying Gronwall's inequality, we find
		\begin{align}\label{4.9}\nonumber
			&\sup_{t\in[0,T]}\|u_{\e,\phi}(t)-u_{0,\phi}(t)\|_{\L^p}^p\chi_{\Upsilon_{P}^\e}\\& \nonumber\leq C\e^{\frac{p}{2}}\sup_{t\in[0,T]}\|J_\e^4(t)\|_{\L^p}^p\chi_{\Upsilon_{P}^\e}\\&\nonumber\quad\times\exp\bigg\{C\int_0^T\left[|t-s|^{-\frac{1}{2}-\frac{\delta}{2p}}+|t-s|^{-\frac{\delta}{2p}}+\gamma+|t-s|^{-\frac{\delta}{p}}+|t-s|^{-1+\frac{1}{p}}\right]\d s\bigg\}\\& \leq C(\alpha,\beta,\delta,\gamma,p,T,M,P)\e^{\frac{p}{2}}\sup_{t\in[0,T]}\|J_\e^4(t)\|_{\L^p}^p\chi_{\Upsilon_{P}^\e}.
		\end{align}Let $\vp>0$. Using Proposition \ref{PrioriEstimate}, there exists $P>0$ large enough such that
		\begin{align}\label{4.10}\nonumber
			&	\sup_{\|u^0\|_{\L^p}\leq N}\sup_{\phi\in\PP_2^M}\sup_{\e\in(0,1]}\big[1-\P(\Upsilon_{P}^\e)\big]\\&\nonumber\leq 
			\sup_{\|u^0\|_{\L^p}\leq N}\sup_{\phi\in\PP_2^M}\sup_{\e\in(0,1]}\bigg\{\P\bigg(\sup_{t\in[0,T]}\|u_{\e,\phi}(t)\|_{\L^p}> P\bigg)+\P\bigg( \sup_{t\in[0,T]}\|u_{0,\phi}(t)\|_{\L^p}>P \bigg)\bigg\}\leq \vp. \\&
		\end{align}
	Using \eqref{4.9} and \eqref{4.10}, 
	for any $\e\in(0,1], u^0\in\L^p([0,1])$ and $\phi\in\PP_2^M$, we obtain
		\begin{align}\label{4.11}\nonumber
			&	\sup_{\|u^0\|_{\L^p}\leq N}\sup_{\phi\in\PP_2^M}\P\bigg(\sup_{t\in[0,T]}\|u_{\e,\phi}(t)-u_{0,\phi}(t)\|_{\L^p}>\eta\bigg) \\&\nonumber\leq 	\sup_{\|u^0\|_{\L^p}\leq N}\sup_{\phi\in\PP_2^M}\big[1-\P(\Upsilon_{P}^\e)\big]+\sup_{\|u^0\|_{\L^p}\leq N}\sup_{\phi\in\PP_2^M}\P\bigg(\sup_{t\in[0,T]}\|u_{\e,\phi}(t)-u_{0,\phi}(t)\|_{\L^p}\chi_{\Upsilon_{P}^\e}>\eta\bigg)\\
			& \leq \vp +\sup_{\|u^0\|_{\L^p}\leq N}\sup_{\phi\in\PP_2^M}\P\bigg(\sqrt{\e}\sup_{t\in[0,T]}\|J_\e^4(t)\|_{\L^p}\chi_{\Upsilon_{P}^\e}>\frac{\eta}{C(\alpha,\beta,\delta,\gamma,p,T,M,P)}\bigg).
		\end{align}
		Due to the linear growth of $g$, we have for any $t\in [0,T]$
		\begin{align}\label{417}
			\sup_{\e\in(0,1]}\E \bigg[\sup_{t\in[0,T]}\|g(t,\cdot,u_{\e,\phi}(t, \cdot))\|_{\L^p}^{p}\bigg]\leq C\bigg(1+\sup_{\e\in(0,1]}\E\bigg[\sup_{t\in[0,T]}\|u_{\e,\phi}(t)\|_{\L^p}^{p}\bigg]\bigg)<\infty.
		\end{align}
		Therefore,  Lemma \ref{lemE2} ensures that  for any  $\|u^0\|_{\L^p}\leq N$ and $\phi\in\PP_2^M$ the family
		$$\{J_\e^4(t):\e\in(0,1] \}$$
		is tight in $\C([0,T]\times[0,1])$, for $p>6$. Since we have the embedding $\C([0,T]\times[0,1])\hookrightarrow \C([0,T];\L^p([0,1]))$, the family $\{J_\e^4(\cdot)\}$ is also tight in the space $\C([0,T];\L^p([0,1]))$,  for $p>6$. Now, applying Prokhorov's Theorem, we get that for any sequences $\epsilon_m\rightarrow 0,\phi_m\in\PP_2^M$ and  $\|u^0_m\|_{\L^p}\leq N$, there exists a subsequence  (relabeled $(\e_m,\phi_m,u^0_m )$) such that $J_{\e_m}^4(t)$ converges in distribution. Now, if we multiply by the sequence $\sqrt{\e_m}\downarrow0$, then it is obvious that
		$$\sqrt{\e_m} \sup_{0\leq t\leq T}\|J_{\e_m}^4(t)\|_{\L^p}\chi_{\Upsilon_{P}^{\e_m}} \Rightarrow 0\ \ \text{  as   }\ \ m\rightarrow \infty.$$
		Therefore, the original sequence
		$\sqrt{\e_m} \sup\limits_{0\leq t\leq T}\|J_{\e_m}^4(t)\|_{\L^p}\chi_{\Upsilon_{P}^{\e_m}} \Rightarrow0$ as $m\to\infty$. Since the limit is zero which is a constant,  therefore it converges in probability also, that is,
		$$\lim_{\e\to0}\P\bigg(\sqrt{\e}\sup_{t\in[0,T]}\|J_\e^4(t)\|_{\L^p}\chi_{\Upsilon_{P}^\e}>\frac{\eta}{C(\alpha,\beta,\delta,\gamma,p,T,M,P)}\bigg)=0.$$
		Now, due to arbitrary choice of $\e_m\downarrow 0,\phi_m\in \PP_2^M,$ and $\|u^0_m\|_{\L^p}\leq N $, we conclude that 
		\begin{align}\label{4.12}
			\lim_{\e\to0} \sup_{\|u^0\|_{\L^p}\leq N} \sup_{\phi\in\PP_2^M}\P \bigg(\sup_{t\in[0,T]}\|u_{\e,\phi}(t)-u_{0,\phi}(t)\|_{\L^p}>\eta\bigg)\leq \varpi.
		\end{align}
		Since the choice of $\vp$ is arbitrary, therefore \eqref{4.1} holds.
	\end{proof}
	
	\begin{proof}[Proof of Theorem \ref{LDPinLp}] Combining  Theorems \ref{ExistenceofControled}  and \ref{UCP}, we obtain the required result.
	\end{proof}
	
		Let us now move to the proof of Theorem \ref{LDPinC}, that is, the ULDP in the {$\C([0,T]\times[0,1])$} topology.
		\begin{lemma}\label{TightnessInC}
			For any $\|u^0\|_{\C([0,1])}\leq N,$ the family
			$$\left\{u_{\e,\phi}-u_{0,\phi}: \phi\in\PP_2^M, \e\in (0,1] \right\}$$
			is tight in $\C([0,T]\times[0,1])$.
		\end{lemma}
		\begin{proof}
			Let us recall \eqref{4.2}, that is, 
		\begin{align*}	u_{\e,\phi}(t,x)-u_{0,\phi}(t,x)=\frac{\alpha}{\delta+1}J_\e^1(t,x)+\beta J_\e^2(t,x)+J_\e^3(t,x)+\sqrt{\e}J_\e^4(t,x).
			\end{align*}
			Using similar calculations as in Lemma \ref{lemma3.1} (1), Taylor's formula and the embedding \\$\L^q([0,1])\hookrightarrow \L^p([0,1])$, for $q\geq p$, we obtain 
			\begin{align}\label{J1}
				\|J_\e^1(t)\|_{\L^p}& \leq C\int_0^t(t-s)^{-\frac{1}{2}+\frac{1}{2p}}\big(\|u_{\e,\phi}(s)\|_{\L^q}+\|u_{0,\phi}(s)\|_{\L^q}\big)^{\delta}\big(\|u_{\e,\phi}(s)\|_{\L^q}+\|u_{0,\phi}(s)\|_{\L^q}\big)\d s,
			\end{align}for $q\in[\delta+1,p]$.
			Using similar arguments to the above inequality, we also have 
			\begin{align}\label{J2}\nonumber
				&	\|J_\e^2(t)\|_{\L^p}\\ &\nonumber\leq C\int_0^t(t-s)^{-\frac{1}{p}+\frac{1}{2p}}\bigg\{(1+\delta)(1+\gamma)\big(\|u_{\e,\phi}(s)\|_{\L^q}+\|u_{0,\phi}(s)\|_{\L^q}\big)^\delta\\&\quad+\gamma+(1+2\delta)\big(\|u_{\e,\phi}(s)\|_{\L^q}+\|u_{0,\phi}(s)\|_{\L^q}\big)^{2\delta}\bigg\}\big(\|u_{\e,\phi}(s)\|_{\L^q}+\|u_{0,\phi}(s)\|_{\L^q}\big)\d s,
			\end{align} for $q\in[2\delta+1,p]$.
			By Corollary \ref{UCPC}, we obtain that the terms appearing in the left hand side of the estimates (\ref{J1}) and (\ref{J2})
			are uniformly bounded in probability over bounded subsets of initial conditions in $\C([0,1])$  and  for all $\e\in (0,1]$. Therefore, by Lemma \ref{lemma3.6},  for all $\phi\in \PP_2^M$ and $\e\in (0,1]$, the family
			$$\{J_\e^i(\cdot): \|u^0\|_{\C([0,1])}<R\}$$
			are tight in $\C([0,T]\times[0,1])$ for $i=1,2.$
			Using \eqref{417},   Lemma \ref{lemE2} guarantees  that  for any  $\|u^0\|_{\C([0,1])}\leq N$ and $\phi\in\PP_2^M$ the family
			$$\{J_\e^4(t):\e\in(0,1] \}$$
			is tight in $\C([0,T]\times[0,1])$.
	
		The penultimate term $J_\e^3(\cdot)$ in the right hand side of \eqref{4.2} is not covered by Lemma \ref{lemma3.6}. Using Lemma 3.1 (iv) and Proposition 3.5, \cite{IGCR}, for $\rho=\infty,\ \gamma=2, \ q=2, \ p=2,\ \kappa=3/4, \ |\delta|=0$ and $\ell_1\in(0,1/2)$, we find
		\begin{align*}
			|J_\e^3(t,x)-J_\e^3(t,x+z)| & \leq C |z|^{\ell_1} \bigg(\int_0^t\big\|\big(u_{\e,\phi}(s)-u_{0,\phi}(s)\big)\phi(s)\big\|_{\L^2}^2\d s\bigg)^\frac{1}{2} \\&\leq C |z|^{\ell_1}\sup_{s\in[0,t]}\|u_{\e,\phi}(s)-u_{0,\phi}(s)\|_{\L^\infty}\bigg(\int_0^t\|\phi(s)\|_{\L^2}^2\d s\bigg)^{\frac{1}{2}}\\&\leq C(M)|z|^{\ell_1} \sup_{s\in[0,T]}\|u_{\e,\phi}(s)-u_{0,\phi}(s)\|_{\L^\infty}.
		\end{align*} 
	Similarly, H\"older's continuity in time $t$ follows from Lemma 3.1 (iv) and Proposition 3.5, \cite{IGCR}, for $\rho=\infty,\ \gamma=2, \ q=2, \ p=2,\ \kappa=3/4, \ |\delta|=0$ and $\ell_2\in(0,1/4)$, since 
		\begin{align*}
			\|J_\e^3(t)-J_\e^3(s)\|_{\L^\infty} &\leq C|t-s|^{\ell_2}\bigg(\int_0^t\big\|\big(u_{\e,\phi}(s)-u_{0,\phi}(s)\big)\phi(s)\big\|_{\L^2}^2\d s\bigg)^\frac{1}{2} \\&\leq C |t-s|^{\ell_2} \sup_{s\in[0,t]}\|u_{\e,\phi}(s)-u_{0,\phi}(s)\|_{\L^\infty}\bigg(\int_0^t\|\phi(s)\|_{\L^2}^2\d s\bigg)^{\frac{1}{2}}\\&\leq C(M)|t-s|^{\ell_2} \sup_{s\in[0,T]}\|u_{\e,\phi}(s)-u_{0,\phi}(s)\|_{\L^\infty}.
		\end{align*}By the above estimates, we obtain that $J_\e^3(\cdot)$ is H\"older continuous in both time and space and the H\"older norm is also uniformly bounded over $\e\in(0,1], \ \phi\in\PP_2^M$ and $\|u^0\|_{\C([0,1])} \leq N$. Thus, the family $\{J_\e^3(\cdot): \|u^0\|_{\C([0,1])}\leq N, \ \e\in(0,1], \ \phi\in\PP_2^M\}$ is  also tight in the space $\C([0,T]\times[0,1])$.
			\end{proof}	
		
		\begin{theorem}[Convergence in probability in supremum norm]\label{ContinuousCase}
			For any $T>0,\ \eta>0,\ N>0$, and $M>0$, 
			\begin{align}\label{4.1C}
				\lim_{\e\to0}\sup_{\|u^0\|_{\C([0,1])}\leq N}\sup_{\phi\in\PP_2^M}\P\bigg(\|u_{\e,\phi}-u_{0,\phi}\|_{\C([0,T]\times[0,1])}>\eta\bigg)=0.
			\end{align}
		\end{theorem}
		\begin{proof}In Lemma \ref{TightnessInC}, we have shown that 
			for any $\|u^0\|_{\C([0,1])}\leq N,$ the family
			$$\left\{u_{\e,\phi}-u_{0,\phi}: \phi\in\PP_2^M, \e\in (0,1] \right\},$$
			is tight in $\C([0,T]\times[0,1])$. Therefore, by applying Prokhorov's Theorem, we get that for any sequences $\epsilon_m\rightarrow 0,\phi_m\in\PP_2^M$ and  $\|u^0_m\|_{\L^p}\leq N$, there exists a subsequence  (relabeled $(\e_m,\phi_m,u^0_m )$) such that $u_{\e_m,\phi_m}-u_{0,\phi}$ converges in distribution. Let us denote the limit function  by $u^*$. Since we have the embedding $\C([0,T]\times[0,1])\hookrightarrow \C([0,T], \L^p([0,1])),$ the above convergence  implies $u_{\e_m,\phi_m}-u_{0,\phi}\Rightarrow u^*$ in the $\C([0,T];\L^p([0,1]))$ topology also. Therefore, Theorem \ref{UCP} yields 
			$$\sup_{t\in[0,T]}\|u^*(t)\|_{\L^p}^p=0$$
			with probability $1$  and it implies that $\|u^*\|_{\C([0,T]\times[0,1])}=0$ with probability one. Since $\|u_{\e,\phi}-u_{0,\phi}\|_{\C([0,T]\times[0,1])}$ is converging to zero which is a constant (non-random), therefore
			$$\lim_{m\rightarrow \infty}\P\left(\|u_{\e_m,\phi_m}-u_{0,\phi}\|_{\C([0,T]\times[0,1])}>\eta\right)=0.$$
			Due to the arbitrariness of  $\e_m\downarrow 0, \|u^0_m\|_{\C([0,1])}\leq R$ and $\phi_m\in \PP_2^M$, we obtain the required result \eqref{4.1C}.	
		\end{proof}
		\begin{proof}[Proof of Theorem \ref{LDPinC}]: Combining Theorems \ref{ExistenceofControled} and \ref{ContinuousCase} yield the required result.
		\end{proof}

		\section{Bounded Non-linearity  with Space-time White Noise}\label{ULDPBDD}
		In this section, we define the noise as the derivatives of  Brownian sheet $\frac{\partial^2\W(\cdot,\cdot)}{\partial t\partial x}:=\mathcal{W}(\cdot,\cdot)$ on the probability space $(\Omega,\mathscr{F},\{\mathscr{F}_t\}_{t\geq 0},\P)$, that is, space-time white noise (see Subsection \ref{SN}). Note that the growth condition of the coefficient $g$ was linear in the previous sections, but we assume that the coefficient $g$ is of bounded nonlinearity in this section. Moreover, we prove our main results of this section for the case $p\geq  2\delta+1$. Let us start  by  introducing the assumptions on the noise coefficient $g$:
		\begin{hypothesis}\label{hyp2}
			The function $g:[0,T]\times[0,1]\times\R\to\R$ is a measurable function  satisfying 
			\begin{align}\label{AP1}
				|g(t,x,r)|\leq K,\ \ \text{ and } \ \ |g(t,x,r)-g(t,x,s)|\leq L|r-s|,			
			\end{align}for all $t\in[0,T],\ x\in[0,1]$ and $r,s,\in\R$, where $K,L$ are some positive constants.
		\end{hypothesis}
		For the basic definitions of EULP and ULDP, we refer Section \ref{SecULDP}. Let us comment on the solvability under the new Hypothesis \ref{hyp2}. In this part we are not using  the global solvability result obtained in Section \ref{Sec3}, since the noise is space-time white noise.  For this part, we recall the following solvability result from \cite{AKMTM3} under Hypothesis \ref{hyp2}.
		\begin{theorem}[Theorem 3.5, \cite{AKMTM3}]\label{thrm2.3R}
			Let us assume that the initial data $u^0\in\L^p([0,1])$, for $p\geq 2\delta+1$. Then, there exists a unique $\L^p([0,1])$-valued $\mathscr{F}_t$-adapted continuous process $u_\e$, for $\e\in(0,1]$ satisfying \eqref{2.2} such that 
			\begin{align}\label{AP2}
				\E\bigg[\sup_{t\in[0,T]}\|u_\e(t)\|_{\L^p}^p\bigg]\leq C(T).
			\end{align}
		\end{theorem}Our aim is to prove ULDP under the new Hypothesis \ref{hyp2} on the noise coefficient.

		Let us define $\displaystyle\zeta_{\e,\phi}(t,x):= \sqrt{\e} 
		\int_0^t\int_0^1G(t-s,x,y)g(s,y,u_{\e,\phi}(s,y))\W(\d s,\d y)$, and set 
		\begin{align}\label{AP3}
			z_{\e,\phi}:=u_{\e,\phi}-\zeta_{\e,\phi}.
		\end{align}In order to prove that $u_{\e,\phi}(\cdot)$ is bounded in probability uniformly in $\e,u^0$ and $\phi$, it is enough to show that both $z_{\e,\phi}(\cdot)$ and $\zeta_{\e,\phi}(\cdot)$ are bounded in probability uniformly in $\e,u^0$ and $\phi$. We start with a result which  gives the required bound for $\zeta_{\e,\phi}(\cdot)$.
		
		Using Hypothesis \ref{hyp2}, Corollary 4.3, \cite{IG} for every $p\geq1,\ T>0$, there exists  a constant $C$ such that the following holds:
		\begin{align*}
			\E\big[|\zeta(t,x)-\zeta(s,y)|^{2p}\big]\leq C\big(|t-s|^\frac{1}{4}+|x-y|^\frac{1}{2}\big)^{2p}, \ \ \text{ for all } \ \ s,t\in[0,T], \ \ x,y\in[0,1].
		\end{align*}Hence by Kolmogorov's continuity  theorem (see \cite{AGERHR,JBW}), we have 
		\begin{align}\label{AP4}
			\sup_{\|u^0\|_{\L^p}\leq  N}\sup_{\e\in(0,1]}\sup_{\phi\in\PP_2^M}\E\bigg[\sup_{(t,x)\in[0,T]\times[0,1]}|\zeta_{\e,\phi}(t,x)|^p\bigg]<\infty.
		\end{align}Define 
		\begin{align}\label{AP5}
			\zeta^*_{\e,\phi}:=\sup_{(t,x)\in[0,T]\times[0,1]}|\zeta_{\e,\phi}(t,x)|,
		\end{align}then, applying Markov's inequality yields $\zeta^*_{\e,\phi}$ is bounded in probability uniformly in $\e,u^0$ and $\phi$. Let us now prove that $z_{\e,\phi}(\cdot)$ is bounded in  probability uniformly in $\e, u^0$ and $\phi$.
	
		\begin{lemma}\label{lem3.5}
			The random field $\sup\limits_{t\in[0,T]}\|z_{\e,\phi}(t)\|_{\L^p}$ is bounded in probability, that is, for any given $T>0$, $M>0$, and $N>0$, 
			\begin{align}\label{AP6}
				\lim_{R\to\infty}\sup_{\|u^0\|_{\L^p}\leq N}\sup_{\phi\in\PP_2^M}\sup_{\e\in(0,1]}\P\bigg[\sup_{t\in[0,T]}\|z_{\e,\phi}(t)\|_{\L^p}>R\bigg]=0,
			\end{align}
		for all $p\geq 2\delta+1$. 
		\end{lemma}
		
		Before embarking to the proof, let us provide a useful result which will be used in the proof. 
			\begin{remark}
			Let us take $u\in\C_0^{\infty}([0,1])$. Then $|u|^{p/2}\in\H_0^1([0,1])$, for all $p\in[2,\infty]$, and from the Sobolev embedding, we have $\H_0^1([0,1])\hookrightarrow\L^p([0,1])$, for all $p\in[2,\infty]$, we obtain 
			\begin{align*}
				\|u\|^{q}_{\L^{pq}}=\||u|^{\frac{q}{2}}\|_{\L^{2p}}^2\leq C\int_{0}^1|\partial_x|u(x)|^{\frac{q}{2}}|^2\d x.
			\end{align*}
			But it can be easily seen that $|\partial_x|u|^{\frac{q}{2}}|^2\leq C_q|u|^{q-2}|\partial_xu|^2$. Therefore we  have 
			\begin{align*}
				\|u\|^{q}_{\L^{pq}}\leq C_q\int_{0}^1|\partial_xu(x)|^2|u(x)|^{q-2}\d x,
			\end{align*}
			for any $p\in[2,\infty]$.  In particular, we have 
			\begin{align}\label{528}
				\int_0^t\|u(s)\|_{\L^{\infty}}^q\d s\leq C_q\int_0^t\||u(s)|^{\frac{q-2}{2}}\partial_xu(s)\|_{\L^2}^2\d s,
			\end{align}
			for all $t\in[0,T]$. 
		\end{remark}
		
		\begin{proof}[Proof of Lemma \ref{lem3.5}] Note that the process $z_{\e,\phi}(\cdot)$ satisfies the following system:
			\begin {equation}\label{AP7}
			\left\{
			\begin{aligned}
				\frac{\partial z_{\e,\phi}}{\partial t}(t,x)&=\nu\frac{\partial^2 z_{\e,\phi}}{\partial x^2}(t,x)+\beta\c((z_{\e,\phi}+\zeta_{\e,\phi})(t,x))-\frac{\alpha}{\delta+1}\frac{\partial}{\partial x} \p((z_{\e,\phi}+\zeta_{\e,\phi})(t,x))\\&\quad +g(t,x,(z_{\e,\phi}+\zeta_{\e,\phi})(t,x))\phi(t,x),\\
				z_{\e,\phi}(t,0)&=z_{\e,\phi}(t,1)=0, \ \ t\in(0,T],\\
				z_{\e,\phi}(0,x)&=u^0(x), \ \ x\in[0,1].
			\end{aligned}
			\right.
		\end{equation}
			One can follow similar arguments in the proof of  Proposition 3.4, \cite{AKMTM3} to obtain the $\L^p$-energy estimate. In fact, the term  \begin{align}\label{529}I(t):=\int_0^t\big(g(s,(z_{\e,\phi}+\zeta_{\e,\phi})(s))\phi(s),|z_{\e,\phi}(s)|^{q-2}z(s)\big)\d s,\end{align}while taking the inner product with $|z_{\e,\phi}|^{q-2}z_{\e,\phi}$ in  \eqref{AP7}, can be handled in the following way: 
				\begin{align}\label{530}\nonumber
			|	I(t)|&\leq K\int_0^t\|\phi(s)\|_{\L^2}\|z_{\e,\phi}(s)\|_{\L^{2(q-1)}}^{q-1}\d s \nonumber\\&\leq K\int_0^t\|\phi(s)\|_{\L^2}\|z_{\e,\phi}(s)\|_{\L^q}^{\frac{q(2\delta+2-q)}{4\delta}}\|z_{\e,\phi}(s)\|_{\L^{q+2\delta}}^{\frac{(q-2)(2\delta+q)}{4\delta}}\d s\nonumber\\&\nonumber\leq \frac{p\beta}{16}\int_0^t\|z_{\e,\phi}(s)\|_{\L^{q+2\delta}}^{q+2\delta}\d s+C\int_0^t \|\phi(s)\|_{\L^2}^{\frac{4\delta}{4\delta-q+2}}\|z_{\e,\phi}(s)\|_{\L^q}^{\frac{q(2\delta+2-q)}{4\delta-q+2}}\d s\\&\leq  \frac{p\beta}{16}\int_0^t\|z_{\e,\phi}(s)\|_{\L^{q+2\delta}}^{q+2\delta}\d s+C\int_0^t\|\phi(s)\|_{\L^2}^2\d s+C\int_0^t\|z_{\e,\phi}(s)\|_{\L^q}^q\d s,
			\end{align}
	provided $2\leq q\leq 2\delta+2$, 	where we have used the uniform bound for $g$, interpolation, H\"older's and Young's inequalities. Now taking the inner product with $|z_{\e,\phi}|^{q-2}z_{\e,\phi}$ and performing calculations similar to the proof of  Proposition 3.4, \cite{AKMTM3} and \eqref{530}, we deduce for all $t\in[0,T]$ and $q\in[2,2\delta+2]$
		 \begin{align}\label{3.13}\nonumber
			&\|z_{\e,\phi}(t)\|_{\L^q}^q+\frac{\nu q(q-1)}{2}\int_0^t\||z_{\e,\phi}(s)|^{\frac{q-2}{2}}\partial_x z_{\e,\phi}(s)\|_{\L^2}^2\d s+\beta p\gamma\int_0^t\|z_{\e,\phi}(s)\|_{\L^q}^q\d s\nonumber\\&\nonumber\quad+\frac{\beta q}{8}\int_0^t\|z_{\e,\phi}(s)\|_{\L^{q+2\delta}}^{q+2\delta}\d s \\& \leq e^{CT}\bigg\{ \|u^0\|_{\L^q}^q+CT+CT\sup_{t\in[0,T]}\|\zeta_{\e,\phi}(t)\|_{\L^{q(\delta+1)}}^{q(\delta+1)}+CT\sup_{t\in[0,T]}\|\zeta_{\e,\phi}(t)\|_{\L^{q+2\delta}}^{q+2\delta}\nonumber\\&\qquad+C\int_0^T\|\phi(t)\|_{\L^2}^2\d t\bigg\}=:\mathcal{K}(u^0,\phi,\zeta^*_{\e,\phi},T),
		\end{align}
	where $C$ is a positive constant. 

	From \eqref{528}, one can easily deduce that 
	\begin{align}\label{533}
	\sup_{s\in[0,T]}\|z_{\e,\phi}(t)\|_{\L^q}^q+	\int_0^T\|z_{\e,\phi}(t)\|_{\L^{\infty}}^q\d t\leq \mathcal{K}(u^0,\phi,\zeta^*_{\e,\phi},T),
	\end{align}
for all $q\in[2,2\delta+2].$ 

Let us now consider the case $q=2\delta+3$. We just need to estimate the term $I(t)$ given in \eqref{529} in a proper way. Using H\"older's and Young's inequalities and \eqref{533}, we estimate 
\begin{align}
	|I(t)|&\leq K\int_0^t\|\phi(s)\|_{\L^2}\||z_{\e,\phi}(s)|^{2\delta+1}z_{\e,\phi}(s)\|_{\L^{2}}\d s \nonumber\\&\leq K\int_0^t\|\phi(s)\|_{\L^2}\|z_{\e,\phi}(s)\|_{\L^{\infty}}^{\delta+1}\|z_{\e,\phi}(s)\|_{\L^{2(\delta+1)}}^{\delta+1}\d s\nonumber\\&\leq \int_0^t\|z_{\e,\phi}(s)\|_{\L^{\infty}}^{2(\delta+1)}\d s+C\sup_{s\in[0,t]}\|z_{\e,\phi}(s)\|_{\L^{2(\delta+1)}}^{2(\delta+1)}\int_0^t\|\phi(s)\|_{\L^2}^2\d s,
\end{align}
for all $t\in[0,T]$. Therefore, \eqref{3.13} holds true for all $q\in[2,2\delta+3]$, and a similar calculation yields for $q=2\delta+4$. Continuing like this, one can show that \eqref{3.13}  holds true for all $q\in[2,\infty)$.

	Since, $\zeta^*_{\e,\phi}(\cdot)$ is bounded in probability, an application of Markov's inequality results that $\sup\limits_{t\in[0,T]} 	\|z_{\e,\phi}(t)\|_{\L^q}$, for $q\in[2,\infty)$ is also uniformly bounded in probability.
	\end{proof}
	Combining \eqref{AP4}, \eqref{AP5} and Lemma \ref{lem3.5}, one can establish Corollary \ref{UCPC}, that is, the random field $\sup\limits_{t\in[0,T]}\|u_{\e,\phi}(t)\|_{\L^p}$ is bounded in probability for $p\geq 2\delta+1$.
	
	In order to  prove our main Theorem \ref{LDPinLp} under Hypothesis \ref{hyp2}, it is sufficient to verify  Condition \ref{cond}. Now, for the verification of Condition \ref{cond}, we need to prove Theorem \ref{UCP} under Hypothesis \ref{hyp2}, that is, the uniform convergence in  probability.
	
	\begin{proof}[Proof of Theorem \ref{UCP}]
		For any $\phi\in\PP_2^M, \ \|u^0\|_{\L^p}\leq N$, for $p\geq 2\delta+1$ and $\e\in(0,1]$, we have 
		\begin{align}\label{AP14}\nonumber
		&	u_{\e,\phi}(t,x)-u_{0,\phi}(t,x)\\&\nonumber=\frac{\alpha}{\delta+1}\int_0^t\int_0^1\frac{\partial G}{\partial y}(t-s,x,y)\big(\p(u_{\e,\phi}(s,y))-\p(u_{0,\phi}(s,y))\big)\d y\d s\\&\nonumber\quad +\beta\int_0^t\int_0^1G(t-s,x,y)\big(\c(u_{\e,\phi}(s,y))-\c(u_{0,\phi}(s,y))\big)\d y\d s\\&\nonumber\nonumber\quad+
			\int_0^t\int_0^1G(t-s,x,y)\big(g(s,y,u_{\e,\phi}(s,y))-g(s,y,u_{0,\phi}(s,y))\big)\phi(s,y)\d y\d s\\&\nonumber\quad+ 
			\sqrt{\e} 
			\int_0^t\int_0^1G(t-s,x,y)g(s,y,u_{\e,\phi}(s,y))\W(\d s,\d y)\\&=\frac{\alpha}{\delta+1}J_\e^1(t,x)+\beta J_\e^2(t,x)+J_\e^3(t,x)+\sqrt{\e}J_\e^4(t,x).
		\end{align}Now, we consider $J_\e^1(\cdot,\cdot)$.  From \eqref{4.3}-\eqref{4.5}, we have
		\begin{align}\label{AP17}
			\|J_\e^1(t)\|_{\L^p}^p \chi_{\Upsilon_{P}^\e}\leq C \chi_{\Upsilon_{P}^\e}\int_0^t(t-s)^{-\frac{1}{2}-\frac{\delta}{2p}}\|u_{\e,\phi}(s)-u_{\e,\phi}(s)\|_{\L^p}^p\d s,
		\end{align}where $\chi_A$ is the characteristic function of the set $A$.
		Using \eqref{4.4}, \eqref{4.6} and \eqref{4.7}, we obtain
		\begin{align}\label{AP19}\nonumber
			&	\|J_\e^2(t)\|_{\L^p}^p \chi_{\Upsilon_{P}^\e}\\&\leq C \chi_{\Upsilon_{P}^\e}\int_0^t\bigg\{(t-s)^{-\frac{\delta}{2p}}(1+\delta)(1+\gamma)+\gamma+(t-s)^{-\frac{\delta}{p}}(1+2\delta)\bigg\}\|u_{\e,\phi}(s)-u_{0,\phi}(s)\|_{\L^p}^p\d s,
		\end{align}
	where we have used H\"older's inequality. Using Lemma 3.4, \cite{IGDN}, Hypothesis \ref{hyp2} (bound of $g$) and \eqref{4.4}, we find 
		\begin{align}\label{AP20}\nonumber
			\E\bigg[\sup_{t\in[0,T]}\|J_\e^3(t)\|_{\L^p}^p\bigg]&\leq \E\bigg[\sup_{t\in[0,T]}\bigg\|\int_0^t\int_0^1G(t-s,\cdot,y)g(s,y,u_{\e,\phi}(s,y))\W(\d s,\d y)\bigg\|_{\L^p}^p\bigg]\\&\leq C(p,T,K), 
		\end{align}for $p>2$. Combining \eqref{AP14}-\eqref{AP20} and applying Fubini's theorem followed by Gronwall's inequality, we arrive at
		\begin{align}\label{AP21}\nonumber
			&\E\bigg[\sup_{t\in[0,T]}\|u_{\e,\phi}(t)-u_{0,\phi}(t)\|_{\L^p}^p\chi_{\Upsilon_{P}^\e}\bigg]\\& \nonumber\leq C\e^{\frac{p}{2}}\exp\bigg\{C\int_0^T\left[|t-s|^{-\frac{1}{2}-\frac{\delta}{2p}}+|t-s|^{-\frac{\delta}{2p}}+\gamma+|t-s|^{-\frac{\delta}{p}}\right]\d s\bigg\}\\&\leq C(\alpha,\beta,\delta,\gamma,p,T,M,K,P)\e^{\frac{p}{2}}.
		\end{align}For any $\vp>0$. Using similar calculations as in \eqref{4.10} and Theorem \ref{thrm2.3R}, there exists $P>0$ large enough such that
		\begin{align}\label{AP22}
			&	\sup_{\|u^0\|_{\L^p}\leq N}\sup_{\phi\in\PP_2^M}\sup_{\e\in(0,1]}\big[1-\P(\Upsilon_{P}^\e)\big] \leq \vp.
		\end{align}Using \eqref{AP21}, \eqref{AP22}, and Markov's inequality, for any $\e\in(0,1], u^0\in\L^p([0,1])$ and $\phi\in\PP_2^M$,
		\begin{align}\label{AP23}\nonumber
			&	\sup_{\|u^0\|_{\L^p}\leq N}\sup_{\phi\in\PP_2^M}\P\bigg(\sup_{t\in[0,T]}\|u_{\e,\phi}(t)-u_{0,\phi}(t)\|_{\L^p}>\eta\bigg) \\&\nonumber\leq 	\sup_{\|u^0\|_{\L^p}\leq N}\sup_{\phi\in\PP_2^M}\big[1-\P(\Upsilon_{P}^\e)\big]+\sup_{\|u^0\|_{\L^p}\leq N}\sup_{\phi\in\PP_2^M}\P\bigg(\sup_{t\in[0,T]}\|u_{\e,\phi}(t)-u_{0,\phi}(t)\|_{\L^p}\chi_{\Upsilon_{P}^\e}>\eta\bigg)\\& \leq \vp +\frac{C(\alpha,\beta,\delta,\gamma,p,T,M,P)\e^{\frac{p}{2}}}{\eta^p}. 
		\end{align}Due to arbitrary choice of $\vp$, we obtain the required result \eqref{4.1}.
	\end{proof}
	\begin{proof}[Proof of Theorem \ref{LDPinLp}]
		We established Theorem \ref{UCP} under Hypothesis \ref{hyp2}, which validates Condition \ref{cond}. Therefore,  EULP holds and by Theorem \ref{thrm3.3}, we conclude that ULDP remains valid. 
	\end{proof}
	
		\begin{proof}[Proof of Theorem \ref{LDPinC}]: One can obtain the proof of Lemma \ref{TightnessInC} and Theorem \ref{ContinuousCase} by repeating the same calculations, since the Lipschitz condition of the noise coefficient $g$ is same in both Hypothesis \ref{hyp1} and \ref{hyp2}, and hence the theorem is over. 
	\end{proof}

	\begin{appendix}
		\renewcommand{\thesection}{\Alph{section}}
		\numberwithin{equation}{section}		\section{Some Useful Results}\label{SUR}
		Let us recall some useful results from \cite{IG,IGCR}. 
		
		\vspace{2mm}
		\noindent
		\textbf{Green's function estimates:} The following estimates have been frequently used in this work (cf. \cite{IG,IGDN,LZN,AKMTM3})
		\begin{align}\label{A1}
			|G(t,x,y)| &\leq Ct^{-\frac{1}{2}}e^{-\frac{|x-y|^2}{a_1t}},\\
			\label{A2}
			\bigg|\frac{\partial G}{\partial y}(t,x,y)\bigg| &\leq Ct^{-1}e^{-\frac{|x-y|^2}{a_2t}},\\
			\label{A3}
			\bigg|\frac{\partial G}{\partial t}(t,x,y)\bigg| &\leq Ct^{-\frac{3}{2}}e^{-\frac{|x-y|^2}{a_3t}},\\
			\label{A4}
			\bigg|\frac{\partial^2 G}{\partial y\partial t}(t,x,y)\bigg| &\leq Ct^{-2}e^{-\frac{|x-y|^2}{a_4t}},\\
			\label{A5}
			|G(t,x,z)-G(t,y,z)| &\leq C|x-y|^{\vartheta}t^{-\frac{\vartheta}{2}-\frac{1}{2}}\max\bigg\{e^{-\frac{|x-z|^2}{a_5t}},e^{-\frac{|y-z|^2}{a_5t}}\bigg\},\\
			\label{A6}
			\bigg|\frac{\partial G}{\partial z}(t,x,z)-\frac{\partial G}{\partial z}(t,y,z)\bigg| &\leq C|x-y|^{\vartheta}t^{-1-\frac{\vartheta}{2}}\max\bigg\{e^{-\frac{|x-y|^2}{a_6t}},e^{-\frac{|x-z|^2}{a_6t}}\bigg\},
		\end{align}for all $t\in(0,T], \ x,y,z\in[0,1]$, $C,\;a_i$ for $i=1,\ldots,6$, are some positive constants and $\vartheta\in[0,1]$.  The following estimate (see \cite{IG}) has been used frequently in the sequel
		\begin{align}\label{A7}
			\big\|e^{-\frac{|\cdot|^2}{a(t-s)}}\big\|_{\L^p} \leq C (t-s)^{\frac{1}{2p}},
		\end{align} for any positive constant $a$ and $p\geq 1$.
		
		Now, we define the operators $J_1(\cdot)$ and $J_2(\cdot)$ by 
		\begin{align*}
			J_1u(t,x)&:=\int_0^t\int_0^1G(t-s,x,y)\big((1+\gamma)u^{\delta+1}-\gamma u-u^{2\delta+1}\big)(s,y)\d y\d s\\&=:\big((1+\gamma)J_{11}-\gamma J_{12}-J_{13}\big)u(t,x),	\\
			J_2u(t,x)&:=\int_0^t\int_0^1 \frac{\partial G}{\partial y}(t-s,x,y)u^{\delta+1}(s,y)\d y\d s,
		\end{align*}for all $t\in[0,T], \ x,y\in[0,1]$, where $u\in\L^\infty(0,T;\L^p([0,1]))$, for some $p\geq1$. A similar result can be found in \cite{IG} (cf. \cite{IGDN,IGCR,AKMTM3}).  We are providing for $J_{13}(\cdot)$  (considering the highest order non-linearity) only, other term $J_{11}(\cdot)$ can be handled in similar manner. Note that for $J_{12}(\cdot)$ one can obtain the estimates from \cite{IG,IGCR}.

		\begin{lemma}\label{lemma3.1}
		Under $\eqref{A1}-\eqref{A7}$, we have 
		\begin{enumerate}
			\item $J_{13}(\cdot)$ is a bounded operator from $\L^{\sigma_3}(0,T;\L^q([0,1]))$ into $\C([0,T];\L^p([0,1]))$ for $p\in[1,\infty]$, $q\in [2\delta+1,p]$, $\kappa_3=1+\frac{1}{p}-\frac{2\delta+1}{q}$ and $\sigma_3>\frac{2(2\delta+1)}{\kappa_3+1}$.
		\end{enumerate}Moreover, for the values of $\kappa_3$ given above, the following estimate hold: \begin{enumerate}

			\item For every $0\leq t\leq T$, there is a constant $C$ such that 
			\begin{align*}
				\|J_{13}v(t,\cdot)\|_{\L^p}	&\leq C\int_0^t(t-s)^{-\frac{1}{2}+\frac{\kappa_3}{2}}\|v(s)\|_{\L^q}^{2\delta+1}\d s\\&\leq Ct^{\left\{\left(\frac{(\kappa_3-1)\sigma_3}{2(\sigma_3-2\delta-1)}+1\right)\frac{\sigma_3-2\delta-1}{\sigma_3}\right\}}\bigg(\int_{0}^{t}\|u(s)\|_{\L^q}^{\sigma_3}\d s\bigg)^\frac{2\delta+1}{\sigma_3},
			\end{align*} where $\sigma_3$ are given above.
			\item For $0\leq s\leq t \leq T$,  $\theta_3\in(0,\frac{\kappa_3+1}{2})$, $\sigma_6>\frac{2(2\delta+1)}{\kappa_3+1-2\theta_3}$, there is a constant $C$ such that 
			\begin{align*}
			\|J_{13}v(t,\cdot)-J_{13}v(s,\cdot)\|_{\L^p}& \leq 
				C(t-s)^{\vartheta_3}\bigg(\int_{0}^{t}\|v(s)\|_{\L^q}^{\sigma_6}\d s\bigg)^{\frac{2\delta+1}{\sigma_6}}.
			\end{align*}
			\item For every $0\leq t \leq T$,  $\rho_6\in(0,\kappa_3+1)$, $\sigma_9>\frac{2(2\delta+1)}{\kappa_3+1-\rho_6}$, there is a constant $C$ such that \begin{align*}
				\|J_{13}v(t,\cdot)-J_{13}v(t,\cdot+z)\|_{\L^p} &\leq C|z|^{\rho_6}
				\bigg(\int_{0}^{t}\|v(s)\|_{\L^q}^{\sigma_9}\d s\bigg)^{\frac{2\delta+1}{\sigma_9}},
			\end{align*}for all $z\in\R$.  We also set $J_{1i}u(t,y):=0, \text{ for } i=\{1,2,3\}$, whenever $y\in\R\backslash[0,1]$.
		\end{enumerate}
	\end{lemma}

		\begin{lemma}\label{lemma3.2}
			Under  $\eqref{A2}-\eqref{A7}$, for $p\in[1,\infty],\ q\in[\delta+1,p], \ \kappa_1=1+\frac{1}{p}-\frac{\delta+1}{q}$  and $\sigma_{10}>\frac{2(\delta+1)}{\kappa_1}$, the operator $J_2(\cdot)$ is bounded from $\L^{\sigma_{10}}(0,T;\L^q([0,1]))$ into $\C([0,T];\L^p([0,1]))$. Moreover the following estimates hold:
			\begin{enumerate}
				\item For every $0\leq t\leq T$, there is a constant $C$ such that 
				\begin{align*}
					\|J_2v(t,\cdot)\|_{\L^{p}}&\leq C\int_0^t(t-s)^{-1+\frac{\kappa_1}{2}}\|v(s)\|_{\L^q}^{\delta+1}\d s\\&\leq 
					C t^{\left\{\left(\frac{(\kappa_1-2)\sigma_{10}}{2(\sigma_{10}-\delta-1)}+1\right)\frac{\sigma_{10}-\delta-1}{\sigma_{10}}\right\}}\bigg(\int_0^t\|v(s)\|_{\L^q}^{\sigma_{10}}\d s\bigg)^{\frac{\delta+1}{\sigma_{10}}}
				\end{align*}
				\item For $0\leq s\leq t\leq T$, $\theta_4\in (0,\frac{\kappa_1}{2})$, $\sigma_{11}>\frac{2(\delta+1)}{\kappa_1-2\theta_5}$, there is a constant $C$ such that
				\begin{align*}
					\|J_2v(t,\cdot)-J_2v(s,\cdot)\|_{\L^{p}} & \leq 
					C(t-s)^{\theta_4}\bigg(\int_{0}^{t}\|v(s)\|_{\L^q}^{\sigma_{11}}\d s\bigg)^{\frac{\delta+1}{\sigma_{11}}},.
				\end{align*}
				\item For every $0\leq t\leq T,\; \rho_7\in(0,\kappa_1)$ and $\sigma_{12}>\frac{2(\delta+1)}{\kappa_1-\rho_7}$, there is a constant $C$ such that 
				\begin{align*}
					\|J_2v(t,\cdot)-J_2v(t,\cdot+z)\|_{\L^{p}}&\leq C|z|^{\rho_7}
					\bigg(\int_{0}^{t}\|v(s)\|_{\L^q}^{\sigma_{12}}\d s\bigg)^{\frac{\delta+1}{\sigma_{12}}},
				\end{align*}for all $z\in\R$. We also set $J_2v(t,y):=0$, whenever $y\in\R\backslash[0,1]$.
			\end{enumerate}
		\end{lemma}
		Let us recall a result from \cite{IGCR}, which helps us to obtain the uniform tightness for the operators $J_i$: 
		\begin{lemma}[Corollary 3.2, \cite{IGCR}]\label{lemma3.6}
			Let $\zeta_n(t,y)$ be a sequence of random fields on $[0,T]\times[0,1]$ such that $
			\sup\limits_{0\leq t\leq T}\|\zeta_n(t,\cdot)\|_{\L^q}\leq \xi_n$, for $q\in[1,p],$ where $\xi_n$ is a finite random variable for every $n$. Assume that the sequence $\{\xi_n\}$ is bounded in probability, that is, 
			\begin{align*}
				\lim_{C\to\infty}\sup_{n}\P(\xi_n\geq C)=0.
			\end{align*} Under $\eqref{A1}-\eqref{A4}$, the sequences $J_i(\zeta_n)$, for $i=\{1,2\}$, are uniformly tight in \\$\C([0,T];\L^p([0,1])),$ for $p\geq 1$. Moreover in the case $p=\infty$, the sequence $J_i(\zeta_n)$ is tight in the space $\C([0,T]\times[0,1])$.
		\end{lemma}
		
	\end{appendix}
	
	\medskip\noindent
	\textbf{Acknowledgments:} The first author would like to thank Ministry of Education, Government of India - MHRD for financial assistance. The second author would like to thank Prof. B. Rajeev (ISI Bangalore) for the  financial assistance through his CRG grant (SERB sanction File No.: CRG/2019/002595), and the  Department of Atomic Energy (DAE), Government of India for assisting through NBHM post-doctoral fellowship (File No.: 0204/6/2022/R\&D-II/5635).  M. T. Mohan would  like to thank the Department of Science and Technology (DST), India for Innovation in Science Pursuit for Inspired Research (INSPIRE) Faculty Award (IFA17-MA110).

%
%
%
%
%
%
%
%
%


\begin{thebibliography}{99}
		
		
		\bibitem{PB}	P. Billingsley, \emph{Convergence of probability measures,} Second edition, John Wiley \& Sons, Inc., New York, 1999.
		
		
		\bibitem{ABPFD} A. Budhiraja and P. Dupuis, \emph{Analysis and Approximation of Rare Events: Representations and Weak Convergence Methods}, Springer, 2019.
		
		\bibitem{ADPDVM}A. Budhiraja, P. Dupuis and V. Maroulas, Large deviations for infinite dimensional stochastic dynamical systems. \emph{Ann. Probab.}, {\bf 36} (2008), 1390--1420.
		
		\bibitem{BJM}J. M. Burgers,  A mathematical model illustrating the theory of turbulence, \emph{Advances in applied mechanics}, {\bf 1} (1948), 171--199.
		
		
		\bibitem{QCHG} Q. Cao and H. Gao, The large deviation of semilinear stochastic partial differential equation driven by Brownian sheet, \url{https://arxiv.org/pdf/2112.02785.pdf}.
		
		
		\bibitem{FCAM} 	F. Chenal and A. Millet,	Uniform large deviations for parabolic SPDEs and applications,	\emph{Stochastic Process. Appl.}, {\bf 72} (1997), 161--186. 
		
		
		
		\bibitem{DDR}
		G. Da Prato, A. Debussche and R. Temam, Stochastic Burgers' equation, \emph{NoDEA Nonlinear Differential Equations and Appl.}, \textbf{1} (1994), 389--402. 
		
		
		
		\bibitem{GDDG} 	G. 	Da Prato and D. Gatarek,  Stochastic Burgers equation with correlated noise, \emph{Stochastics Stochastics Rep.}, {\bf 52} (1995), 29--41.
		
		
		
		
		\bibitem{DaZ}
		G. Da Prato and J. Zabczyk, \emph{Stochastic Equations in Infinite Dimensions}, Second Edition, Cambridge University Press, 2014.
		
		\bibitem{DV85}	M. D. Donsker, and S. Varadhan,  Large deviations for stationary Gaussian processes, \emph{Comm. Math. Phys.}, {\bf 97} (1985), 187--210.
		
		\bibitem{RMD} R. M. Dudley, Distances of probability measures and random variables, \emph{Ann. Math. Statist.} {\bf 39} (1968), 1563--1572. 
		
		\bibitem{MIFADW}M. I. Freidlin and A. D. Wentzell, \emph{Random Perturbations of Dynamical Systems}, Springer, NewYork (1984).
		
		
		
		
		
		
		
		
		\bibitem{VJE} V. J. Ervin, J. E. Mac\'ias-D\'iaz and J. Ruiz-Ram\'ireza, A positive and bounded finite element approximation of the generalized Burgers-Huxley equation, \emph{J. Math. Anal. Appl.}, {\bf 424} (2015), 1143--1160.
		
		\bibitem{AGERHR}
		A. Garsia, E. Rodemich and H. Rumsey, A real variable lemma and the continuity of paths of some Gaussian processes, \emph{Indiana Univ. Math. Journal}, {\bf 20} (1970), 565--578.
		
		
		
		
		\bibitem{IG}
		I. Gy\"ongy, Existence and uniqueness results for semilinear stochastic partial differential equations, \emph{Stochastic Process. Appl.}, {\bf 73} (1998), 271--299.
		
		\bibitem{IGDN}
		I. Gy\"ongy and D. Nualart, On the stochastic Burgers equation in the real line, \emph{Ann. Probab.}, {\bf 27} (1999), 782--802.
		\bibitem{IGNK}
		I. Gy\"ongy and N. Krylov, Existence of strong solutions for It\^o's stochastic equations via approximations, \emph{Probab. Theory Related Fields}, {\bf 105} (1996), 143--158.
		
		\bibitem{IGCR}
		I. Gy\"ongy and C. Rovira, On $L^p$-solutions of semilinear stochastic partial differential equations, \emph{Stochastic Process. Appl.}, {\bf 90} (2000), 83--108.
		
		
		\bibitem{AKMT}  A. Khan, M. T. Mohan and R. Ruiz-Baier, Conforming, nonconforming and DG methods for the stationary generalized Burgers-Huxley equation, \emph{J. Sci. Comput.}, {\bf  88} (2021),  Paper No. 52, 26 pp.
		
		\bibitem{AKMTM3} A. Kumar and M. T. Mohan, 	Absolute continuity of the solution to stochastic generalized Burgers-Huxley equation, \emph{Submitted}, \url{https://arxiv.org/pdf/2202.01496.pdf}.	
		
		\bibitem{AKMTM2}
		A. Kumar and M. T. Mohan, Large deviation principle for occupation measures of stochastic generalized Burgers-Huxley equation, \emph{J. Theoret. Probab.}, (2022), DOI
		\url{https://doi.org/10.1007/s10959-022-01180-2}. 
		
			\bibitem{VMG} V. Kumar, M. T.  Mohan and A. K. Giri, On a generalized stochastic Burgers' equation perturbed by Volterra noise, \emph{Journal of Mathematical Analysis and Applications}, {\bf 506}  (2022), 125638.
		
		\bibitem{NVK}  	N. V. Krylov, \emph{Introduction to the Theory of Diffusion Processes}, Translations of Mathematical
		Monographs, vol. 142, American Mathematical Society, Providence, 1995. 
		
		
	
		
		
		\bibitem{LZN}
		N. Lanjri Zaidi and D. Nualart, Burgers equation driven by space-time white noise: absolute continuity of solution, \emph{Stochastics Stochastics Rep.}, {\bf 66} (1999), 273--292.
		
		\bibitem{SVLBLR}S. V. Lototsky and B. L. Rozovsky,  \emph{Stochastic partial differential equations},  Springer, Cham, 2017.
		
		
		
		
		
		
		
		\bibitem{MTMSBH}	M. T. Mohan,	Stochastic Burgers-Huxley equation: Global solvability, large deviations and ergodicity, \emph{Submitted}, (2021), \url{https://arxiv.org/pdf/2010.09023.pdf}. 
		
		\bibitem{MTMSGBH}	M. T. Mohan, Mild solution for stochastic generalized Burgers-Huxley equation, \emph{J. Theoret. Probab.},  {\bf 35} (2022), 1511--1536.
		
		\bibitem{MTMAK}	M. T. Mohan and A. Khan, On the generalized Burgers-Huxley equation: Existence, uniqueness, regularity, global attractors and numerical studies, \emph{Discrete Contin. Dyn. Syst. Ser. B},  {\bf 26} (2021), 3943--3988.
		
		
		
		
		
		
		
		
		
		\bibitem{MS}M. Salins, Equivalence and counter examples between several definitions of the uniform large deviations principle, \emph{Probab. Surv.}, {\bf 16} (2019), 99--142.
		
		\bibitem{MSABPD19}M. Salins, A.  Budhiraja and P. Dupuis, Uniform large deviation principles for Banach space valued stochastic evolution equations, \emph{Trans. Amer. Math. Soc.}, {\bf 372} (2019), 8363--8421. 
		
		\bibitem{MSLS} M. Salins and L. Setayeshgar,  Uniform large deviations for a class of Burgers-type stochastic partial differential equations in any space dimension, \emph{Potential Analysis}, {\bf 58} (2023), 181--201.
		
		
		\bibitem{JS}	J. Satsuma, Exact solutions of Burgers' equation with reaction terms, \emph{Topics in soliton theory and exact solvable nonlinear equations}, (1987), 255--262.
		
		\bibitem{LS23} L. Setayeshgar, Uniform large deviations for a class of semilinear stochastic partial differential equations driven by a Brownian sheet, \emph{Partial Differ. Equ. Appl.}, {\bf 4} (2023), 1--12.
		
		
		
		\bibitem{LS} L. Setayeshgar, Large deviations for a stochastic Burgers’ equation, \emph{Commun. Stoch. Anal.}, {\bf 8} (2014),
		141--154.
		
		
		\bibitem{V66}S. S. Varadhan,  Asymptotic probabilities and differential equations, \emph{Comm. Pure Appl. Math.}, {\bf 19} (1966), 261--286.
		
		\bibitem{V84}S. S. Varadhan,  \emph{Large deviations and applications}, CBMS-NSF Regional Conference Series in Applied Mathematics, 46. Society for Industrial and Applied Mathematics (SIAM), Philadelphia, PA, 1984.
		
		\bibitem{MCV} M. C. Veraar and I. Yaroslavtsev, Pointwise properties of martingales with values in Banach function spaces, In: \emph{High Dimensional Probability VIII}, Springer (2019), 321--340.
		
		\bibitem{JBW}
		J. B. Walsh, An introduction to stochastic partial differential equations,  \emph{{\'E}cole d'{\'E}t{\'e} de Probabilit{\'e}s de Saint Flour XIV-1984}, Springer, (1986), 265--439.
		
		\bibitem{XW}X. Y. Wang, Nerve propagation and wall in liquid crystals, \emph{Phys. Lett. A}, {\bf 112}, (1985), 402--406.
		
		
		
		
		\bibitem{IY}  I. Yaroslavtsev,	Burkholder-Davis-Gundy inequalities in UMD Banach	spaces, \emph{Commun. Math. Phys.}, {\bf 379} (2020), 417--459. 
		
		
		
	\end{thebibliography}
\end{document}